%% file: LangProof-BW.tex
\documentclass[a4paper,12pt]{article}

\usepackage{amsmath}
\usepackage[all,cmtip]{xy}
\usepackage{epigraph}
\usepackage[nottoc]{tocbibind}

\title{Proof of Serge Lang's Heights Conjecture and an almost optimal Bound for the
Torsion of Elliptic Curves}

\author{Benjamin Wagener}

\input{LangProof-BW-macromaths}

\input{LangProof-BW-macros}

\usepackage{hyperref}
\usepackage{amssymb,amscd}

\begin{document}

\maketitle
\epigraph{\textit{"You cannot solve a problem on the same level that it was created. You have to rise above it to the next level."} A. Einstein}


\begin{abstract}
This paper focuses on the proof of Serge Lang's Heights Conjecture in a form that is completely effective. As a complementary result the author provides a new proof of Mazur-Merel theorem about a bound for the torsion of elliptic curves in terms of the degree of the ground field and improves the known result by providing a small polynomial bound in terms of this degree.
\end{abstract}

\tableofcontents

\section{Introduction}
The theory of heights for elliptic curves has been the starting point of the current theory that is now very developed.
Heights of a point or more generally of a cycle, or still of an abelian variety, are real numbers that are seen as representing a kind
of arithmetic-geometric complexity of the corresponding object. 

More precisely, on any elliptic curves $E/K$ defined over a number field $K$, there is a special function, the canonical height,

\[\hat{h}:P\in E(\bar{K})\rightarrow\R_{\geq 0},\]

which is a semi-definite quadratic form over the algebraic points satisfying the parallelogram law,

\[\forall\:P,Q\in E(K),\:\h(P+Q)+\h(P-Q)=2\h(P)+2\h(Q),\]
so that,
\[\forall \:P\in E(K),\forall\: n\in\Z,\:\hat{h}([n]P)=n^2\hat{h}(P).\]

Moreover the corresponding isotropy subgroup is given by the group of torsion points,

\[\h(P)=0\Leftrightarrow\:P\:\textrm{is a torsion point of E}.\]

The theory of heights appeared first from the corresponding theory on algebraic numbers
and has been much more developed when it appeared to have a clear geometrical formulation, shown by Arakelov, that we will use bellow.

There are various equivalent ways of defining heights, either with the help of embedings in the projective space, either by intersection theory, either related to norms of sections of line bundles. Heights are associated to divisors or line bundles and decompose into the sum of local heights associated to local places. In our case, the canonical height $\h$ is associated to the divisor corresponding to the neutral element of the curve, $(O_E)$.

As elliptic curves are smooth algebraic groups of dimension 1 and genus 0, it doesn't matter to take the height with respect
to the divisor $(O_E)$ or with respect to any other divisor. Divisors of degree 1 are all equivalent.

More precisely, Andr\'e N\'eron has shown that, for each place $v\in M_K$ there exists
a function $\la_v:P\in E(K_v)\setminus\{O_E\}\rightarrow \R$ with suitable properties so that, after having suitably normalized the absolute values,

\[\forall\: P\in E(K),\:\h(P)=\frac{1}{[K:\Q]}\sum\limits_{v\in M_K}\la_v(P).\]

The conjecture discussed in this paper appeared first in one of the many books of Serge Lang, Elliptic Curves/Diophantine Analysis,
dated from 1978 in the following form:

\textit{"... We select a point $P_1$ of minimum height $\neq 0$ (the height is the N\'eron-Tate height). We then let $P_2$ be a
	point linearly independent of $P_1$ of height minimum $\geq 0$. We continue in this manner to obtain a maximum set of
	linearly independent elements $P_1,\ldots, P_r$ called the} \textbf{successive minima}.

\textit{We suppose that A is defined by a Weierstrass equation}
\[y^2=x^3+ax+b,\]
\textit{and we assume for simplicity that $K=\mathbb{Q}$, so that we take $a,b\in\mathbb{Z}$....}

\textit{It seems a reasonable guess that uniformly for all such models of elliptic curves over $\mathbb{Z}$, one has}
\[\hat h(P_1)\gg \log|\Delta|.\]
\textit{It is difficult to guess how the next successive minima may behave. Do the ratios of the successive heights remain
	bounded by a similar bound involving $|\Delta|$, or possibly of power of $\log|\Delta|$, or can they become much larger?..."}

Heights provide fundamental diophantine information. They provide "measurements" of intrinsic arithmetic-geometric
complexity that are crucial in any diophantine analysis. The Heights Conjecture says roughly that the height of
a rational point is greater than the Height of the corresponding curve. 

In this paper we prove the following theorem:

\begin{theorem} Let $d\geq 1$ be an integer, there exist some positive constants $B_d$, $C_d$, $C_d'$, effectively computable and depending only on $d$, such that for any
	elliptic curve $E/K$ defined over a number field $K$ of degree $d$, for $\hat{h}(\cdot)$ the associated N\'eron-Tate height, $h_F(E/K)$ the Faltings heights of $E/K$, and $N_{K/\Q}\Delta_{E/K}$ the norm on $\Q$ of the minimal discriminant ideal of $E/K$, then:
	\begin{itemize}
		\item For any $K$-rationnal torsion point, $P$, of order $\mathrm{Ord}(P)$:
		\[\mathrm{Ord}(P)\leq B_d.\]
		\item For any $K$-rational point of infinite order we have:
		\[\hat{h}(P)\geq C_d\log|N_{K/\Q}\Delta_{E/K}|,\]
		and simultaneously:
		\[\hat{h}(P)\geq C_d'h_F(E/K).\]
		\end{itemize}
		
		Moreover the preceding inequalities are satisfied with the following constants: \begin{gather*}
		B_d=2*10^{14}*d^{2.08}(log(2d))^{1.54}\\
	C_d=\frac{1}{7*10^{45}*d^{5.24}(\log(48d))^{3.08}}\\
C_d'=\frac{1}{2*10^{46}*d^{5.24}(\log(48d))^{3.08}}.
		\end{gather*}
		
		As a direct corollary of the upper bound on the orders of torsion, we obtain that for any elliptic curve $E/K$ over a number field of degree $d$, the cardinality of the group of torsion points satisfies:
		
		\[\mathrm{Card}(E_{\mathrm{tors}}(K))\leq 4*10^{28}d^{4.16}(\log(2d))^{3.08}.\]
\end{theorem}

\begin{remark}

\begin{itemize}\item The author doesn't know how to find the optimal form of the bound $B_d$. Our bound is already a striking improvement of the bound that was obtained by the method of Mazur and Merel which at best provides a bound of the form $C2^d$. Recently however an optimal bound was found for CM elliptic curves by Clarke and Pollack,  \cite{Breuer,Clarke1,Clarke2}:
	
	\[\limsup_{d\rightarrow\infty}\frac{T_{\mathrm{CM}}(d)}{d\log(\log(d))}=\frac{e^{\gamma}\pi}{\sqrt{3}},\]
	
	where :
	
	\[T_{\mathrm{CM}}(d)=\sup_{[K:\Q]=d,\; E/K\;\mathrm{of\; CM\; type}}\mathrm{Card}\;(E_{\mathrm{tors}}(K))\]
	
	We are therefore quite close to a global optimal bound. We know therefore that a polynomial bound for the torsion has an exponent between $1$ and $4.16$.
	\item It could also be possible that this exponent would be a universal coefficient, as for example in statistics and dynamics.
	\item The quite strange exponents with floating points comes from an approximation of $\frac{1}{e\log(2)}\approx 0.54$. \end{itemize}
	\end{remark}
\subsection{Notations}

In all this text, $E/K$ is an elliptic curve over an number field $K$ of degree $d$, $\Delta_{E/K}$ denotes its minimal discriminant ideal and $N_{K/\Q}\Delta_{E/K}$ its norm over $\Q$.

We denote by $M_K$ the set of places of $K$ that we subdivide in the set of archimedean places $M_K^\infty$ and the set of non-archimedean places $M_K^0$. We furthermore denote by $M_K^{0,sm}$ as the set, possibly empty, of non-archimedean places of $K$ where $E/K$ has split-multiplicative reduction.

The integer ring of $K$ will be denoted by $\Ocal_K$ and, for $v\in M_K^0$, the integer ring of $K_v$ will be denoted by $\Ocal_{K_v}$.

We denote by $N_{K/\Q}$ the norm function over $\Q$ of ideals of $\Ocal_K$.

For $v\in M_K$ we denote by $K_v$ the completion of $K$ at the place $v$.

For $\sigma\in M_K^\infty$, $n_\sigma:=[K_\sigma:\R ]$.

For $v\in M_K^0$ we denote by $\ord_v(\cdot)$ the natural valuation on $K_v$ defined by $v$ and by $|\cdot|_v$ the absolute value defined by
\[|t|_v=(N_{K/\Q}v)^{-\ord_v(t)}.\]

We denote the natural complex absolute value as an absolute value with no subscript: $|\cdot|$.

For $v\in M_K^{0,sm}$, we have an isomorphism $E(K_v)\cong K_v^*/q_v^\Z$, moreover $q_v$ is uniquely defined up to a unit by $\mathrm{ord}_v(q_v)=\mathrm{ord}_v(\Delta_{E/K})$. With respect to this isomorphism we denote by $t_{P,v}$, for a rational point $P$, a representative of $P$ in $K_v^*$ such that $|q_v|_v<|t_{P,v}|_v\leq 1$.

We denote by $\Nr/\spec\;\Ocal_K$ the N\'eron Model of the fixed elliptic curve $E/K$. We will denote also by $\Nr'=\Nr\times_{\Ocal_K}\Nr$ the N\'eron Model of $E\times_KE$. Moreover throughout this text, $p_1$ and $p_2$ will denote the projection on the first factor and, respectively, on the second factor of either $E\times_K E/K$ or, as well, of $\Nr'/\spec\;\Ocal_K$. 

In all this text, the neutral element of $E(K)$ will be denoted by $O_E$ and the neutral section of $\Nr/\spec\;\Ocal_K$ by $O_\Nr$.

Throughout this text the morphisms denoted by $\pi$ with possibly some upper or lower script denotes some structural morphisms.

We define the function $\log^{(1)}:\R^{+*}\rightarrow\R$ by:
\[x\to \max\{1,\log(x)\}.\]

We moreover define the function $\lceil\cdot\rceil:\R\rightarrow\Z$ by $\lceil x\rceil$ being the least integer greater or equal to $x$. We also denote by $\lfloor\cdot\rfloor$ the function that gives the integral part.
\subsection{Some details about this proof}

	Our proof is based upon previously attempted trials through local heights decomposition but applied here in geometric terms.

If geometry is quite well known to encapsulate algebraic data, it is used fundamentally here for being very sharp.

The proof uses local height decomposition in geometric settings and relies crucially on a proper use of the Slope Method
due to Jean-Beno\^it Bost.

The slope method is analogous to its algebraic counterpart, the Baker's Method, but should be sharper as being geometric.
Such tools are known as "transcendence constructions" because they originate from the earliest proof of the transcendence
of some numbers but in this context they rather provide sharp algebraic evaluations.

Using geometric tools in this arithmetic context means that we use through all this text various ingredients from
Arakelov Geometry.

\subsection{Some remarks about this paper}

This paper consists in 4 parts.

\begin{enumerate}
	\item The first par, chapter 1 and chapter 2, is a reminder of the context and of some Arakelov Geometry in
	the context of elliptic curves.
	
	We establish the lemma 1 which is a rather elementary comparison between the Faltings height and the discriminant.
	
	More importantly we establish the corollary 1 that is a lower bound for the Arakelov degree of the line bundle
	we consider through all this text.
	
	\item The second part, chapter 3 and chapter 4, is rather elementary but is important. We call chapter 3 "Reductions" because
	we make various estimates in which the Lang Conjecture is true and finally we reduce the truth of a conjecture to a last
	situation. Chapter 4 is a summary.
	
		The last estimate of chapter 4 tells us that in the case still to be proven some important part of the
		discriminant is rather big and indeed quite simply comparable to the whole discriminant.
		
	\item The third part, from chapter 5 to chapter 8, is the heart of our work. We elaborate a transcendence construction
	thanks to the Slope Method, in the last case of the reductions of chapter 4.
	
	In chapter 5 we set the basis of the structure we use later in order to apply the Slope Inequality.
	
	Chapter 6 simply state and prove the crucial "Zeros Lemma" in this context. This is an injectivity criterion that
	is necessary for the Slope Inequality to be true.
	
	Chapter 7 contains the evaluations of the various terms of the inequality. This is the core of the proof. Section 7.4
	is the hardest part of this and it contains various things that are completely new.
	
	Especially in section 7.4 we establish a non-archimedean Schwartz Lemma in the situation of the Slope Method. This "Schwartz Lemma"
	is a non-trivial non-archimedean estimate of the slope morphism. While an archimedean estimate was known since the earliest
	work about the Slope Method in the 90's no non-archimedean estimate was known before our work.
	
	Chapter 8 is the proof of the Height Conjecture in the case of a semi-stable elliptic curve. It is obtained simply by
	estimating the various terms, previously computed, of the Slope Inequality. In order to do this, we need to fix some
	free parameters, this is known as "the extrapolation" in transcendence proofs
	
	\item Chapter 9 establish the proof in full generality thanks to the previously established semi-stable case. This chapter
	is merely elementary.  
\end{enumerate}

\section{Preliminaries: Heights and the Lang's Conjecture}

In this section we begin by presenting some known formulas for local heights.

We continue in sections 2.2 and section 2.3 by showing where do those local heights come from,
either algebraically of geometrically. In section 2.3 we introduce the line bundle $\mathcal{L}^{(D)}$
that is central for us through all this text. We also establish the lemma 1 which is a special inequality
involving the Faltings height that will be used later.

In section 2.4 we essentially prove the inequality of corollary 1 section 2.4. For that purpose we have
to work deeper about the associated geometry and we introduction Moret-Bailly models in order to
prove Theorem 4 section 2.3.

\subsection{Reminder about local heights}

The usual definition of local heights is originally algebraic and it implies that
canonical local heights are well defined only up to a constant. As a consequence
the link between the canonical height $\hat{h}$ and the local heights $(\lambda_v(\cdot)))_v$ writes
in full generality:

\[\hat{h}(P)=C^{cte}+\frac{1}{[K:\Q]}\sum\limits_v \lambda_v(P),\]
for some constant $C^{cte}$.

  John Tate in a letter to Jean-Pierre Serre \cite{Tate} has provided suitable
  normalized formulas for the local height functions:
  
  \begin{theorem}\label{local-heights}(Tate, 1968)
  
  Let $E/K$ be an elliptic curve over a number field $K$ of minimal discriminant ideal $\disc$.
  
  The canonical local heights are given by the following formulas,
  \begin{itemize}
  \item
  For an archimedean place $\sigma\in M_K^{\infty}$ with by the uniformization theorem
  $E(\C_\sigma)\cong \C/\Lambda$ where $\Lambda_\sigma$ is an $\R-$lattice in $\C$, in which a point $P\in E(\C_\sigma)$ corresponds
  to $z_P\in\C $,
  $\la_\sigma:P\in E(\C_\sigma)\setminus\{O_E\}\rightarrow\R$ is given by,
  \[\la_\sigma(P)=-\log\left|\exp\left(-\frac{1}{2}z_P\eta(z_P)\right)\Sigma(z_P)\Delta(\Lambda_\sigma)^{\frac{1}{12}}\right|,\]
  where $\Sigma$ is the Weierstrass function, $\eta$ is its associated etha function and $\Delta$ is the discriminant function associated to the corresponding lattice.
  \item
  For a non-archimedean place $v\in M_K^0$ of $K$, for all point $P\in E_{0,v}(K)$ of smooth reduction modulo $v$,
  $\la_v:E(K_v)\setminus\{O_E\}\rightarrow \R$ is given by,
  \[\la_v(P)=\frac{1}{2}\textrm{max}\{0,-\log|x(P)|_v\}+\frac{1}{12}N_v\log(N_{K/Q}v),\]
  for $x(P)$ representing the corresponding coordinate of $P$ for any given Weierstrass equation of $E$,
  \begin{enumerate}
  \item As a first consequence, if $v$ is a place of good reduction,
  \[\la_v(P)= \frac{1}{2}\textrm{max}\{0,-\log|x(P)|_v\}\geq 0,\]
  \item As a second consequence, the theorem of Kodaira-N\'eron shows that if $v$ is a place of additive
  reduction or of non-split multiplicative reduction $[12]P\in E_{0,v}(K)$ for any rational point $P$ of $E$,
  so that
  \[\la_v([12]P)=\frac{1}{2}\textrm{max}\{0,-\log|x([12]P)|_v\}+\frac{1}{12}N_v\log(N_{K/Q}v),\]
  we note therefore that in this case:
  \[\la_v([12]P)\geq\frac{1}{12}N_v\log(N_{K/Q}v).\]
  \end{enumerate}
  \item For a place $v$ of split multiplicative reduction of $K$,
  \[\la_v(P)=\frac{1}{2}\textrm{max}\{0,-\log|x(P)|_v\}+\frac{1}{2}B_2\left(\frac{ord_v(P)}{N_v}\right)N_v\log(N_{K/\Q}v),\]
  where $x(P)$ is given as the corresponding coordinate of any Weierstrass equation, $B_2(T)=T^2-T+\frac{1}{6}$ is the second Bernouilly
  polynomial and $\mathrm{ord}_v(P)$ is defined by the correspondence with the Tate curve as follows.
  
  By a usual Theorem of Tate, we have an isomorphism $E(K_v)\cong K_v^*/q_v^{\Z}$ for some element $q_v\in K_v$ such that
  $|q_v|_v<1$ and $\textrm{ord}_v(q_v)=\textrm{ord}_v(\Delta_{E/K})$ $=N_v$ then if $P$ corresponds to the parameter $t_{P,v}$ chosen so that
  $0\leq\textrm{ord}_v(t_{P,v})<N_v$, wee define $ord_v(P)=ord_v(t_{P,v})$.
  
  \end{itemize}
  \end{theorem}

If the local heights are defined as in the preceding theorem of Tate then for the
canonical height:

\[\hat{h}(P)=\frac{1}{[K:Q]}\sum_v\lambda_v(P).\]

Note that for the case of split multiplicative reduction, we have also the following formula:
 
  \begin{proposition}
  
  We note that in general, local heights are defined up to a constant, for us
  and later in this article we will use the fact that in the case of split multiplicative reduction, the local heights
  can be written up to a real constant $C_v^{cte}$
  
  \[ \la_v(P)=-\log|\theta_0(t_{P,v})|_v+\frac{\ord_v(t_{P,v})(\ord_v(t_{P,v})-N_v)}{2N_v}\log\left(N_{K/\Q}v\right)+C_v^{Cte},\]
  
  where 
  
  \[\theta_0(t)=(1-t)\prod\limits_{n\geq 1}(1-q^nt)(1-q^nt^{-1}),\]
  is the usual theta function, $t_{P,v}$ is the parameter corresponding to $P$ in the identification with
  the Tate curve $E(K_v)\cong K_v^*/q_v^{\mathbb{Z}}$.
  
  We note also that with the previous normalization,
  \[\la_v(P)\geq -\frac{1}{24}N_v\log(N_{K/Q}v).\]
  \end{proposition}
  \begin{proof}
  	This is well known, see for example \cite{Joe2}, theorem 4.2. p.473
  \end{proof}
  We thus observe that there are two kinds of places that we should deal especially with for the proof of Lang's conjecture: the archimedean local heights and the heights at places of split multiplicative reduction that can both contribute negatively.
  
  That's why we proceed as follows, first we will use the following lemma \ref{HS} chapter 3 of Marc Hindry and Joseph Silverman  in order to have positive archimedean contributions, this being complemented by a lemma \ref{Elkies} chapter 3 of Elkies in order to diminish and average the corresponding loss, then once we will be in a situation where archimedean contributions are positive we will concentrate on the split multiplicative places, for which we observe that the contribution given by $B_2$ is negative precisely when the corresponding points of the Tate curve have high exponent "modulo $N_v$" which means that $v-$adically they are all close to each others. This remark that allows a proof by transcendence was communicated to me by Prof. Marc Hindry.

\subsection{Heights and the Lang Conjecture}

In this section we trace back the canonical height from its algebraic origin to a more geometric context.

Let $K$ be a number field, of which we denote the set of places by $M_K$, then for $P=(\alpha_0:\alpha_1:\ldots:\alpha_n)\in\mathbb{P}(K)$ the absolute logarithmic height associated to $P$ is defined as

\[h(P)=\frac{1}{[K:\Q]}\sum\limits_{v\in M_K}\log\max_i{|\alpha_i|_v^{n_v}},\]

and as defined is independent of any extension of $K$.

For an elliptic curve given by a Weierstrass equation, let say for example,

\[E:y^2=x^3+ax+b,\]

the canonical height is the normalized logarithmic height associated to the following naive height

\[\forall P\in E(K,)h(P):=\frac{1}{2}h(x(P)),\]

it is normalized in the sense that the naive height satisfies

\[\forall m\in \mathbb{N},\forall P\in E(K),h([m]P)=m^2h(P)+O(1),\]

so that by a process due to John Tate, the canonical height is then defined as,

\[\hat{h}(P)=\lim\limits_{n\rightarrow\infty}\frac{h([n]P)}{n^2}.\]

One can show, in the corresponding theory of heights associated to divisors or line bundles that this canonical height corresponds to the divisor $(O_E)$.

For the construction of this, see for example the section B of \cite{Hindry1}.

Indeed, if $h$ now denotes the above absolute logarithmic heights, to an hyperplane $H$ of $\mathbb{P}^n$ is associated an height

\[h_{\mathbb{P}^n,H}(P)=h(P)+O(1),\]

while the theorem B.3.2 of \cite{Hindry1} asserts that if $\phi:V\rightarrow W$ is a morphism of projective varieties and if $D$ is a divisor on $W$:

\[h_{V,\phi*D}(P)=h_{W,D}(P)+O(1),\]

therefore if $x:E\rightarrow \mathbb{P}^1$ is the coordinate function,

\[h_{E,x^*(\infty)}(P)=h(x(P))+O(1).\]

Now as $x^*(\infty)=2(O_E)$ we find the desired result.

This will allow us to use all the power of the geometry underlying the theory of heights.

There are some objects of special interest for us, the discriminant of an elliptic curve, $\Delta$, the $j$-invariant, $j$, and the invariant differential, $\omega$. As they are very classical objects on elliptic curves we refer to \cite{Joe2}.

For an elliptic curve $E/K$ over a number field we define classically, $\Delta_{E/K}$ as the minimal discriminant ideal of $E/K$. 

\subsection{The Geometric Landscape}

Here we continue to present some (Arakelov) geometry naturally associated to heights.

We have seen that the canonical height is associated to the divisor $(O_E)$. But the relation seems a little bit non-canonical because it follows from a regularization process. If one consider line bundles which is equivalent to divisors in our case, there is the well known Arakelov theory.

In this theory one consider an integral model of the abelian variety and a line bundle on this model. For example one can consider the N\'eron Model and a hermitian line bundle on this model.

Once we have an integral model of our abelian variety $A/K$ (i.e. a scheme over $\mathcal{O}_K$ of which the generic fiber is our variety) endowed with an hermitian line bundle $(\mathcal{A},\mathcal{L})$, one can define for any point $P\in A(K)$ that extends to a section $\Pcal:\spec\Ocal_K\rightarrow \mathcal{A}$ an associated height by

\[h_{\mathcal{A},\Line}(P)=\frac{1}{[K:\Q]}\widehat{deg}\left(\Pcal^*\Line\right),\]
where $\widehat{deg}$ is the well known Arakelov degree.

The Arakelov degree being defined for any hermitian line bundle over $\spec\;\Ocal_K$, we also define the slope of an hermitian line bundle $\bar{H}$ over $\spec\;\Ocal_K$ as the number denoted and defined by:

\[\hat{\mu}(\bar{H})=\frac{1}{[K:\Q]\mathrm{rk}(H)}\widehat{\mathrm{deg}}(\bigwedge^{\mathrm{max}}\bar{H}),\]
where $\mathrm{rk}$ denotes the rank and we define as well the normalized Arakelov degree
\[\widehat{\mathrm{deg}}_n(\bar{H})=\frac{1}{[K:\Q]}\widehat{\mathrm{deg}}(\bigwedge^{\mathrm{max}}\bar{H})\]

 One can construct directly a quadratic height if one consider integral hermitian line bundles that are cubist.

The cubist relationship is a well known relation for line bundles over abelian varieties that are defined over a field.

In the case of integral hermitian line bundles, the cubist relationship is stated as an isometric isomorphism.

\[\bigotimes\limits_{\emptyset\neq I\subset\{1,2,3\}}\left(\left(p_I^*\Line\right)^{\otimes\textrm{Card}(I)}\right)\cong\Ocal_{\mathcal{A}^3},\]

where the RHS is provided with the trivial metric.

An hermitian integral line bundle that satisfy this relationship is said to be cubist.

The (Arakelov) height associated to any such hermitian cubist line bundle is quadratic (especially because the cubist relation kills higher orders), that is:

\[h_{\mathcal{A},\mathcal{L}}(P+Q)=h_{\mathcal{A},\mathcal{L}}(P)+2<P,Q>_{\mathcal{L}}+h_{\mathcal{A},\mathcal{L}}(Q),\]
where $<P,Q>_{\mathcal{L}}$ is bilinear as soon as $\mathcal{L}$ is cubist.

There are a universal integral models for abelian varieties, the N\'eron Models $\mathcal{N}/\spec \;\Ocal_K$, which is associated to each abelian variety over a number field, $A/K$. It is a smooth group scheme such that any rational point $P\in A(K)$ extends to an integral section $\epsilon_P: \spec \;\Ocal_K \rightarrow \mathcal{N}$ of generic fiber $P$.

Following Moret-Bailly \cite{MB}, if one consider the divisor $(O_E)$ on $E/K$ and the corresponding line bundle $L=\Ocal_E(O_E)$, this line bundle doesn't extend directly to a cubist line bundle over the N\'eron model, we have to modify the coefficients along the components of the special fibers of the N\'eron model.

\begin{definition}\label{lcm}
Suppose that the elliptic curve $E/K$ is semi-stable, denote by $\Delta_{E/K}$ its minimal discriminant ideal, we define:
\[N_E=lcm\left\{N_v\left|\right.\Delta_{E/K}=\prod\limits_v \mathfrak{p}_v^{N_v}\right\},\]
and for such semi-stable elliptic curve define $\mathrm{Bad}(E/K)$ as the set of non-archimedean places dividing $\Delta_{E/K}$.
\end{definition}

 We suppose here that $E/K$ is semi-stable, we then know from the classical theory of elliptic curve that for each $v\in\mathrm{Bad}(A/K)$, the special fiber of the N\'eron model of $E/K$ at $v$ is a group that is isomorphic to $\mathbb{G}_m\times \Z/N_v\Z$ where $N_v$ is the exponent of $v$ in $\Delta_{E/K}$. Each component of this special fiber can then be labeled as $[C_{v,i}]$ for $i=0,\ldots,N_v-1$.

Moret-Bailly, in \cite{MB} 1.5. page 69, computed the cubist divisor, and so the cubist line bundle, associated to $(O_E)$, for a semi-stable elliptic curve it is given by:

\begin{equation}\label{eq:div-cub}\mathcal{F}=O_E+\sum\limits_{v\in\mathrm{Bad}(A/K)}\sum\limits_{i=1}^{N_v-1}\frac{i(i-N_v)}{2N_v}[C_{v,i}].\end{equation}

As such, it is a divisor with rational coefficients but then any multiple of $2N_E\mathcal{F}$ provides a true cubist divisor over $\mathcal{N}_{E/K}$ and therefore a true cubist line bundle.

\begin{definition}\label{cubline}
In what follows we will consider the line bundle:
\[\Line^{(D)}=\mathcal{O}_{\mathcal{N}_{E/K}}(D\mathcal{F}),\]
for some positive integer $D$ such that $4N_E|D$.
\end{definition}

\begin{remark}
	As a summary of what if $h(P)=\frac{1}{2}h(x(P))$ is the naive height and if $\hat{h}(P)=\lim\limits_{\infty}\frac{h([N]P)}{N^2}$ is the associated regularized (i.e. canonical) height, then:
	\[\hat{h}(P)=\frac{1}{D}h_{E,\mathcal{L}^{(D)}}(P),\]
	and this height is quadratic.
	\end{remark}
We have the following formula for the Faltings height of an elliptic curve,
see for example \cite{Corn} Ch.X:
\begin{equation}\label{eq:falt0}
h_F(E/K)=\frac{1}{12[K:\Q]}\left(\log|N_{K/\Q}\Delta_{E/K}|-\sum\limits_{\sigma\in M_K^\infty}n_\sigma\log|(2\pi)^{12}\Delta(\tau_\sigma)(\textrm{Im}
\:\tau_\sigma)^6|\right),
\end{equation}

where $\Delta(\tau_\sigma)$ is the complex discriminant function applied to the element $\tau_\sigma$ of the Poincar\'e upper half-plane representing the elliptic curve $E\times_\sigma\C $ isomorphic to $\C /(\Z+\tau_\sigma\Z)$ via the usual isomorphism.

We take this formula as a definition in this paper as this is almost all what we will use about the Faltings height but we can remind that the Faltings height is generally defined for an abelian variety as the (suitably normalized) Arakelov degree of the bundle of invariant differential forms of the N\'eron model of the abelian variety.
\begin{lemma}
	
With the previous notations, we have:
\begin{equation}\label{Falt-disc}
h_F(E/K)\leq\frac{1}{12[K:\Q]}\left(\log|N_{K/\Q}\Delta_{E/K}|+2\pi\sum\limits_{\sigma\in M_K^\infty}n_\sigma\textrm{Im}(\tau_\sigma)\right)-2.7572
\end{equation}
\end{lemma}

\begin{proof}

We have \begin{equation}\label{Falt-disc1}\Delta(\tau_\sigma)=(2\pi)q_\sigma\prod\limits_{n\geq 1}(1-q_\sigma^n)^{24},\;\textrm{with}\;q_\sigma=\exp(2\pi i\tau_\sigma),\end{equation}
where $\tau$ can be chosen in the usual fundamental domain of the Poincar\'e upper half-plane:
\[\tau_\sigma\in\mathcal{F}=\{\tau\in\C|\:|\tau|\geq 1\:\textrm{et}\:|\textrm{Re}(\tau)|\leq\frac{1}{2}\},\]
so that using the proof given in lemma 2.2 of \cite{Hindry2}:

\begin{equation}\label{Falt-disc2}\log\left|\prod\limits_{n\geq 1}(1-q_\sigma^n)^{24}\right|\geq-0,104927.\end{equation}

Applying this to the formula \ref{eq:falt0} then gives:

\begin{multline}h_F(E/K)=\\\frac{1}{12[K:\Q]}\left(\log|N_{K/\Q}\Delta_{E/K}|-\sum\limits_{\sigma\in M_k^\infty} n_
\sigma\log|\Delta(\tau_\sigma)|-6\sum\limits_{\sigma\in M_k^\infty}n_\sigma\log(\textrm{Im}\;\tau_\sigma)\right)\\
=\frac{1}{12[K:\Q]}(\log|N_{K/\Q}\Delta_{E/K}|-12[K:\Q]\log(2\pi)-\sum\limits_{\sigma\in M_k^\infty}n_
\sigma\log|q_\sigma|\cdots\\ \left.+0,104927[K:\Q]+0,863046[K:\Q]-[K:\Q]\log(2\pi)\right),\end{multline}
where we have used equations \ref{eq:falt0}, \ref{Falt-disc1}, \ref{Falt-disc2} and the following inequality \[\log(\textrm{Im}\;\tau_\sigma)\geq\log(\frac{\sqrt{3}}{2})\geq-0,143841.\]

Putting $|q_\sigma|=\exp(-2\pi\textrm{Im}\;\tau_\sigma)$ we obtain finally:

\[h_F(E/K)\leq\frac{1}{12[K:\Q]}\left(\log|N_{K/\Q}\Delta_{E/K}|+2\pi\sum\limits_{\sigma\in M_k^\infty}n_\sigma\textrm{Im}\;\tau_\sigma\right)-2,7572.\]
\end{proof}

\subsection{Heights and Cubist Models}

The goal of this section is to establish the inequality of corollary \ref{degreecub} section 2.4.. For this purpose we establish the theorem \ref{Moret-Bailly} section 2.4. and the necessary tools to prove it.

In order to prove theorem \ref{Moret-Bailly} below we have to introduce a few things, as for example Moret-Bailly models.

If $\Line$ is an hermitian line bundle over the N\'eron model $\Nr$ of an abelian variety $A/K$ there is a natural way to make of $H^0(\Nr,\Line)$ an hermitian line bundle over $\spec\;\Ocal_K$. For this purpose all what we need are norms on each $H^0(\Nr,\Line)\otimes_\sigma\C $ for $\sigma\in M_K^\infty$. But $\Line$ is endowed with metrics for each $\sigma\in M_K^\infty$ therefore if $d\mu_{\sigma}$ denotes the Haar measure on $A_\sigma(\C)$, we can define for $s\in H^0(\Nr,\Line)\otimes_\sigma\C $:
\[\|s\|_\sigma=\sqrt{\int\limits_{A_\sigma(\C)}\|s(x)\|_\sigma^2 d\mu_\sigma}.\]
In this text all the spaces of sections of hermitian line bundles will be endowed with such $L^2$-norms at the archimedean places.

As we have just seen, it is generally necessary to take some power of a line bundle over an abelian variety $A/K$ in order that the resulting line bundle extends properly as a cubist hermitian line bundle over the N\'eron model.

However, after the work of Moret-Bailly, Bost defined integral models of a symmetrically polarized abelian variety $(A/\bar{\Q},L)$ endowed with a finite set of rational points $F\subset A(\bar{\Q})$ so that on the resulting integral model, $(\mathcal{A}/\spec\;\Ocal_K)$ the line bundle $L$ extends properly to a hermitian cubist line bundle and moreover all elements of $F$ extend to integral sections.

Bost defines Moret-Bailly models in the following way:

\begin{definition}(Moret-Bailly models, \cite{Bost2} p.56)
	Let $A$ be an abelian variety over $\bar{\Q}$, $L$ be an symmetric ample line bundle over $A$ and $F$ a finite subset of $A(\bar{\Q})$. We define a Moret-Bailly model of $(A,L,F)$ over a number field $K$ as the data of:
	\begin{itemize}
		\item A semi-stable group scheme $\pi:\mathcal{A}\rightarrow\spec\;\Ocal_K$;
		\item An isomorphism $i:A\cong\mathcal{A}_{\bar{\Q}}$ of abelian varieties over $\bar{\Q}$;
		\item An hermitian cubist line bundle $\Line$ over $\mathcal{A}$;
		\item An isomorphism $\phi:L\cong i^*\Line_{\bar{Q}}$ of line bundles over $A$;
		\item For any $P\in F$, an integral section $\epsilon_P:\spec\;\Ocal_K\rightarrow\mathcal{A}$ of $\pi$, such that the generic point $\epsilon_{P,\bar{Q}}$ is equal to $i(P)$;
	\end{itemize}
	moreover the following technical condition is required: the Mumford group $K(L^{\otimes 2})$ extends to $K(\Line^{\otimes 2})$ which has to be a flat and finite group scheme over $\spec\;\Ocal_K$.

	\end{definition}
	
	The main interest for us is in the following concise form of theorem 4.10. p.56 of \cite{Bost2}:
	
	\begin{theorem}(Bost, part of theorem 4.10 page 56 of \cite{Bost2})
		
		For any number field $K_0$ and for any $(A,L,F)$ as in the previous definition there exists some Moret-Bailly model of $(A,L,F)$ defined over some number field $K$ containing $K_0$.
		
		Moreover for any such Moret-Bailly model:
		
		\[\hat{\mu}(\pi_{*}\Line)=\frac{1}{[K:\Q]rk(H^0(\mathcal{A},\Line))}\widehat{\mathrm{deg}}(H^0(\mathcal{A},\Line))=-\frac{1}{2}h_F(A)+\frac{1}{4}\log\left(\frac{\chi(A,L)}{(2\pi)^2}\right),\]
		$rk$ denoting the rank, $h_F$ the stable Faltings height, and $\chi(A,L)$ the Euler-Poincar\'e characteristic.
		
	\end{theorem}
\begin{definition}\label{totalsym}(Totally symmetric cubist line bundles)

Following \cite{MB} definition 2.3 p. 135, we define a totally symmetric cubist line bundle $\Line$ over a group scheme $\pi:G\rightarrow S$ defined $S$  as
a line bundle (cubist) that can be written locally for the fppf topology,

\[\Line=\mathcal{M}\otimes[-1]^*\mathcal{M},\]
where $\mathcal{M}$ is a cubist line bundle.

An equivalent definition is that there exits a dominant and quasi-finite base change $S'\rightarrow S$ such that the line bundle qu $\mathcal{L}_{S'}$ obtained by this base change can be written $\Line_{S'}=\mathcal{M}^{\otimes 2}$ where $\mathcal{M}$ is a cubist line bundle.

\end{definition}
 
\begin{lemma}\label{cubcub} Let $A/K$ be a semi-stable abelian variety over a number field $K$, $\Lambda$ an amble line bundle over $A/K$, then there exists a Moret-Bailly Model  $(\mathcal{A}_{\mor},\Line_{\mor})$ of $(A,\Lambda^{\otimes 2})$ over some finite extension of $K$ such that $\Line_\mor$ is a totally symmetric cubist hermitian line bundle. 
\end{lemma}

\begin{remark}
	It is useless for us to precise the finite set of rational points that extend to the Moret-Bailly model, however we can see that given any finite such set we can, according to \cite{Bost2}, obtain such a model with an extension of each of this rational points. 
\end{remark}

\begin{proof} 
For the proof we consider the original proof of the existence of Moret-Bailly Models given by Jean-Beno�t Bost in \cite{Bost2} (proof of existence \cite{Bost2} p.58 in the proof of the theorem 4.10 p.56) in the following way:

\begin{itemize}
\item First one can choose an extension $K_1$ such that the points of $K(\Lambda^{\otimes 4})$ be rational over $K_1$ (and therefore also the set of fixed rational points if needed). One then consider the N\'eron Model $\mathcal{A}_1/\spec\;\Ocal_{K_1}$ of $A\times_K K_1$.
\item Moreover, if one consider a field, $K_2$, that is sufficiently ramified over the primes of bad reduction of $\mathcal{A}_1$, applying the proposition 1.2.2 chapitre II of \cite{MB} and 2.1 chapitre II of \cite{MB} shows that the line bundle $\Lambda\times K_2$ extends to a unique (hermitian) cubist line bundle over $\mathcal{A}_1\times_{\Ocal_{K_1}}\spec\;\Ocal_{K_2}$.
\end{itemize}

This provides a Moret-Bailly Model $(\mathcal{A}_\mor,\Lambda_\mor)$ de $(A,\Lambda)$ such that $K(\Lambda_\mor^{\otimes 4})$ is flat and finite over its base. 

Then we observe directly that in this case $(\mathcal{A}_\mor,\Line_\mor:=\Lambda_\mor^{\otimes 2})$ is a Moret-Bailly model with a totally symmetric cubist line bundle of $(A,\Lambda^{\otimes 2})$.

\end{proof}

\begin{theorem}\label{Moret-Bailly}

 Let $A/K$ be a semi-stable abelian variety defined over a number field $K$, let $\Nr/\textrm{Spec}\:\Ocal_K$ be its N\'eron Model and $n$ an integer that kills the group of components of the special fibers of $\Nr$,
 let $L$ be an symmetric ample line bundle over $A$ that can be written in the form $L=\Lambda^{4n}$ where $\Lambda$ is a ample line bundle, then $L$ extends to a totally symmetric line bundle  $\Line$ over $\Nr$ in the preceding sense, and

\textrm{(Partial Key Formula)}
\begin{equation}\label{formule}
\hat{\mu}\left(\pi_{*}\Line\right)
\geq-\frac{1}{2}h_F(A)+\frac{1}{4}\log\left(\frac{\chi(A,L)}{(2\pi)^2}\right)\end{equation}
where $\chi(A,L)$ is the Euler-Poincar\'e characteristic of $L$, and
$h_F(A)$ is the Faltings Height of $A$.

\end{theorem}
\begin{proof}
Let us consider a Moret-Bailly Model $(\pi_\mor:\ab_\mor\rightarrow S_1,\Line_\mor)$ of $(A,L)$ in the sense of the preceding definition for a quasi-finite base change $S_1$ of $\Ocal_K$, notice then than the group scheme $K(\Line_\mor^{\otimes 2})$ is finite over $S_1$, see for example
 \cite{Bost2,MB}. This is a semi-stable group-scheme and we know by the main result of \cite{Bost2} that:
\begin{equation}\label{keykey}
\hat{\mu}\left(\pi_{\mor*}\Line_{\mor}\right)=-\frac{1}{2}h_F(A)+\frac{1}{4}\log\left(\frac{\chi(A,L)}{(2\pi)^2}\right).\end{equation}

We can moreover suppose that $\Line_\mor$ is totally symmetric thanks to the preceding lemma \ref{cubcub}.

Moreover the proposition 1.2.1 p. 45 of \cite{MB}, show that $\Lambda^{2n}$ extends to a cubist line bundle over $\Nr$ because $2n$ kills the obstruction for such an extension to be not possible, hence $L$ extends to a totally symmetric cubist line bundle over $\Nr$ that we denote by $\Line$.

Let us consider the base change $\pi_1:\Nr_1=\Nr\times_{\Ocal_K}S_1\rightarrow S_1$ that we endow with the line bundle
 $\Line_1=f_1^*\Line$ where $f_1:\Nr_1\rightarrow \Nr$ is the base change, then $\Line_1$
is a cubist line bundle because the cubist structure of $\Line$ transposes to
 $\Line_1$ by $f_1$.

On the other hand, let us consider a N\'eron Model  $\Nr_{S_1}$ of the base change of $A$ by the generic point of $S_1$
that we denote $K_1$.

The semi-stable reduction theorem \cite{SGA7} IX.3.2 implies that the isomorphisms on on the generic fibers,
$\ab_{\mor|K_1}\cong\Nr_{S_1|K_1}$ and $(\Nr\times S_1)_{|K_1}\cong\Nr_{S_1|K_1}$ induce open immersions
$\ab_\mor\hookrightarrow \Nr_{S_1}$ et $\Nr\times S_1\hookrightarrow \Nr_{S_1}$ that identify $\ab_\mor$ and $(\Nr\times S_1)$
with open sub-schemes of $\Nr_{S_1}$.

The line bundle thus obtained $\Line_1$ is then totally symmetric by this base change.

By the corollary 2.4.1 page 136 of \cite{MB}, as $K(\Line_\mor^{\otimes 2})$ is finite, $K(\Line_\mor)$ is also finite
and, as the line bundle is totally symmetric, we have:

\begin{equation}H^0(\ab_\mor,\Line_\mor)=H^0(\ab_\mor^{[1]},\Line_\mor),\end{equation}
where $\ab_\mor^{[1]}$ is the smallest open subscheme containing $K(\Line_\mor)$.

Then consider the semi-stable group scheme (following proposition 8.1 page 104 de \cite{MB}), obtained by the amalgamation of
$\Nr_1$ and of $K(\Line_\mor)$,

\[\Nr^\circledast=\Nr_1\circledast K(\Line_\mor),\]

such a scheme is obtained by the "technique de gonflement de Moret-Bailly" that we can set in this context as:

$\Nr^\circledast$ is the cokernel of the following:

\[
\xymatrix{0\ar[r] & \Nr_1\cap K(\Line_\mor)\ar[r] & \Nr_1\times K(\Line_\mor)\ar[r] & \Nr^\circledast\ar[r]& 0\\
& h\ar@{|->}[r]&  (h,-h)&&
}
\]
The proof that this scheme has the good properties is given p.104 of \cite{MB}.

On $\Nr^\circledast$ there is naturally a totally symmetric cubist line bundle given by the following:

We endow $\Nr_1\times K(\Line_\mor) $ with the totally symmetric cubist line bundle (because $\Line_1$ and $\Line_{\textrm{mb}}$ are) $p_1^*\Line_{1}\otimes p_2^*\Line_{\mor}$,
where $p_1$ and $p_2$ are the obvious projections. We then notice that this line bundle is invariant under the action of the group
 $\Nr_1\cap K(\Line_\mor)$ by translation and thus defines a totally symmetric cubist line bundle over the quotient. In fact, this last group lets $\Line_\mb$ invariant because $K(\Line_\mb)$ is defined as the group letting this line bundle invariant under translation but it also lets $\Line_1$ invariant the same way because $\Nr_1\cap K(\Line_\mor)=K(\Line_1)\cap K(\Line_\mor)$. 
 
We denote this line bundle by $\Line^\circledast$

We then deduce still by the corollary 2.4.1 page 136 of \cite{MB} that

\[H^0(\Nr^\circledast,\Line^\circledast)=H^0(\Nr^{\circledast [1]},\Line^\circledast)=H^0(\ab_\mor^{[1]},\Line_\mor)=H^0(\ab_\mor,\Line_\mor),\]
in fact $\Nr^\circledast$ is an open subscheme containing $K(\Line_\mor)$ therefore also $\Nr^{\circledast [1]}=\ab_\mor^{[1]}$, because $\Line^\circledast_{|\ab_{\mor}^{[1]}}=\Line_{\mb|\ab_{\mor}^{[1]}}$ and because of the construction.

Hence,

\begin{equation}
H^0(\ab_\mor,\Line_\mor)\subset H^0(\Nr_1,\Line_1)
\end{equation}

therefore as those two last modules, that are nothing else than $\pi_{1*}\Line_1$ and $\pi_{\mor*}\Line_\mor$, are
are two modules of the same rank, we have for the slopes:

\begin{equation}
\hat{\mu}(\pi_{\mor*}\Line_\mor)\leq\hat{\mu}(\pi_{1*}\Line_1).
\end{equation}

Moreover, as $\Line$ and $\Line_1$ are obtained from each other by a finite base change:

\begin{equation}\hat{\mu}(\pi_{*}\Line)=\hat{\mu}(\pi_{1*}\Line_1),\end{equation}
which finally gives
\begin{equation}
\hat{\mu}(\pi_{*}\Line)\geq\hat{\mu}(\pi_{\mor*}\Line_\mor),
\end{equation}

we then conclude the inequality of the theorem thanks to the equation \ref{keykey}.

\end{proof}

\begin{corollary}\label{degreecub}
	Let $E/K$ be a semi-stable elliptic curve, $\Nr$ its N\'eron model, for $N_E$ the integer defined in section 2.2 definition \ref{lcm} and for an integer $D$ such that $4N_E|D$ let $\Line^{(D)}$ be the line bundle introduced in the section 2.2 definition \ref{cubline}. Let $p_1$ and $p_2$ be the natural projections over the first factor and the second factor respectively of $\Nr'=\Nr\times_{\Ocal_K}\Nr$. Consider $\Line^{(D,D)}=p_1^*\Line^{(D)}\otimes p_2^*\Line^{(D)}$ the corresponding line bundle over $\pi:\Nr\times_{\Ocal_K}\Nr\rightarrow\spec\;\Ocal_K$, then the normalized Arakelov degree satisfies:
	\[\widehat{\mathrm{deg_n}}(\pi_{*}\Line^{(D,D)})\geq-D^2h_F(E)+\frac{D^2}{2}\log\left(\frac{D}{2\pi}\right),\]
	where $h_F(E)$, which is equal to $h_F(E/K)$ in this case, is the stable Faltings height of $E$.
	
	\end{corollary}
	
	\begin{proof}
		We know from the theorem \ref{Moret-Bailly}, that $\Line^{(D,D)}$ is a totally symmetric line bundle over $\Nr\times_{\Ocal_K}\Nr$ that satisfies:
		\begin{gather*}
		\widehat{\mathrm{deg_n}}(\pi_{*}\Line^{(D,D)})\\\geq\\
		-\chi(E\times E,p_1^*L^D\otimes p_2^*L)\frac{1}{2}h_F(E\times E)+\ldots\\\ldots+\frac{\chi(E\times E,p_1^*L^D\otimes p_2^*L)}{4}\log\frac{\chi(E\times E,p_1^*L^D\otimes p_2^*L)}{(2\pi)^2}
	\end{gather*}
	
	On the one hand we know from classical results that:
	\[h_F(E\times E)=2h_F(E),\]
	on the other hand we have for the self-intersection:
	\[(p_1^*L\otimes p_2^*L)\cdot (p_1^*L\otimes p_2^*L)=(\{O_E\}\times E+E\times\{O_E\})^2=2(\{O_E\}\times E)\cdot (E\times\{O_E\})=2.\]
	Because:
	\begin{itemize}
		\item the self-intersections are zero as shown by applying the moving lemma to the divisor $\{O_E\}$.
		\item the intersection is transversal as well by applying the moving lemma.
	\end{itemize}
This shows that:
\[\chi(E\times E,p_1^*L^D\otimes p_2^*L^D)=\frac{c_1(p_1^*L^D\otimes p_2^*L^D)^2}{2!}=D^2\frac{(p_1^*L^D\otimes p_2^*L^D)^2}{2}=D^2,\]
which concludes.
	\end{proof}
\section{Reductions}

The results of this section are summarized in the section 4 theorem \ref{reduction}.

In this chapter we are going to reduce the proof of Lang's conjecture to a special case with a very special but suitable hypothesis.

We are going to see in sections 3.1, 3.2 and 3.3 that the conjecture is already true in various circumstances.

Especially in section 3.1 we prove the lemma \ref{pigeonhole} which shows how the proof can be reduced to positive (and controlled) archimedean contributions. For this we use the lemma \ref{HS} of Hindry and Silverman complemented by the lemma \ref{Elkies} from Elikies that allows us to "average" the corresponding " loss ". 

Then in sections 3.2 and 3.3 we distinguish two cases depending on the the "relative size" of the discriminant. In section 3.2 we establish a first case of the conjecture for "relatively" big discriminant. Proposition \ref{grandisc} section 3.3 provides Lang's conjecture for "not too big discriminants. Proposition \ref{lastred} of section 3.3 finally establish the special situation still to be proven in the remaining of this document.

The next chapter, chapter 4, provide a whole theorem that sums up all the results of chapter 3.

\subsection{Preliminary}

In this subsection we mainly reduce the proof of the conjecture to a situation where the archimedean local heights are positive (Lemma \ref{pigeonhole}).

Recall that we defined in the notations the function $\log^{(1)}$ by:
\begin{definition}\label{log1}
	The function $\log^{(1)}:\R^{+*}\rightarrow\R$:
	\[x\to \max\{1,\log(x)\}.\]
	
	We moreover define the function $\lceil\cdot\rceil:\R\rightarrow\Z$ by $\lceil x\rceil$ being the least integer greater or equal to $x$.
	
\end{definition}

\begin{lemma}\label{HS}(Hindry-Silverman \cite{Hindry3} Lemma 1 and \cite{Hindry2} Proposition 2.3 p.427)
Let $E/\C$ be an elliptic curve over $\C$ represented isomorphically by $\C/\Z+\tau\Z$, with $\tau$ in the usual fundamental domain of the Poincar\'e upper half-plane, let $j_E$ be the $j$-invariant of $E/\C$, $\la$ the corresponding archimedean canonical local height
 and let $0< \epsilon<1 $ be a
real number . 

Then for any point
$P=\alpha+\tau\beta\in E(\C)$,
\[\textrm{max}\{|\alpha|,|\beta|\}\leq\frac{\epsilon}{23}\Rightarrow\la(P)\geq\frac{1-\epsilon}{12}\lo |j_E|\]
\end{lemma}

\begin{lemma}\label{Elkies}(Elkies' Lemma)

Let $E/\C$ be an elliptic curve with j-invariant $j_E$, $\la$ the associated archimedean local height and let $P_1,P_2,\ldots,P_N\in E(\C)$ be $N$ distinct points.

Then,

\begin{equation}\label{eq:elkies}
\sum\limits_{1\leq i\neq j\leq N}\la(P_i-P_j)\geq-\frac{N\log\:N}{2}-\frac{N}{12}\lo|j_E|-\frac{16}{5}N
\end{equation}

\end{lemma}

\begin{proof}

This theorem is proven for example in \cite{j}, Appendix A. Here we have simplified the theorem for our own purpose by ignoring the positive contribution of what is called the "Discrepancy" in this last paper and we have replaced the $\log^{+}$ by a $\log^{(1)}$.
\end{proof}
\begin{remark} The use of this Lemma is not absolutely necessary for the proof of the conjecture in some form. However it turns that without Elkies' Lemma the various evaluations needed in the main theorem of this paper would be exponential in the degree whereas thanks to this Lemma the estimates become polynomial in the degree.
	\end{remark}
\begin{lemma}\label{pigeonhole}(Reduction to positive archimedean contributions)
Let $E/K$ be an elliptic curve over a number field $K$ , 
Let $\sigma_0$ be an archimedean place of $K$ such that,
\[\lo|j_{\sigma_0}|=\max\limits_{\sigma\in M_K^{\infty}}\left\{\lo|j_\sigma|\right\},\]
where each $j_\sigma$ is the $j$-invariant of the elliptic curve $E\otimes_\sigma\C/\C$.

Let $0<\epsilon<1$ be a real number and $N,n$ be positive integers such that
\[N\geq n\ceil^2+1,\]
and
\[n\geq 200\frac{\dg}{1-\epsilon}\log\left(\frac{\dg}{1-\epsilon}\right).\]

 Then for any $K-$rational point of $E/K$,
\begin{itemize}
\item either $P$ is a torsion point of order
\[\textrm{Ord}(P)\leq N,\]
\item or there exist integers $1\leq m_1<m_2<\ldots<m_n\leq N$ and corresponding points $P_1=[m_1]P$, $P_2=[m_2]P$,\ldots, $P_n=[m_n]P$ with $P_i\neq P_j$ for $i\neq j$ such that

\[\forall 1\leq k\leq n,\:\frac{1}{n(n-1)}\sum\limits_{\sigma\in M_K^{\infty}}\sum\limits_{1\leq i\neq j\leq n}\la_{\sigma}(P_i-P_j)\geq \frac{1-\epsilon}{24}\log^{(1)}|j_{\sigma_0}|\]
\end{itemize}
\end{lemma}

\begin{proof}
Let $\sigma_0$ be an archimedean place of $K$ such that,
\[\lo|j_{\sigma_0}|=\max\limits_{\sigma\in M_K^{\infty}}\left\{\lo|j_\sigma|\right\}.\]

Let $E_{\sigma_0}(\C)$ be the complex points of $E$ with respect to the embedding $\sigma_0$, by the
usual isomorphism theorem we know that this can be written isomorphically as $\C/\Z+\tau_{\sigma_0}\Z$ for $\tau_{\sigma_0}$ in the usual fundamental domain of the Poincar\'e upper half-plane, in such a way that
 those points can be written as $\alpha+\tau_{\sigma_0}\beta$ with real non-negative numbers
$0\leq\alpha\leq 1$ and $0\leq\beta\leq 1$.

Considering the square $[0,1]\times[0,1]$ of the $(\alpha,\beta)$'s that we cover with $\ceil^2$
smaller squares of side $\frac{\epsilon}{23}$ we can apply the pigeon hole principle which says
that among the $N\geq n\left\lceil\frac{23}{\epsilon}\right\rceil^2+1$ consecutive multiples of $P$ there are
at least $n+1$ that are in the same small square. Having $n+1$ in the same small square means that either the point is a torsion point of order given, which gives the torsion case, or that that the points obtained by subtracting by the smallest multiple are different and are in the square $[0,\frac{\epsilon}{23}]\times[0,\frac{\epsilon}{23}]$.

Let us denote those points by $P_i=[m_i]P$ with $1\leq m_i\leq N$ for $1\leq i\leq n$ those points, now applying lemma \ref{HS} to the special place $\sigma_0$ and the lemma \ref{Elkies} to the other places we obtain:

\begin{multline}\frac{1}{n(n-1)}\sum\limits_{1\leq i\neq j\leq n}\sum\limits_{\sigma\in M_K^{\infty}}\la_{\sigma}(P_i-P_j) \
=\frac{1}{n(n-1)}\sum\limits_{1\leq i\neq j\leq n}\la_{\sigma_0}(P_i-P_j)+... \\
+\frac{1}{n(n-1)}\sum\limits_{\sigma_0\neq\sigma\in M_K^\infty}\sum\limits_{1\leq i\neq j\leq n}\la_\sigma (P_i-P_j)\\
\geq \frac{1-\epsilon}{12}\lo|j_{\sigma_0}|-\frac{\dg\log(n)}{2(n-1)}-\frac{1}{12(n-1)}\sum\limits_{\sigma\in M_K^\infty}\lo|j_\sigma|-\frac{16\dg}{5(n-1)}\\
\geq\frac{1-\epsilon}{12}\lo|j_{\sigma_0}|\times\cdots\\\\\times\underbrace{\left(1-\frac{6\dg\log(n)}{(1-\epsilon)(n-1)}-\frac{\dg}{(1-\epsilon)(n-1)}-\frac{192\dg}{5(1-\epsilon)(n-1)}\right)}_{\geq\frac{1}{2}\:\textrm{with our condition on $n$}},\end{multline}
the inequality $"\geq \frac{1}{2}"$ is obtained simply by determining a constant $C$ for which the inequality holds as soon as $n\geq C\frac{[K:\Q]}{1-\epsilon}\log\left(\frac{[K:\Q]}{1-\epsilon}\right)$, which concludes.

\end{proof}
\begin{remark}
	Note that the lemma that precedes can be adapted directly to a multiple of $P$ in spite of $P$.
	\end{remark}

Below is an intermediary estimate that we will use according to the next remark.

\begin{lemma}\label{max} Let $N\geq n$ be positive integers and define
\[V_{N,n}=\left\{(m_1,m_2,\ldots,m_N)\in \N^n\left|\right. \forall\: 1\leq i\neq j\leq n,\; 1\leq m_i\neq m_j\leq N \right\},\]
then
\[\max\limits_{(m_1,m_2,\ldots m_n)\in V_{N,n}}\sum\limits_{1\leq i,j\leq n} (m_i-m_j)^2\leq\frac{(n+1)^2(N+1)^2}{2}.\]

\end{lemma}

\begin{proof}
Proof given in the appendix
\end{proof}

\begin{remark}\label{rem0}

We are going to use this lemma as follows:

We will be given some multiples $P_i=[m_i]P$ of a point $P$ with $1\leq i\leq n$ and $1\leq m_i\leq N$ and we will have a mean of the form

\begin{gather}
\frac{1}{n(n-1)}
\sum\limits_{1\leq i\neq j\leq n}\hat{h}(P_i-P_j)=\\
\frac{1}{n(n-1)}\left(\sum\limits_{1\leq i\neq j\leq n}(m_i-m_j)^2\right)\hat{h}(P)\end{gather}

for which the last lemma provides an upper bound.
\end{remark}

The following lemma from Baker and Petsche establishes a comparison that is going to be useful (Lemma \ref{comparison}).

\begin{lemma}\label{BP}(Baker-Petsche) 

Let $E/K$ be an elliptic curve over a number field $K$, then for any archimedean place $\sigma$ of $K$,
\[\log^{(1)}|j_\sigma|+6\geq2\pi\textrm{Im}(\tau_\sigma),\]
where $j_\sigma$ is the $j$-invariant of $E\times_\sigma\C $ and $\tau_\sigma$ is the element of the usual fundamental domain of the Poincar\'e upper-half plane representing $E_\sigma(\C )\cong\Z/(\Z+\tau_\sigma\Z)$.
\end{lemma}

\begin{proof}\cite{j} lemma 24 p.28
\end{proof}

\subsection{Case of relatively small discriminants}

Here we establish in proposition \ref{smalldisc} a full case of the conjecture for relatively small discriminants.

\begin{proposition}\label{smalldisc}
Let $E/K$ be an elliptic curve over a number field $K$, let $C_1>1$ and $0<\epsilon<1$ be a real numbers, 

Suppose that
\begin{equation}\label{eq:H0}\max\limits_{\sigma}\{\log^{(1)}|j_\sigma|\}\geq\frac{C_1}{1-\epsilon}\norm,
\end{equation}
where for each archimedean place $\sigma$, $j_\sigma$ denotes the $j$-invariant of $E\times_\sigma\C/\C $, then we have that for any rational point $P$ of $E/K$, 
\begin{itemize}
\item Either $P$ is a torsion point of order
\begin{equation}\textrm{Ord}(P)\leq 2412\left\lceil\frac{23}{\epsilon}\right\rceil^2\frac{\dg}{1-\epsilon}\log\left(\frac{\dg}{1-\epsilon}\right),\end{equation}
\item or
\begin{gather}\hat{h}(P)\geq\frac{(C_1-1)(1-\epsilon)^3}{1536*10^4\dg^3\left(\log\left(\frac{\dg}{1-\epsilon}\right)\right)^2\left\lceil\frac{23}{\epsilon}\right\rceil^4}\log\left|N_{K/\Q}\Delta_{E/K}\right|,\end{gather}
and simultaneously
\begin{equation}
\begin{gathered}\hat{h}(P)\geq\\ \frac{(1-\epsilon)^ 3(C_1-1)}{1152*10^4*\dg^3\left(\log\left(\frac{\dg}{1-\epsilon}\right)\right)^2(1-\epsilon+\dg C_1)\left\lceil\frac{23}{\epsilon}\right\rceil^4}h_F(E/K).\end{gathered}\end{equation}
\end{itemize}
\end{proposition}
\begin{proof}

First we replace the point $P$ by $\tilde{P}=[12]P,$ in such a way that $\tilde{P}$ has good reduction for each place of additive or non-split multiplicative reduction, according N\'eron-Kodaira theorem.

Next choose positive integers $N$ and $n$ such that 
\[N=201\ceil^2\left\lfloor\frac{\dg}{1-\epsilon}\log\left(\frac{\dg}{1-\epsilon}\right)\right\rfloor\geq\ceil^2n+1\]
with
\[n= \left\lceil200\frac{\dg}{1-\epsilon}\log\left(\frac{\dg}{1-\epsilon}\right)\right\rceil.\]

Applying lemma \ref{pigeonhole} (slightly adapted)  we see that either $\tilde{P}$ is a torsion point which order $\leq N$ satisfies the mentioned inequality and therefore $P$ has order $\leq 12N$  or there exists $1\leq m_1<m_2<\cdots<m_n\leq N$ distinct multiples defined by $\tilde{P}_i=[m_i]\tilde{P}$ with $\tilde{P}_i\neq\tilde{P}_j$ for $i\neq j$ such that 

\[\frac{1}{n(n-1)}\sum\limits_{\sigma\in M_K^\infty}\sum\limits_{1\leq i\neq j\leq n}\la_\sigma(\tilde{P}_i-\tilde{P}_j)\geq\frac{1-\epsilon}{24}\max\limits_{\sigma\in M_K^\infty}\left\{\lo|j_\sigma|\right\}.\]

Using moreover the fact that the non-archimedean local heights satisfy
\[\forall\: v\in M_K^0,\forall\: P\in\: E(K),\la_v([12]P)\geq-\frac{1}{24}N_v\log(N_{K/Q}v),\]
we deduce that,

\begin{gather}
\frac{1}{n(n-1)}\sum\limits_{1\leq i\neq j\leq n}  \hat{h}(\tilde{P}_i-\tilde{P}_j)\geq...  \\
  \frac{1}{24[K:\Q]}\left((1-\epsilon)\max\limits_{\sigma\in M_K^\infty}\left\{\lo|j_\sigma|\right\}
-\norm\right)
 \end{gather}
.

which gives under the hypothesis \ref{eq:H0},
\[\frac{1}{n(n-1)}\sum\limits_{1\leq i\neq j\leq n}\hat{h}(\tilde{P}_i-\tilde{P}_j)\geq\frac{C_1-1}{24[K:\Q]}\norm.\]

We can now apply lemma \ref{max} gives with $\tilde{P}=[12]P$,

\begin{equation}\label{eq:red11}\hat{h}(P)\geq\frac{n(n-1)}{72(n+1)^2(N+1)^2}\times\frac{C_1-1}{24[K:\Q]}\norm.\end{equation}

On the other hand the hypothesis \ref{eq:H0} gives also
\[-\norm\geq-\frac{(1-\epsilon)}{C_1}\max\limits_{\sigma}\{\log^{(1)}|j_\sigma|\},\]
so that,

\begin{equation}\label{eq:20}\frac{1}{n(n-1)}\sum\limits_{1\leq i\neq j\leq n}\hat{h}(\tilde{P}_i-\tilde{P}_j)\geq\frac{(1-\epsilon)(C_1-1)}{24\dg C_1}\max\limits_\sigma\{\log^{(1)}|j_\sigma|\}\end{equation}

Moreover the proposition \ref{eq:falt0} gives
\[\norm+2\pi\sum\limits_\sigma\textrm{Im}(\tau_\sigma)\geq 12[K:\Q]h_F(E/K)+12[K:\Q]C,\]
where $C=2.7572$ is the constant of the referred proposition  and with the hypothesis \ref{eq:H0} we obtain
\[\norm+
2\pi\sum\limits_\sigma\textrm{Im}(\tau_\sigma)\leq \left(\frac{1-\epsilon}{C_1}+[K:\Q]\right)\max\limits_\sigma\{\log^{(1)}|j_\sigma|\}+6[K:\Q],\]
so that
\begin{gather}\left(\frac{1-\epsilon}{C_1}+[K:\Q]\right)\max\limits_\sigma\{\log^{(1)}|j_\sigma|\}\geq...\\ 12[K:\Q]h_F(E/K)+12 C[k:\Q]-6[k:\Q]\geq 12[K:\Q]h_F(E/K),\end{gather}
we get therefore with the equation \ref{eq:20},
\[\frac{1}{n(n-1)}\sum\limits_{1\leq i\neq j\leq n}\hat{h}(\tilde{P}_i-\tilde{P}_j)\geq\frac{(1-\epsilon)(C_1-1)}{2(1-\epsilon+C_1[K:\Q])}h_F(E/K).\]

Applying lemma \ref{max} and the remark which follows it with $\tilde{P}=[12]P$ we get,

\begin{equation}\label{eq:red12}\hat{h}(P)\geq\frac{2n(n-1)}{144(n+1)^2(N+1)^2}\frac{(1-\epsilon)(C_1-1)}{2(1-\epsilon+C_1\dg)}h_F(E/K).\end{equation}

Now, one can choose, 
\[n=\left\lceil 200\frac{\dg}{1-\epsilon}\log\left(\frac{\dg}{1-\epsilon}\right)\right\rceil\] and \[N=n\ceil^2+1\]
, which gives after some elementary manipulations,
\[\frac{n(n-1)}{(n+1)^2(N+1)^2}\geq\frac{(1-\epsilon)^2}{16*10^ 4* [K:\Q]^2*\ceil^4*\left(\log\left(\frac{\dg}{1-\epsilon}\right)\right)^2},\]
which once included in equations \ref{eq:red11} and \ref{eq:red12} gives the result of
the proposition.

\end{proof}

\subsection{Case of relatively big discriminants}
 
 In this section subsection we finally reduce the proof to a special case. It is a situation where the discriminant is relatively big. Proposition \ref{grandisc} establishes the conjecture when the discriminant is big but not too big whereas Proposition \ref{lastred} consider our most important case (case 2.(b) of proposition \ref{lastred}). 
 
 After proposition \ref{lastred} proven we will work with this special case through the remaining of this document.

\begin{lemma}\label{comparison}
Let $E/K$ be an elliptic curve over a number field $K$, let $C_1>1$ and $0<\epsilon<1$ be a real number then if
\begin{equation}\label{eq:H1}\max\limits_{\sigma}\{\log^{(1)}|j_\sigma|\}\leq\frac{C_1}{1-\epsilon}\norm, 
\end{equation}
, where as usual for each archimedean place $\sigma$, $j_\sigma$ is the $j$-invariant of $E\times_\sigma\C/\C$,then
\begin{equation}\label{eq:falt-disc2}
h_F(E/K)\leq\frac{1}{12\dg}\left(1+\frac{C_1\dg}{1-\epsilon}\right)\norm
\end{equation}
\end{lemma}
\begin{proof} The inequality \ref{eq:falt-disc2} is a direct consequence of the hypothesis \ref{eq:H1} applied to the inequality \ref{Falt-disc} using lemma \ref{BP}.
\end{proof}

\begin{definition}
	We define the integer $N_E$ associated to $E/K$ by:
 \begin{equation}\label{con}
N_E=\mathrm{lcm}\left\{N_v\left|\right.\Delta_{E/K}=\prod\limits_{v|\Delta_{E/K}}\p_v^{N_v}\right\}
\end{equation}
\end{definition}

We then have an easy comparison:

\begin{lemma}\label{N}
\[N_E\leq \left(|N_{K/\Q}\Delta_{E/K}|\right)^{1/e\log(2)}\leq\left(|N_{K/\Q}\Delta_{E/K}|\right)^{0.54}.\]
\end{lemma}
\begin{proof}
Proof given in the appendix.
\end{proof}

\begin{proposition}\label{grandisc}(Second Reduction, case of relatively small  discriminants)
Let $E/K$ be an elliptic curve over a number field $K$, let $C_1>1$ and $0<\epsilon< 1$ be real numbers suppose that

\[\max\limits_{\sigma\in M_K^\infty}\left\{\lo|j_\sigma|\right\}\leq \frac{C_1}{1-\epsilon}\norm,\]
where, as usual, $j_\sigma$ denotes the $j$-invariant of the elliptic curve $E\times_\sigma\C /\C $, and suppose moreover that for some constant $C_0(d)$ depending only on the degree $d$ of the number field $K$ such that,

\begin{equation}\label{granddisc}|N_{K/\Q}\Delta_{E/K}|\leq C_0(\dg).\end{equation}

Then,
\begin{itemize}
\item either $P$ is a torsion point of order
\begin{equation}\textrm{Ord}(P)\leq 2412\frac{\dg\ceil^2}{1-\epsilon}\log\left(\frac{\dg}{1-\epsilon}\right)C_0(d)^{0.54},\end{equation}
\item or $P$ is of infinite order and
\begin{equation}
\begin{gathered}
\hat{h}(P)\\
\geq\\
 \frac{(1-\epsilon)^2}{4068*10^4 C_0(\dg)^{1.08}\dg^2*\left(\log\left(\frac{\dg}{1-\epsilon}\right)\right)^2\ceil^4}\norm
\end{gathered}
\end{equation}

with simultaneously,

\begin{equation}\begin{gathered}\hat{h}(P)\\\geq\\ \frac{(1-\epsilon)^3}{339*10^4*C_0(\dg)^{1.08}(1-\epsilon+\dg C_1)\dg*\left(\log\left(\frac{\dg}{1-\epsilon}\right)\right)^2\ceil^4}h_F(E/K),\end{gathered}\end{equation}
\end{itemize}
\end{proposition}
\begin{proof}

In the situation of the proposition, applying Kodaira-N\'eron theorem we obtain that for all $K-$rational point
$P$ of E, of corresponding integral section on the N\'eron model , $\Pcal:\spec\;\Ocal_K\rightarrow \Nr$, the section $[12N_E]\Pcal$ and so any multiple of $[12N_E]\Pcal$ meets only the neutral fibers of $\Nr$, which means that it has good reduction at all non-archimedean places. 

Let now define
Next choose positive integers $N$ and $n$ such that 
\[N= n\ceil^2+1,\]
and
\[n=\left\lceil 200\frac{\dg}{1-\epsilon}\log\left(\frac{\dg}{1-\epsilon}\right)\right\rceil.\]

 Applying the lemma \ref{pigeonhole} (slightly adapted), to multiples of $[12N_E]P$, we obtain directly that
 \begin{itemize}
 \item
  either $P$ is a torsion point of order
\[\textrm{Ord}(P)\leq  12NN_E\leq 201*12\frac{\dg\ceil^2}{1-\epsilon}\log\left(\frac{\dg}{1-\epsilon}\right)C_0(d)^{0.54},\]

\item or we obtain $n$ multiples of $P$, denoted $P_i=[12N_Em_i]P$ with $m_i\leq N$ for $1\leq i\leq n$,

\[\frac{1}{n(n-1)}\sum\limits_{1\leq i\neq j\leq n}\sum\limits_{\sigma\in M_K^{\infty}}\la_\sigma(P_i-P_j)\geq 0.\]

\end{itemize}
From those two last facts and the theorem \ref{local-heights} we deduce that

\[\frac{1}{n(n-1)}\sum\limits_{1\leq i\neq j\leq n}\hat{h}(P_i-P_j)\geq\frac{1}{12\dg}\norm.\]

From this we deduce with lemma \ref{max} that

\[\hat{h}([12N_E]P)\geq\frac{2n(n-1)}{12(n+1)^2(N+1)^2}\norm.\]

and so

\[\hat{h}(P)\geq\frac{2n(n-1)}{12*12^2(n+1)^2(N+1)^2N_E^2}\norm.\]

finally, from the hypothesis and the lemma \ref{N}, we have

\[N_E\leq C_0(\dg)^{0.54},\]

and from the definition of $n$ and of $N$ and elementary manipulations,

\[\frac{n(n-1)}{(n+1)^2(N+1^2)}\geq \frac{(1-\epsilon)^2}{16*10^4*\dg^2*\left(\log\left(\frac{\dg}{1-\epsilon}\right)\right)^2\ceil^4}.\]

Therefore

\begin{equation*}
\begin{gathered}
\hat{h}(P)\\
\geq\\
 \frac{(1-\epsilon)^2}{4068*10^4 C_0(\dg)^{1.08}\dg^2*\left(\log\left(\frac{\dg}{1-\epsilon}\right)\right)^2\ceil^4}\norm
\end{gathered}
\end{equation*}

and simultaneously, using the comparison of lemma \ref{comparison} that is valid in our case,

\begin{equation*}\begin{gathered}\hat{h}(P)\\\geq\\ \frac{(1-\epsilon)^3}{339*10^4*C_0(\dg)^{1.08}(1-\epsilon+\dg C_1)\dg*\left(\log\left(\frac{\dg}{1-\epsilon}\right)\right)^2\ceil^4}h_F(E/K),\end{gathered}\end{equation*}
which concludes the proof.

\end{proof}
\begin{definition}\label{def:S}Let us consider an elliptic curve $E/K$ over a number field $K$ and let $v\in M_K^{0,\mathrm{sm}}$ be a place of split-multiplicative reduction.
	
	First we know, as the reduction is split-multiplicative, that $E(K_v)\cong K_v^*/q_v^{\Z}$ then for any rational point $P\in E(K)$ we define $\mathrm{ord}_v(P):=\ord_v(t_{P,v})$ where $t_{P,v}$ is a representative of $P$ in the disc defined by $|q_v|_v\leq|t|_v< 1$.

With this, for any K-rational point $P$ of $E/K$ we define the following sets:

\[S(P)=\left\{v\:\textrm{place of}\: K\left|\right.v\in M_K^{0,\mathrm{sm}}\;\textrm{and} \:\frac{1}{6}\leq\frac{\mathrm{ord}_v(P)}{N_v}\leq\frac{5}{6}\right\},\]
and
\[\tilde{S}(P)=\left\{v\:\textrm{place of}\: K\left|\right.v\in M_K^{0,\mathrm{sm}}\;\textrm{and} \:\frac{1}{3}\leq\frac{\mathrm{ord}_v(P)}{N_v}\leq\frac{2}{3}\right\}.\]
\end{definition}

With those definitions we have the following lemma

\begin{lemma}\label{lem:S} Let $E/K$ be an elliptic curve over the number field $K$,
then for any $K-$rational point $P$,
\[S(P)=\tilde{S}(P)\coprod\tilde{S}([2]P)\;\textrm{(disjoint union)}.\]
\end{lemma}

\begin{proof}
We know by definition that $\tilde{S}(P)\subset S(P)$.

On the other hand if $v\in S(P)\setminus\tilde{S}(P)$,
then
\begin{itemize}
\item either
\[\frac{1}{6}\leq\frac{\mathrm{ord}_v(P)}{N_v}\leq\frac{1}{3},\]
\item or
\[\frac{2}{3}\leq\frac{\mathrm{ord}_v(P)}{N_v}\leq\frac{5}{6},\]
\end{itemize}
with the relation $\ord_v([2]P)=2\mathrm{ord}_v(P)\:\textrm{mod}\:N_v$ this gives that in such a case $v\in\tilde{S}([2]P)$.

 On the other hand the preceding relation also shows that, $\tilde{S}([2]P)\subset S(P)$ which concludes, this is clear for all the place $v$ such that $N_v\neq 2$ and for $N_v=2$ we notice that $v\in S(P)$ if and only if $v\in\tilde{S}(P)$.
\end{proof}

\begin{lemma}\label{combi}(Combinatorial Lemma)
Let $(r_v)_{v\in S}$ be a sequence of non-negative real numbers indexed by a set $S$. Let $S_i$, $1\leq i\leq n$ be subsets of $S$ and let $l$ be a positive real number such that
\[\forall\/1\leq i\leq n,\;\sum\limits_{v\in S_i}r_v\geq\frac{1}{l}\sum_{v\in S}r_v.\]

Under the assumption that there exists an integer $Z$ such that $n\geq l(Z+1)$, there exists distinct integers $i_0,i_1,\ldots,i_Z$ such that:
\[\forall\:1\leq j\leq Z,\:\sum\limits_{S_{i_0}\cap S_{i_j}}r_v\geq\frac{2}{l^2(l+1)}\sum\limits_{v\in S}r_v.\]
\end{lemma}
\begin{proof}
Proof given in the appendix
\end{proof}

\begin{proposition}\label{lastred}
Let $E/K$ be an elliptic curve over a number field $K$. Let $C_1>1$ and $0<\epsilon<1$ be real numbers and let $N,n$ and $Z$ be positive integers such that,
\[n \geq 200\frac{\dg}{1-\epsilon}\log\left(\frac{\dg}{1-\epsilon}\right),\]
\[N\geq n\ceil^2+1,\]
\[n\geq 2880 (Z+1).\]

Suppose that
\[\max\limits_{\sigma\in M_K^\infty}\left\{\lo|j_\sigma|\right\}\leq \frac{C_1}{1-\epsilon}\norm,\]
where the $j_\sigma$ are defined as usual.

Then
\begin{enumerate}
\item
Either $P$ is a torsion point of order,
\[\textrm{Ord}(P)\leq 12N,\]
\item Either $P$ is of infinite order and then
\begin{enumerate}
\item either
\begin{equation}\hat{h}(P)\geq\frac{n(n-1)}{31104\dg(n+1)^2(N+1)^2}\norm,\end{equation}
with simultaneously,
\[\hat{h}(P)\geq \frac{n(n-1)(1-\epsilon)}{5184(n+1)^2(N+1)^2(1+dC_1)}h_F(E/K)\]
\item Or there exists $Z+1$ distinct multiples $P_k=[m_k]P$ with
\[\forall \:0\leq k\leq Z,\:-24(N-1)\leq m_k\leq 24(N-1),\]
such that
 \[\forall\: 0\leq k\leq Z,\:\sum\limits_{v\in \tilde{S}(P_0)\cap\tilde{S}(P_k)}N_v\log\left(N_{K/Q}v\right)\geq\frac{1}{550}\norm.\]
 \end{enumerate}
 \end{enumerate}
\end{proposition}

\begin{proof}
 First we apply lemma \ref{pigeonhole} (slightly adapted) to the point $\tilde{P}=[12]P$, which gives for $N$ and $n$ as in the settings, that either $\tilde{P}$ is a torsion point of order less or equal to $N$ and so either
 \[\textrm{Ord}(P)\leq 12N,\]
 
 or $P$ is of infinite order and then there are distinct multiples $\tilde{P}_k=[12m_k]P$ with $1\leq m_1< m_2<\cdots<m_n\leq n$ such that
 
 \begin{equation}\label{archi0}\frac{1}{n(n-1)}\sum\limits_{1\leq i\neq j\leq n}\sum\limits_{\sigma\in M_K^\infty}
 \la_\sigma(\tilde{P}_i-\tilde{P}_j)\geq 0.\end{equation}
 
 Now we define $P^{n-m}(E/K)$ the set of good, additive or non-split multiplicative reduction of $E/K$ and for a point $P$ we define $S^c(P)$ the complement of $S(P)$ defined previously in the set of places of split multiplicative reduction of $E/K$ and
 
 \begin{equation}\begin{gathered}\delta^{n-m}=\sum\limits_{v\in P^{n-m}(E/K)}N_v\log(N_{K/\Q}v)\\
 \delta^{S}(P)=\sum\limits_{v\in S(P)}N_v\log(N_{K/\Q}v)\\
 \delta^{S^c}(P)=\sum\limits_{v\in S^c(P)}N_v\log(N_{K/\Q}v).\end{gathered}\end{equation}
 
 We have from the theorem \ref{local-heights} giving the local heights, from the relation \ref{archi0} above and for the possible values of the polynomial $B_2$,
 
 \[\begin{array}{c}
 \frac{1}{n(n-1)}\sum\limits_{1\leq i\neq j\leq n}\hat{h}(\tilde{P}_i-\tilde{P}_j)\\
 \geq\\
 \frac{1}{72\dg n(n-1)}\sum\limits_{1\leq i\neq j\leq n}\left(6\delta^{n-m}+\delta^{S^c}(\tilde{P}_i-\tilde{P}_j)-3\delta^{S}(\tilde{P}_i-\tilde{P}_j)\right).\end{array}\]
 
 Now, as for any point $K-$rational point $P$,
 
 \[\norm=\delta^{n-m}+\delta^{S^c}(P)+\delta^S(P),\]
 
 we deduce that
 
 \begin{equation}\label{delta}\begin{array}{c}\frac{1}{n(n-1)}\sum\limits_{1\leq i\neq j\leq n}\hat{h}(\tilde{P}_i-\tilde{P}_j)
 \\\geq\\\frac{1}{72\dg n(n-1)}\sum\limits_{1\leq i\neq j\leq n}\left(
 \norm+5\delta^{n-m}-4\delta^{S}(\tilde{P}_i-\tilde{P_j})\right)
 \\\geq\\ \frac{1}{72\dg}\left(
 \norm-4\frac{1}{n(n-1)}\sum\limits_{1\leq i\neq j\leq n}\delta^{S}(\tilde{P}_i-\tilde{P_j})\right).\end{array}\end{equation}
 
 Let us now define\footnote{Remark: the value $\frac{1}{5}$ in the definition of $A$ is chosen so that $5$ is the smallest integer for which the construction works}
  \[A=\textrm{Card}\left\{(i,j)\left|\right.\delta^{S}(\tilde{P}_i-\tilde{P}_j)\geq\frac{1}{5}\norm\right\},\]
  
  and
 \[\phi=\textrm{Card}(A),\]
  then we deduce from the inequality \ref{delta} and the fact that $\delta^S(\tilde{P}_i-\tilde{P}_j)\leq\frac{1}{5}\norm$ for $(i,j)\notin A$ that
  
  \[\begin{array}{l}\frac{1}{n(n-1)}\sum\limits_{1\leq i\neq j\leq n}\hat{h}(\tilde{P}_i-\tilde{P}_j)\geq\\ \frac{1}{72\dg}\left(\norm-\frac{4(n(n-1)-\phi)}{5n(n-1)}-\frac{4}{n(n-1)}\sum\limits_{(i,j)\in A}\delta^S(\tilde{P}_i-\tilde{P}_j)\right).\end{array}\]
  
  Now we have the following disjunction\footnote{Remark: The value $\frac{1}{6}$ in the disjunction is chosen so that $6$ is the smallest integer for which the construction works}
  
  \begin{enumerate}
  \item
  \underline{Case I}
  
  Either
  \[\begin{array}{c}\left(\norm-\frac{4(n(n-1)-\phi)}{5n(n-1)}\norm-\frac{4}{n(n-1)}\sum\limits_{(i,j)\in A}\delta^S(\tilde{P}_i-\tilde{P}_j)\right)\\\geq\\\frac{1}{6}\norm\end{array},\]
  \item
  \underline{Case II}
  
  or
  
 \[\begin{array}{c}\left(\norm-\frac{4(n(n-1)-\phi)}{5n(n-1)}\norm-\frac{4}{n(n-1)}\sum\limits_{(i,j)\in A}\delta^S(\tilde{P}_i-\tilde{P}_j)\right)\\\leq\\\frac{1}{6}\norm\end{array}\]
 \end{enumerate}

 \underline{Treatment of Case I}
 
 In the case I, we have that
 
 \[\frac{1}{n(n-1)}\sum\limits_{1\leq i\neq j\leq n}\hat{h}(\tilde{P}_i-\tilde{P}_j)
 \geq\frac{1}{6*72*\dg}\norm.\]
 
 Now, as $\tilde{P}_i=[12m_i]P$ with $1\leq m_i\leq N$ and $1\leq i\leq n$ we can apply the lemma \ref{max} which gives
 
 \[\hat{h}(P)\geq\frac{n(n-1)}{144*6*72\dg(n+1)^2(N+1)^2}\norm.\]

With the lemma \ref{comparison}, which is valid in this case, this gives simultaneously that

\[\hat{h}(P)\geq \frac{n(n-1)(1-\epsilon)}{5184(1+dC_1)(n+1^2)(N+1)^2}h_F(E/K).\]

\underline{Treatment of case II:}

The inequality of case II can be rewritten as

\[\frac{1}{30}\norm\leq\frac{4\phi}{5n(n-1)}\norm+\frac{4}{n(n-1)}\sum\limits_{(i,j)\in A}\delta^S(\tilde{P}_i-\tilde{P}_j),\]

now, as for each $(i,j)\in A$, $\delta^S(\tilde{P}_i-\tilde{P}_j)\leq\norm$ we get:
\[\frac{1}{30}\norm\leq\frac{4\phi}{5n(n-1)}\norm+\frac{4\phi}{n(n-1)}\norm,\]
 
 hence,
 
 \[\phi\geq\frac{n(n-1)}{144}.\]
 
 As fixing $i$ we have in this case $(n-1)$ distinct points $\tilde{P}_i-\tilde{P}_j$ by varying
 $1\leq j\leq n$ and $j\neq i$ we deduce from a pigeon hole principle that the relation given on $\phi$, there exists a integer $\bar{n}$,
 \[\bar{n}\geq\frac{n-1}{144},\]
 
 of distinct multiples $\bar{P}_k=[12n_k]P$ with $1-N\leq n_k\leq N-1$ such that
 
 \[\forall 1\leq k\leq \bar{n},
 \sum\limits_{v\in S(\bar{P}_k)}N_v\log(N_{K/Q}v)\geq\frac{1}{5}.\]
 
 for each $1\leq k\leq\bar{n}$ this gives using lemma \ref{lem:S},
 
 either
 \[\sum\limits_{v\in \tilde{S}(\bar{P}_k)}N_v\log(N_{K/Q}v)\geq\frac{1}{10},\]
 or
 \[\sum\limits_{v\in \tilde{S}([2]\bar{P}_k)}N_v\log(N_{K/Q}v)\geq\frac{1}{10},\]
 
 Now by applying\footnote{Here we have to be careful that we could have for some $i\neq j$, $\bar{P}_i=[2]\bar{P}_j$.} the combinatorial lemma (\ref{combi}) to this last situation we obtain that in this case there exists for $\bar{n}\geq 10(Z+1)$, which is the case for $n> 144*10(Z+1)*2 =2880(Z+1)$, a sequence of $Z+1$ distinct points $P_0=[\bar{m}_0]P,P_1=[\bar{m}_1]P,\ldots,P_Z=[\bar{m}_Z]P$ with
 with
 \[\forall\:0\leq k\leq Z,\:-2*12(N-1)\leq \bar{m}_k\leq 2*12(N-1),\]
 such that
 \[\forall\: 0\leq k\leq Z,\:\sum\limits_{v\in \tilde{S}(P_0)\cap\tilde{S}(P_k)}N_v\log\left(N_{K/Q}v\right)\geq\frac{1}{550}\norm.\]
 
 which concludes the proof.
 \end{proof}

\begin{remark} This last proposition tells us that in the case 2.(b) and for the points and local heights that is interesting to us, the contribution to the height is relatively big with respect to the discriminant.
	\end{remark}
	
\section{Summary of the Reductions}
The following theorem is just a summary of the results of the preceding section.

\begin{theorem}\label{reduction}
	
	Let $E/K$ be an elliptic curve over a number field $K$ of degree $d$. We denote by $\Delta_{E/K}$ its minimal discriminant ideal and by $N_{K/\Q}\Delta_{E/K}$ its norm over $\Q$. We denote by $h_F(E/K)$ the Faltings height of $E/K$.
	
	We choose $\epsilon\in]0,1[$ and $C_1>1$. We remind that for $P$ a rational point of $E/K$ we denoted by:
	\[\tilde{S}(P)=\left\{v\in M_K^{0,\mathrm{sm}}\left|\right.\frac{1}{3}\leq\frac{\mathrm{ord}_v(P)}{N_v}\leq\frac{2}{3}\right\},\]
	where $\mathrm{ord}_v(P)=\mathrm{ord}_v(t_{P,v})$ for a representative $t_{P,v}$ of $P$ in the isomorphism $E(K_v)\cong K_v^*/q_v^\Z$ such that $|q_v|\leq|t_{P,v}|<1$.
	
	For $\sigma\in M_K^{\infty}$ we denote $j_\sigma$ the $j$-invariant of $E\times_\sigma\C /\C $.
	
	\begin{itemize}
		\item Either
		\[\max\limits_{\sigma\in M_K^\infty}\left\{\log^{(1)}|j_\sigma|\right\}\geq\frac{C_1}{1-\epsilon}\log|N_{K/\Q}\Delta_{E/K}|,\]
		then
		\begin{itemize}
			\item if $P$ is a torsion rational point, its order satisfies:
			\[\mathrm{Ord}(P)\leq 2412\left\lceil\frac{23}{\epsilon}\right\rceil^2\frac{\dg}{1-\epsilon}\log\left(\frac{\dg}{1-\epsilon}\right).\]
			\item if $P$ is of infinite order,
			\[\hat{h}(P)\geq\frac{(C_1-1)(1-\epsilon)^3}{1536*10^4d^3\left(\log\left(\frac{d}{1-\epsilon}\right)\right)^2\left\lceil\frac{23}{\epsilon}\right\rceil^4}\log|N_{K/\Q}\Delta_{E/K}|,\]
			and simultaneously:
			\[\hat{h}(P)\geq\frac{(C_1-1)(1-\epsilon)^3}{1152*10^4 d^3\left(\log\left(\frac{d}{1-\epsilon}\right)\right)^2(1-\epsilon+dC_1)\left\lceil\frac{23}{\epsilon}\right\rceil^4}h_F(E/K).\]
		\end{itemize}
	\item Either
	\[\max\limits_{\sigma\in M_K^\infty}\left\{\log^{(1)}|j_\sigma|\right\}\leq\frac{C_1}{1-\epsilon}\log|N_{K/\Q}\Delta_{E/K}|,\]
	then:
	\[h_F(E/K)\leq\frac{1}{12d}\left(1+\frac{C_1d}{1-\epsilon}\right)\log|N_{K/\Q}\Delta_{E/K}|.\]
	Moreover:
	\begin{enumerate}
		\item Either there exists some constant $C_0(d)$ depending only on $d$, such that:
		\[|N_{K/\Q}\Delta_{E/K}|\leq C_0(d).\]
		Then:
		\begin{itemize}
			\item If $P$ is a torsion point,
			\[\mathrm{Ord}(P)\leq 2412\frac{d\left\lceil\frac{23}{\epsilon}\right\rceil^2}{1-\epsilon}\log\left(\frac{d}{1-\epsilon}\right)C_0(d)^{0.54}.\]
			
			\item If $P$ is of infinite order:
			\[\hat{h}(P)\geq\frac{(1-\epsilon)^2}{4068*10^4C_0(d)^{1.08}d^2\left(\log\left(\frac{d}{1-\epsilon}\right)\right)^2\left\lceil\frac{23}{\epsilon}\right\rceil^4}\log|N_{K/\Q}\Delta_{E/K}|,\]
			and simultaneously:
			\[\hat{h}(P)\geq\frac{(1-\epsilon)^3}{339*10^4 C_0(d)^{1.08}d\left(\log\left(\frac{d}{1-\epsilon}\right)\right)^2(1-\epsilon+dC_1)\left\lceil\frac{23}{\epsilon}\right\rceil^4}h_F(E/K).\]
		\end{itemize}
	\item or there exists a constant $C_0(d)$ depending only on $d$, such that
	\[|N_{K/\Q}\Delta_{E/K}|\geq C_0(d).\]
	Then we can choose any positive integers $Z,n$ and $N$ such that:
	\begin{gather*}
	n\geq 200\frac{d}{1-\epsilon}\left(\frac{d}{1-\epsilon} \right)\\
	n\geq 2880(Z+1)\\
	N\geq n\left\lceil\frac{23}{\epsilon}\right\rceil^2+1
	\end{gather*}
	then:
	\begin{itemize}
		\item If $P$ is of finite order and:
		\[\mathrm{Ord}(P)\leq 12N.\]
		\item If $P$ is of infinite order and then
		\begin{itemize}
			\item Either
			\[\hat{h}(P)\geq\frac{n(n-1)}{31104d(n+1)^2(N+1)^2}\log|N_{K/\Q}\Delta_{E/K}|,\]
			and simultaneously:
			\[\hat{h}(P)\geq\frac{n(n-1)(1-\epsilon)}{5184(n+1)^2(N+1)^2(1+dC_1)}h_F(E/K).\]
		\item or, finally, there exist $Z+1$ distinct integer multiples $P_k=[m_k]P$, $k=0,\ldots,Z$ with:
		\[-24(N-1)\leq m_k\leq 24(N-1),\] such that:
		\[\sum\limits_{v\in\tilde{S}(P_0)\cap\tilde{S}(P_k)}N_v\log(N_{K/\Q}v)\geq\frac{1}{550}\log|N_{K/\Q}\Delta_{E/K}|\]
		\end{itemize}
	\end{itemize}
	\end{enumerate}
	\end{itemize}
\end{theorem}
\section{Transcendence construction for the proof of the semi-stable case with the slope method}

This chapter has main goal to state the Slope Inequality that we will use. The Slope Inequality is in general a quite adaptable inequality that depends on some geometric objects (sections, infinitesimal neighborhood, module of sections, evaluation of sections).

for this purpose we remind some differential algebraic geometry in section 5.1.

In section 5.2 we state the objects and their context to be used in the Slope Inequality of section 5.2.2. It is this final inequality endowed with the presented structures that we will use for our proof in chapter 6, 7 and 8. Especially this is that inequality that allow us to tackle the last case of the theorem of chapter 4.

The Slope Inequality and the underlying structures, the Slope Method, were already presented in a quite general context by Jean-Beno\^it Bost in \cite{Bost1,Bost3} as a geometric version of Baker's Method. As such almost everything in this chapter was already present in the literature, the only difference is that we adapt it to our context.

The Slope Method by itself is a general framework that can be adapted in various contexts in order that the Slope Inequality holds.

In our case we are interested in an abelian variety of the form $A=E\times_KE$ where $E/K$ is a semi-stable elliptic curve over $K$. Saying that this elliptic curve is semi-stable means concretely that all non-archimedean places of bad reduction are places of multiplicative reduction (split or not). We indeed know from classical results that any elliptic curve becomes semi-stable after a controlled finite extension of the ground field and this will give us the result in full generality in the last section of this paper after we first proved a semi-stable version.

The reason, not obvious, why we have to consider a construction on $E\times_KE$ and not simply on $E/K$ is made more precise at the end of chapter 8, remark \ref{twook}. Said crudely our construction doesn't apply directly on only one elliptic curve and already other constructions where known to work on the square, as for example the use of the "plongement \'etir\'e" which has now become a classic among elliptic curve Diophantine constructions. Professor Jean-Beno\^it Bost made me a quite clear heuristic remark that one "sees" more of an elliptic curve over its square than working over a single elliptic curve. 
 
 \subsection{Reminder}
 
 In this section we mainly remind some differential algebraic geometry that are at the heart of the Slope Method construction
 
For $\sigma:K\hookrightarrow\C $ an embedding representing some archimedean place of $K$, the set of complex points of $A$ with respect to $\sigma$, $A_\sigma(\C )$, is a complex Lie group. Denoting by $t_A$ the tangent space above the origin of $A/K$ and $t_{A_\sigma}=T_A\otimes_\sigma\C $ its complex counterpart we are provided with a natural morphism of complex analytic groups:
\[\exp_{A_\sigma}:t_{A_\sigma}\rightarrow A_\sigma(\C ).\]

The first Chern class $c_1(L)$ of a line bundle $L$ over $A$, is represented, when tensorized by $\sigma$, in the De Rham cohomology by a unique element which is translation invariant. This elements thus identifies with a positive element of $\bigwedge^{1,1}\check{t}_{A_\sigma}$ which thus defines an hermitian form $\|\cdot\|_{L_\sigma}$ on $t_{A_\sigma}$ called the Riemann form of $L_\sigma$.

In our situation, $E/K$, is a semis-stable elliptic curve over a number field $K$ of degree $d$, over which we consider the line bundle:

\[L=\Ocal_E(O_E).\]

The line bundle we consider on $A=E\times_K E$ is a line bundle of the form $L^{D_1,D_2}=p_1^*L^{D_1}\otimes p_2^*L^{D_2}$ and we then consider over its N\'eron Model $\Nr'=\Nr\times_{\Ocal_K}\Nr$ the hermitian cubist line bundle $\Line^{(D_1,D_2)}=p_1^*\Line^{D_1}\otimes p_2^*\Line^{D_2}$ for suitable $D_1$ and $D_2$ where $\Line^{(D.)}$ is the hemitian cubist line bundle introduced in definition \ref{cubline} of section 2.3. We then know the degree inequality of corollary \ref{degreecub} section 2.4 and we know by the work of Moret-Bailly \cite{MB} that the Arakelovian height associated to $\Line$ is the N\'eron-Tate height.

According to the book of Mark Hindry and Joseph Silverman \cite{Hindry1}, the part about the Riemann forms and the theta functions, the Riemann form associated to $\Ocal_E(O_E)\otimes_\sigma\C $ is given by the following hermitian form:
\[H_\sigma(z_1,z_2)=\frac{z_1\bar{z_2}}{\mathrm{Im}\;\tau_\sigma},\]
	where $\tau_\sigma$ is the element of the usual fundamental domain of the Poincar\'e upper-half plane, $\left\{\tau\in\C \left| \right.|\tau|\geq 1,\;\mathrm{Re}(\tau)\geq\frac{1}{2}\right\}$, representing $E_\sigma$ in the isomorphism $E\times_\sigma\C\cong\C/(\Z+\tau_\sigma\Z )$.
	
	As the dimension of $A/K$ is 2, we can consider some non-trivial sub-spaces, $W$, of $t_A$ and thus we can define some infinitesimal neighborhoods at some orders along $W$.
	
	More precisely we have some equivalent definitions in the smooth cases for those infinitesimal neighborhoods.
	
	\begin{definition}(Infinitesimal Neighborhoods (General Definition))
		
		Les $X$ be a scheme of structural sheaf $\Ocal_X$ and $X^{\mathrm{top}}$ the associated topological space.
		
		Let $j:Y\rightarrow X$ a closed immersion and $\mathcal{I}_Y$ the ideal sheaf defining Y.
		
		We denote by $\mathcal{I}_Y^n$ the n-th power of this ideal sheaf.
		
		The infinitesimal neighborhood at the order $n$ of $Y$ is then defined as the scheme whose underlying topological space is $Y^{\mathrm{top}}$ and whose structural sheaf is $j^{-1}\Ocal_{Y,n}$ where $\Ocal_{Y,n}$ is the cokernel of the natural embedding $\mathcal{I}^n_Y\hookrightarrow\Ocal_X$.
		\end{definition}
	
	The next definition is a general definition of infinitesimal neighborhoods in the context of Abelian Varieties. It is contained in the book \cite{Mumford} from David Mumford.
		\begin{definition} (Infinitesimal neighborhoods I) Let $A/K$ be an abelian variety over a number field $K$ and let $P\in A(K)$, $n\geq 0$ an integer, $W$ a sub-vector space of $t_A$ of dimension $e$ and $(X_1,\ldots,X_e)$ a free family of invariant vector fields of $A$ and of which the value at $O_A$ is in $W$.
	
	The infinitesimal neighborhood at the order $n$ and in the direction $W$ of $P$, denoted $V(P,W,n)$ is defined as the closed sub-scheme of $A$, of support $P$ and whose ideal sheaf is defined as:
	
	\begin{gather*}\mathcal{I}_{V(P,W,n)}\\=\\\left\{f\in\Ocal_P\left| \right. \forall\;(j_1,\ldots,j_e)\in \N^e, j_1+\cdots+j_e\leq n\implies X_1^{j_1}\ldots X_e^{j_e}f(P)=0\right\}.\end{gather*}
	\end{definition}  
	
	The next definition is an equivalent definition that one can find in the work of Viada \cite{Viada,Viada2}. It is somehow the geometric construction that underlies in the previous definition.
	
	\begin{definition}(Infinitesimal Neighborhoods II)
Let $A/K$ be an abelian variety over a number field $K$ and $P\in A(K)$, we moreover consider a tangent sub-space $W$ of the tangent space at $P$, $t_{A,P}$. The infinitesimal neighborhood of order $0$ of $P$ is defined as being the point $P$ itself.

Now we consider the ideal sheaf of definition, $\mathcal{I}_{O_A}$, of the neutral element of $A$.

As $A$ is smooth we get an isomorphism $\Ocal_A/\mathcal{I}_{O_A}^2\cong K\otimes\mathcal{I}_{O_A}/\mathcal{I}_{O_A}^2=K+\check{t}_A$, therefore the embedding $W\hookrightarrow t_{A,P}$ induces an embedding $\Ocal_A/\mathcal{I}_P^2\hookrightarrow K\otimes\check{W}$ and thus a closed embedding of schemes:
\[V(O_A,W,1)\hookrightarrow A,\]
called the infinitesimal neighborhood at the order $1$ in the direction $W$ of $O_A$.

The infinitesimal neighborhood at the order $n$ in the direction $W$ of $O_A$ is then defined as the schematic image of $V(O_A,W,1)\times\cdots\times V(0_A,W,1)$ (n-times) by the addition $+_n:A\times\cdots\times A\rightarrow A$ and is denoted $V(O_A,W,n)$.

The infinitesimal neighborhood at the order $n$, in the direction $W$, of $P$, denoted $V(P,W,n)$ is then defined as the translation of $V(O_A,W,n)$ by the morphism of translation by $P$
\end{definition}

For our construction, we will work on a integral model $\mathcal{A}/\spec\;\Ocal_K$ of an abelian variety $A/K$, in our case it will be the N\'eron model $\Nr/\spec\;\Ocal$ of $E/K$ or $\Nr'=\Nr\times_{\Ocal_K}\Nr$ of $E\times_K E$.

\begin{definition}\label{intneir}(Integral infinitesimal neighborhood)
According to the preceding, it is then natural to defined the integral neighborhood of a point $P\in A(K)$, at order $n$ and in a direction $W$, as the schematic closure in $\mathcal{A}$ of $V(P,W,n)$. We denote it:
\[\mathcal{V}(P,W,n).\]
\end{definition}

\begin{remark}\label{neighbor}
	\begin{itemize}
		\item 
	In the case where the sub-space of the tangent space is the total tangent space of an abelian variety, we can see that the three last definitions are equivalent(smooth case).
	\item Moreover taking the schematic closure commutes with the formation of infinitesimal neighborhoods. In the case where $\mathcal{A}=\Nr$ and $A=E$, there is only one possible direction, we thus denote the infinitesimal neighborhoods of a point $P$ by $V(P,n)$ and its integral part as $\mathcal{V}(P,n)$. The previous remark thus shows that the ideal sheaf of definition of $\mathcal{V}(P,n)$ is
	\[\mathcal{I}_{\mathcal{V}(P,n)}=(\mathcal{I}_{\mathcal{P}})^{n+1},\]
		where $\mathcal{I}_{\mathcal{P}}$ is the ideal sheaf defining the section $\mathcal{P}$ extending $P$ in $\Nr$.
\end{itemize}
	\end{remark}

\begin{remark}\label{jets}(jets)
	
	For function $f\in\mathcal{C}^n(\mathbb{R},\mathbb{R})$, i.e. for functions from $\mathbb{R}$ to $\mathbb{R}$ that are $n$-times differentiable on $\mathbb{R}$ and for $a\in\mathbb{R}$, the ideal of functions that are zero at a, is given by:
	
	\[I_a=(x-a)\mathbb{C}^n(\mathbb{R},\mathbb{R}),\]
	
	Therefore \[I_a^n=(x-a)^{(n+1)}\mathcal{C}^n(\mathbb{R},\mathbb{R}),\]
	and $\mathcal{V}(a,n)=\mathcal{C}^n(\mathbb{R}\mathbb{R})/I_a^n$.
	
	Therefore, 
	
	\[f_{|\mathcal{V}(a,n)}=\sum\limits_{k=0}^{n}\frac{f^{(k)}(a)}{k!}(x-a)^k.\]
		
		Hence, $f_{|\mathcal{V}(a,n)}$ correspond to the $n$-th part of the development of Taylor of f, also known as the $n$-th jet of $f$ at $a$.
			
	We keep this analogy for some $s\in H^{0}(A,L)$ and for a neighborhood $V(P,W,T)$ ou $\mathcal{V}(P,W,T)$, the restriction to such neighborhood being called the ($T$-th)"jet" (at $P$ in the direction $W$).
	
	\end{remark}

We will need to introduce "integral tangent spaces" corresponding to tangent spaces of group scheme over $\spec\;\Ocal_K$ this is what does the following definition:

\begin{definition}\label{tangent}
	Let $\mathbb{G}/\spec\;\Ocal_K$ be a group scheme over $\spec\;\Ocal_K$, we denote by $\Omega_{\mathbb{G}/\spec\;\Ocal_K}$ the sheaf of differential 1-forms of $\mathbb{G}$ relatively to $\spec\;\Ocal_K$, we then define the following $\Ocal_K$-module:
		\[t_{\mathbb{G}}=(\epsilon^*\Omega_{\mathbb{G}/\spec\;\Ocal_K})^{\check{}},\]
		obtained by taking the dual of the pull-back of $\Omega_{\mathbb{G}/\spec\;\Ocal_K}$ by the neutral section og $\mathbb{G}$.
		
		$t_\mathbb{G}$ is an $\Ocal_K$-lattice in the tangent space at the origin $t_G$ of the generic fiber $G=\mathbb{G}\times K$.
		
		If one considers a sub-vector space $W$ of $t_G$ it is then natural to consider the integral sub-vector space:
		\[\mathcal{W}=t_{\mathbb{G}}\cap W.\]
	\end{definition}
\subsection{Our version of the slope method}

In this section, we define the structure underlying our adapted version of the Slope Inequality we use in chapter 6, 7 and 8 for the proof of Lang's conjecture in the last case of the theorem of chapter 4. Section 5.2 reminds the objects with which we are going to work (sections, modules of sections, neighborhood, evaluation of sections, filtration of modules of sections.).

Once we got our objects of study, section 5.2.1 introduce metrics and norms on those objects.

Finally, after all those definitions and constructions, section 5.2.2 state the corresponding (slope) inequality of interest for us.

For our construction we consider the following abelian variety:
 \[A=E\times_K E\] where $E/K$, is an elliptic curve with semi-stable reduction over a number
 field $K$, $\mathcal{N}/\textrm{Spec}\;\Ocal_K$ is the N\'eron Model of $E/K$ and we consider the N\'eron Model
\[\Nr'=\mathcal{N}\times_{\Ocal_K}\mathcal{N}\] of $A/K$ over $\spec\:\Ocal_K$ and its structural morphism
\[\pi:\Nr'=\mathcal{N}\times_{\Ocal_K}\mathcal{N}\rightarrow\spec\:\Ocal_K.\]

We now consider the line bundle,
 \begin{equation}L:=\Ocal_E(O_E),\end{equation} which is ample and symmetric over $E$.

For,
\[N_E=\mathrm{lcm}\left\{N_v\left|\right.\Delta_{E/K}=\prod\limits_v\\p_v^{N_v}\right\},\]
we consider an integer $D$ such that $4N_E$ divides $D$
 then the hermitian cubist extension of $L^D$ to $\mathcal {N}$, that we denote $\mathcal{L}^{(D)}$ and which is
which is a totally symmetric hermitian cubist line bundle over $\mathcal{N}$ after the lemma \ref{cubcub} section 2.4.

 We denote by $p_1$ and $p_2$ the projections of $E'=E\times_K E$ or of
 $\Nr'=\Nr\times_{\Ocal_K}\Nr$ over their first and second factors respectively.
 
 For $4N_E|D$, we then consider the ample symmetric line bundle \begin{equation}L^{D,D}=p_1^* L^{\otimes D}\otimes p_2^* L^{
\otimes D}\end{equation}
and its totally symmetric cubist hermitian extension \begin{equation}\mathcal{L}^{(D,D)}_\mb=p_1^*\mathcal{L}_
\mb^{ (D)}\otimes
p_2^*\mathcal{L}_\mb^{(D)}\end{equation} endowed with the metrics obtained by by the pullbacks and tensorial product of those of
$\mathcal{L}_\mb$.

In this section, the points $P_0',P_1',P_2',\ldots,P_Z'$ of $E\times_K E$ are generic points that we will specialize
later following the preceding reductions of the section 4,
and we will work with the integral infinitesimal neighborhoods in $\Nr '=\Nr\times\Nr$ following a direction $W$ of the
tangent space to $E'=E\times_K E$ that we denote at order $i$ and for the $k$-th point:
$\mathcal{V}(P_k',W,i)$.

We are now going to construct a filtration along the evaluations of global section of $\mathcal{L}^{(D,D)}_\mb$ on those neighborhoods.

For this purpose we consider the construction originally due to Jean-Benoit Bost in \cite{Bost1} but modify it a little bit: we define
a filtration in "two directions", which for us means that we filtrate both with respect to the order of derivation but also
point by point. This construction is almost the same but is more precise.

First we define the modules and the filtration then we will define metrics on those modules at each place.

$\mathcal{H}$ is the $\Ocal_K$-module of sections :
\[\mathcal{H}=H^0(\Nr\times\Nr,\Line^{(D,D)})=\pi_*\mathcal{L}_\mb^{(D,D)},\]

The $\Ocal_K$-module $F$ is obtained by restricting the line bundle over the integral infinitesimal neighborhoods $P_k'$ in
the following way:

\[\begin{array}{rl}\mathcal{F}=&\pi_{*}\left(\mathcal
{L}^{(D,D)}_{\mb|\mathcal{V}(\cur_0,W,T_0)}\right)\bigoplus\limits_{\lambda=1}^Z\pi_*\left(\mathcal
{L}^{(D,D)}_{\mb|\mathcal{V}(\cur_\lambda,W,T_1)}\right)\\&\\&
=H^0\left(\V(\cur_0,W,T_0),\mathcal
{L}^{(D,D)}_{\mb|\mathcal{V}(\cur_0,W,T_0)}\right)\bigoplus\limits_{\lambda=1}^Z H^0\left(\V(\cur_\lambda,W,T_1),\mathcal
{L}^{(D,D)}_{\mb|\mathcal{V}(\cur_\lambda,W,T_1)}\right)\end{array},\]
$T_0$ and $T_1$ being positive integers and the notation $\mathcal{L}^{(D,D)}_{|\V(\ldots)}$ signifying
$\mathcal{L}^{(D,D)}\otimes_{\Nr'}\Ocal_{\V(\ldots)}$, where $\Ocal_{\V(\ldots)}$ is the bundle of functions of the corresponding neighborhood
considered as a scheme, in particular:

\[\Ocal_{\V(\ldots)}=\Ocal_\Nr\Big{/}\mathcal{I}_{\V(\ldots)},\]
where $\mathcal{I}_{\V(\ldots)}$ is the corresponding ideal of definition.

The evaluation morphism is
\[\Phi:\mathcal{H}\rightarrow \mathcal{F},\]
which is obtained by restricting the sections on the various neighborhoods.

We now define the filtration with which we are going to work, first we define,

\begin{equation}
\m:=T_0+1+Z(T_1+1).
\end{equation}

\begin{definition}\label{filtration}(Filtration)
We define

 \[\mathcal{F}_0:=\{0\},\]
 
 then, for $k=0,\ldots, T_0$,
 
 \[\mathcal{F}_{k+1}=H^0\left(\mathcal{V}(\cur_0,W,k),\mathcal
{L}^{(D,D)}_{\mb|\mathcal{V}(\cur_0,W,k)}\right),\]

and then $k=T_0+1+(\lambda-1)T_1+\lambda+l$ with $\lambda=1,\ldots, Z$ and $l=0\ldots, T_1$

\begin{gather*} \mathcal{F}_{k+1}\\=\\H^0\left(\mathcal{V}(\cur_k,W,T_0),\mathcal
{L}^{(D,D)}_{\mb|\mathcal{V}(\cur_k,W,T_0)}\right)...\\...\bigoplus\limits_{j=1}^{\lambda-1}H^0\left(\V(\cur_j,W,T_1),\mathcal{L}^{(D,D)}_{
\mb|\V(\cur_j,W,T_1)}
\right)...\\...\bigoplus H^0\left(\V(\cur_\lambda,W,l),\mathcal{L}^{(D,D)}_{\mb|\V(\cur_\lambda,W,l)}
\right)\end{gather*}

continuing this way until

\[\mathcal{F}_{k_{\textrm{max}}+1}=\mathcal{F},\]

\end{definition}

\begin{remark}Remark that in the preceding filtration, $\lambda$ corresponds to the $\lambda$-th point of filtration
and $l$, to the $l$-th infinitesimal order.
\end{remark}

We can now introduce the following onto morphisms, obtained by restriction,
\begin{equation}
\begin{array}{c}\textrm{For}\:k\in\{0,\ldots,\m\},\\ p_k:\mathcal{F}_{k+1}\rightarrow \mathcal{F}_{k}\end{array}\end{equation}
as well as their iterated
\begin{equation}
\begin{array}{c}\textrm{For}\:k\in\{0,\ldots,\m\},\\ q_k:=p_k\circ\cdots\circ p_{\m}:\mathcal{F}\rightarrow \mathcal{F}_{k}
\end{array}\end{equation}
and $q_{\m+1}=\textrm{Id}_\mathcal{F}.$

The filtration is then defined by the following modules
\[\mathcal{F}^k:=\textrm{ker}\:q_{k}=\textrm{ker}\left(\mathcal{F}\rightarrow \mathcal{}F_{k}\right),\]
\[\mathcal{F}^{\m+1}=\{0\}\subset \mathcal{F}^{\m}\subset\cdots\subset \mathcal{F}^{0}=\mathcal{F}.\]

Of essential interest for us there are the sub-quotients:

\[\mathcal{G}^k=\mathrm{ker}(p_k)\cong\mathcal{F}_{k+1}/\mathcal{F}_k\cong\mathcal{F}^k/\mathcal{F}^{k+1},\]
and we define as well:
\begin{gather*}
\Phi_k=q_k\circ\Phi\\\mathcal{H}^k=\Phi^{-1}(\mathcal{F}^k)=\Phi^{-1}(\mathrm{ker}(q_k))=\Phi^{-1}(\mathrm{ker}(\mathcal{F}\rightarrow\mathcal{F}_k))
\end{gather*}
We can remark that if we define the schemes
\[\V_0:=\{0\},\]
for, $k=0,\ldots,T_0,$:
\[\V_{k+1}=\V(\cur_0,W,k),\]
and for $k=T_0+1+(\lambda-1)T_1+\lambda+l$ with $\lambda=1,\ldots,Z$ and $l=0,\ldots,T_1$
\[\V_{k+1}=\V(\cur_0,W,T_0)\bigcup\limits_{j=1}^{\lambda-1}\V(\cur_j,W,T_1)\bigcup\V(\cur_\lambda,W,l),\]

then

\[\mathcal{H}^k=\left\{s\in\mathcal{H}|s_{|\V_k}=0\right\}.\]

We have then for the domains of the applications $\Phi^k$:
\[\Phi^k:\mathcal{H}^k\rightarrow\mathcal{G}^k.\]
 and we have also
\begin{equation}\label{G2}
\mathcal{G}^k\cong \textrm{ker}\left(\mathcal{F}_{k+1}\longrightarrow \mathcal{F}_{k}\right)
\end{equation}
said differently, if $s\in \mathcal{H}^k$,
\begin{equation}\label{G10}\Phi^k(s)=\left\{\begin{array}{l}s_{|\V(\cur_0,W,k)},\:\textrm{for}\:k=0,\ldots,T_0;\\
s_{|\V(\cur_\lambda,W,l)},\;\textrm{for}\: k=T_0+1+(\lambda-1)T_1+\lambda+l\;\textrm{comme ci-dessus}.
\end{array}\right.
\end{equation}

Moreover, as $\mathcal{G}^k\cong\mathcal{F}_{k+1}/\mathcal{F}_k$ we have the following identifications:
\begin{enumerate}
\item For $k=1,\ldots, T_0$, 
\[\mathcal{G}^k\cong H^0\left(\V(P_0',W,k),\Line^{(D,D)}_{|\V(P_0',W,k)}\right)\Big{/}H^0\left(\V(P_0',W,k-1),\Line^{(D,D)}_{|\V(P_0',W,k-1)}\right),\]
\item For $k=T_0+1+(\lambda-1)T_1+\lambda+l$ with $\lambda=1,\ldots,Z$ et $l=0,\ldots,T_1$,
\begin{itemize}\item For $l=0$
\[\mathcal{G}^k\cong H^0\left(\V(P_\lambda',W,0),\Line^{(D,D)}_{|\V(P_\lambda',W,0)}\right),\]
\item or
\[\mathcal{G}^k\cong H^0\left(\V(P_\lambda',W,l),\Line^{(D,D)}_{|\V(P_\lambda',W,l)}\right)\Big{/}H^0\left(\V(P_\lambda',W,l-1),\Line^{(D,D)}_{|\V(P_\lambda',W,l-1)}\right),\]
\end{itemize}
\end{enumerate}
\[\mathcal{G}^{k}=\textrm{ker}(p_k)\cong \mathcal{F}_{k+1}/\mathcal{F}_{k}\cong \mathcal{F}^{k}/\mathcal{F}^{k+1},\]

\begin{remark}\label{slopei}
	
	The Slope Method is modeled on Baker's method that makes intensive use of differentiation and Schwartz-like lemma. We see this in the previous construction, for example the module $\mathcal{G}^k\cong\mathrm{ker}(\mathcal{F}_{k+1}\rightarrow\mathcal{F})$ consists of sections for which the restriction to the neighborhood is not only the jet described in the remark \ref{jets} but correspond directly to the $n$-th derivative of the section the other terms in the jet cancel.
	
	This is precisely the case were it makes sense to speak of derivative of sections even if a connection is not given.

	\end{remark}

We now have to endow this various modules with norms at the various places of $K$.
\subsubsection{Norms}
\begin{remark}All the metrics we consider in this text are usual metrics, more precisely this are the same metrics than the ones used by Jean-Beno�t Bost in the original paper \cite{Bost1}, which construction is made precise in \cite{Viada} and in \cite{Gaudron1}. It is the same for the norms, either archimedean or non-archimedean, with which we endow the various spaces we are working with. From this point of view, all what is about the metrics and the norms in this paper is in \cite{Bost1,Viada,Gaudron1,Viada2,Gaudron0}. In \cite{Viada}
	the explicit construction of the metrics and norms are in the chapters 1 and 2, in \cite{Gaudron1} it is in the annexes B and C.\end{remark}

\begin{itemize}
\item \textbf{The archimedean Norms (Algebraic Norms)}

Here, $\sigma\in M_K^\infty$ is an archimedean place. First we can notice that we already now two kinds of norms, the ones coming from the hermitian cubist line bundle $\Line^{(D)}$ as well as the associated hermitian forms on $t_{A_\sigma}$ coming from $c_1(L_\sigma^D)$. This is actually all what is basically needed in order to provide archimedean norms on our various modules.

First on the elements of the module,
\[\mathcal{H}=H^0(\Nr\times_{\Ocal_K}\Nr,\Line^{(D,D)}).\]

The metric on $L_\sigma$ allows by tensor product to define a metric on $\Line^{(D,D)}\times_\sigma\C =L_\sigma^{D,D}$. We denote this metric by $\|\cdot\|_{\sigma}$. Now given the Haar measure on $A_\sigma$, we can define an $L^2$-norm on $H_\sigma:=\mathcal{H}\otimes_{\sigma}\C$ by:

\[\|s\|_{\sigma}=\sqrt{\int_{A_\sigma(\C)}\|s(x)\|_{\sigma}^2d\mu_{\sigma}}.\]

This norm defines a norm on the elements of the sub-spaces $\mathcal{H}^k\otimes_\sigma\C $ of $H_\sigma$.
\newline\newline
Now we have to endow the modules $\mathcal{G}^k$ with an hermitian structure. They identify with $\textrm{ker}\:p_{k}$ therefore, as the points that we consider
are smooth over $\Nr'=\Nr\times\Nr$ we have the following embedding, which is explicit in the Thesis of Evelina Viada-Aehle
\cite{Viada} (chapters 4 et 5) and of \'Eric Gaudron \cite{Gaudron1}
as well as in their papers \cite{Viada2} et \cite{Gaudron0},

\begin{equation}\label{G1}\begin{array}{l}
\mathcal{G}^{k}\hookrightarrow \mathcal{P}_0^*\mathcal{L}^{(D,D)}_\mb\otimes \textrm{Sym}^k\check{\mathcal{W}};\textrm{for}\: k=0\ldots,T_0\\\\
\mathcal{G}^{k}\hookrightarrow \mathcal{P}_\lambda^*\mathcal{L}^{(D,D)}_\mb\otimes \textrm{Sym}^l\check{\mathcal{W}},\textrm{for}\;k=T_0+1+(\lambda-1)T_1+\lambda+l,
\end{array}
\end{equation}
of which the cokernel is killed by $l!$ (respectively $k!$), where $\mathcal{W}$
denotes, the $\Ocal_K-$module obtained as the integral extension of the subspace of $t_{E\times E}\cong t_E \oplus t_E$ with
respect to which we are doing the derivations:
$\mathcal{W}=t_{\Nr'}\cap W$ which we endow naturally with the structure of an
hermitian line bundle $\bar{\mathcal{W}}$ thanks to the Riemann forms of $\mathcal{L}^{(D,D)}_{\mb\sigma}=L^{D,D}_\sigma$.

We can then endow the right member of \ref{G1} with an hermitian structure
by taking the dual, tensorial and symmetric power, and direct sum thanks to those of
$\mathcal{L}^{(D,D)}_\mb$ and of $\mathcal{W}$.

In our situation, we will see that the embedding \ref{G1} is in fact an isomorphism (see the lemma \ref{WWWW})
and therefore we will have natural norms on $\mathcal{G}^{k}$.
\item\textbf{The non-archimedean norms}

Here $v\in M_K^0$ is a non-archimedean place. First we can notice that given a $\Ocal_{K_v}$-module of maximal rank, M, of a $K_v$-vector space S, there is a natural $v-$adic norm on $S$ associated to $M$ defined as follows, for $s\in S$:
\[\|s\|_M=\min\left\{|\lambda|_v\left| \right.\lambda\in K_v\setminus\{0\},\frac{s}{\lambda}\in M\right\}.\]

Equivalently, this norm is defined by:
\[\|s\|_M\leq 1\Leftrightarrow s\in M.\]

This is all what we need to define non-archimedean norms on our spaces. Such norms as usually called algebraic norms.

With this definition in mind, we define:
\begin{itemize}
	\item The norm on $H_v=\mathcal{H}\otimes K_v$ as the algebraic norm defined by the $\Ocal_{K_v}$-module $\mathcal{H}\otimes\Ocal_{K_v}$.
	\item The norm on $H^k_v=\mathcal{H}^k\otimes K_v$ as the algebraic norm defined by the $\Ocal_{K_v}$-module $\mathcal{H}^k\otimes\Ocal_{K_v}$.
	\item The norm on $G^k_v=\mathcal{G}^k\otimes K_v$ as the algebraic norm defined by the $\Ocal_{K_v}$-module $\mathcal{G}^k\otimes\Ocal_{K_v}$.
\end{itemize}
\item\textbf{Triple Norms}

Here we just remind the usual definition of the triple norm of a linear map between two normed vector spaces, either endowed with archimedean or non archimedean norms, $\phi:(V_1,\|\cdot\|_{V_1})\rightarrow (V_2,\|\cdot\|_{V_2})$:
\[|||\phi|||=\sup_{x\in V_1\setminus\{0\}}\frac{\|\phi(s)\|_{V_2}}{\|s\|_{V_1}}=\sup_{\|x\|_{V_1}=1}\|\phi(s)\|_{V_2}.\]

\end{itemize}

\subsubsection{The slope inequality}
Setting $V=\V_{\m+1}\times K$ and also naturally for an archimedean place $\sigma$,
$V_\sigma=V\times_\sigma\mathbb{C}$, the morphism of $\C-$vector space
\[\Phi_\sigma:\mathcal{H}_\sigma\rightarrow \mathcal{F}_\sigma,\]
identifies with the application
\[H^0(E_\sigma',L_\sigma^{D,D})\rightarrow H^0(V_\sigma,L_{|V_\sigma}^{D,D})\]
therefore the injectivity of $\Phi_\sigma$ is equivalent to the injectivity of
\[\begin{array}{lccc}\Phi_K :&\mathcal{H}_K&\rightarrow & \mathcal{F}_K\\
 & H^0(E',L^{D,D})&\rightarrow & H^0(S,L^{D,D}_{|V}).\end{array}\]

The slope inequality, after Jean-Beno\^it Bost \cite{Bost3} Ch.4, says that if $\Phi_K$ is injective:
\begin{equation}
\widehat{\textrm{deg}_n}(\mathcal{H})\leq\sum\limits_{k=0}^{\m}
\textrm{rg}(\mathcal{H}^{k}/\mathcal{H}^{k+1})\left[\hat{\mu}_{\max}(\mathcal{G}^k)+h(\mathcal{H}^k,
\mathcal{G}^k,\Phi^k)\right]
\end{equation}
where for each $k$:
\begin{equation}\label{pente3}
h(\mathcal{H}^k,\mathcal{G}^k,\Phi^k)=\frac{1}{[K:\Q]}\sum\limits_{v\in M_K^0}\log |||\Phi_k|||_v+\frac{1}{[K:\Q]}
\sum\limits_
{\sigma\in M_K^\infty}
n_\sigma\log |||\Phi_k |||_\sigma
\end{equation}

\begin{remark}\label{(heuristicslope)}
	
	The last inequality will be used as follows:
	
	\begin{itemize}
		\item The modules $\mathcal{G}^k$ contains some evaluations of the sections of the cubist line bundle over some neighborhood. Its slope, degree, therefore contains the height at each points, that will be multiples of some point $P$, plus some term coming from the differentiation. Therefore it is essentially a term providing some multiple of the canonical height.
		\item The term $h(\mathcal{H}^k,\mathcal{G}^k,\Phi^k)$ is an essential part for us. Apart from the archimedean part of this "height", the non archimedean parts consist of the evaluation of sections at a point very close to a zero point. Usual Schwartz-like estimations tells us therefore that those parts are very small (negative). This is going to provide the balancing term that will allow to show that the canonical height is bigger than some term comparable to the log of the discriminant or the Faltings height.
		
		We note that we succeed in providing some non-trivial estimates of the norms $|||\Phi|||_k$ at some non-archimedean places. This is the first job known to the author in which such non-trivial non-archimedean estimates are performed.
		
		\item Apart from the previous essential terms, various secondary terms appear that we will control in our estimates of chapter 8.
		
	\end{itemize}
\end{remark}

\section{The Zeros Lemma: Injectivity Criterion for the Evaluation Morphism}

The inequality \ref{pente3} that we stated previously at the end of chapter 5 is true \textbf{if} the evaluation morphism $\Phi_k$ is injective.

Therefore in this chapter we elaborate some suitable criterion under which $\Phi_k$ is injective, i.e. the slope inequality \ref{pente3} in the context of chapter 5 is true.

The goal of this short chapter is therefore to prove the lemma \ref{z�ros} and its corollary \ref{criter}.

The kind of injectivity criterion we set is essentially a "classical" one: If $\Phi_k(s)$ is zero at sufficient orders and at sufficiently many points, $s$ is zero.

With respect to this we elaborate a quite simple trick so that we can reduce the evaluation on the 2-dimensional product $E\times E/K$ to an evaluation on a single curve. This is interesting to do as on the 1-dimensional curve $E/K$ such injectivity criterion is easy to set and more accurate.

For this purpose we choose strategically the "plongement \'etir\'e" $P\mapsto (P,[M]P)$ and we differentiate accordingly in the direction $W=(X_{inv}^1+MX_{inv}^2)K$ where $X_{inv}^1$ and $X_{inv}^2$ the intuitively suitable normalized invariant vector fields.

For our construction, the scheme we consider is
\[\Nr'_\mb=\mathcal{N}\times_{\Ocal_K}\mathcal{N},\]
that we endow with the preceding totally symmetric cubist hermitian line bundle
\[\mathcal{L}^{(D,D)}_\mb=\left(p_1^*\mathcal{L}_{\mb}^{\otimes(D)}\otimes
p_2^*\mathcal{L}_{\mb}^{\otimes (D)}\right),\], with $4N_2|D$, which provides a totally symmetric cubict hermitian Model
$(\Nr_{\mb}',\mathcal{L}^{(D,D)}_\mb)$ of $(E\times_K E,L^{D,D})$. The points $\cur_k$ are bi-points of the type
 $\cur_k=(P_k,[M]P_k)$ where for some integers $m_k$ and for a rational point $P\in E(K)$,
$P_k=[m_k]P$ are multiples of a same point $P$ and $M$ is a positive integer that we will fix later according to the section 3-4
and to the final extrapolation in the section 8. This gives the points
$P_0',P_1',\ldots,P_Z'$ of $E\times_K E$ and the corresponding sections
$\Pcal_0',\Pcal_1',\ldots,\Pcal_Z':\spec\:\Ocal_K\rightarrow\Nr_\mb'$.

The hyperplane $W$ of the tangent space to $E\times E$ that we consider is:
\[W=(X_{\inv}^1+MX_\inv^2)K,\]
where $i=1,2$, $X_\inv^i$ is the vector field coming from $E$ on the $i$-th factor of $E\times_K E$
as defined below.

 The only condition that should be fulfilled in order to apply the slope inequality is that the evaluation morphism be
 injective. For this purpose there are many zeros lemmas in the literature, for us we will greatly benefit of the fact that
 we can do, in the context of this construction, as if we were only on a single elliptic curve and in the case of dimension
 1 the zeros lemmas are essentially optimal.

Let us be more precise:

Consider the following "plongement \'etir\'e":
\begin{equation}
\phi_M:\begin{array}{rcl} E &\longrightarrow & E\times E\\
P &\longmapsto &(P,[M]P),\end{array}
\end{equation}
as well as the following direction of derivation $E\times E$:

\begin{equation}\label{W}
W=(X_\inv^1+MX_\inv^2)K
\end{equation}
where $X_\inv^1$ and $X_\inv^2$ are defined below, definition \ref{XXX}, and $M$ is an integer.

\begin{definition}\label{XXX}
Considering a vector field $X_\inv$ over $E/K$ we can push it forward thanks to the following embeddings
\[\xymatrix{i_1: & E\ar[r] & E\times E\\
 & P\ar[r] &(P,O_E),}\]

\[\xymatrix{i_2: & E\ar[r] & E\times E\\
 & P\ar[r] &(O_E,P),}\]
 thus we obtain two fields
\[i_{1_*}(X_\inv),\:\textrm{et}\:i_{2_*}(X_\inv)\]
defined respectively over
\[i_1(E),\:\textrm{and}\:i_2(E).\]

We have then the natural identification,
\[t_{E\times E}\cong t_E\oplus t_E\cong t_E\times\{O_E\}\oplus\{O_E\}\times t_E,\]
thanks to which the vector fields $i_{1_*}(X_\inv),\:\textrm{et}\:i_{2_*}(X_\inv)$ are the identifies naturally
with vector fields $X_\inv^1$ and $X_\inv^2$ that are each now defined over $E\times E$.
\end{definition}

Notice the following lemma, which expresses that we are "differentiating in the direction of the embedding" and which explains our choices

\begin{lemma}\label{invariance}
In the sense of what precedes, if $\phi_M:E\rightarrow E\times_K E$ is the "plongement \'etir\'e" $\phi_M(P)=(P,[M]P)$ and if we denote
$\chi_\inv=X_\inv^1+MX_\inv^2$, defined in \ref{XXX}, the preceding vector field on $E\times_KE$, we then have for any function $f$ on $E\times E$:

\begin{equation}\label{M-invariance}\forall\;l\geq 0,\chi_\inv^{l}f_{|\phi_M(P)}=X_\inv^l(f\circ\phi_M)_{|P}.\end{equation}

\end{lemma}

\begin{proof}

From the definition we have given of $X_\inv^1$ and $X_\inv^2$, we have that over the subgroup $\phi_M(E)$
\[\left(\phi_{M_*}(X_\inv)\right)_{|\phi_M(E)}=\left(X_\inv^1+M X_\inv^2\right)_{*|\phi_M(E)}=\left(\chi_\inv\right)_{|\phi_M(E)},\]

then this formula at the power $l$ implies directly the lemma by the definition of the pushforward of a vector field.
\end{proof}

For $Z+1$ points of $E$, that we denote $P_0,\ldots, P_j,\ldots,P_Z$, we consider the infinitesimal neighborhoods
$V((P_j,[M]P_j),W,T_j)$, $1\leq j\leq Z$. The injectivity condition for the evaluation morphism is a condition
over the generic fiber. If we denote,
\[S=\bigcup\limits_{j=1}^Z V((P_j,[M]P_j),W,T_j),\]
a sub-scheme of dimension zero of $E\times_K E$
and
\[S_0=\bigcup\limits_{j=1}^ZV(P_j,T),\]
a sub-scheme of dimension $0$ of $E$, the injectivity of the evaluation morphism is the one of
\begin{equation}\Phi_K:\begin{array}{rcl}H^0(E\times E,p_1^*L^{D}\otimes p_2^* L^{D}) &\longrightarrow &
H^0(S,p_1^*L^{D}\otimes p_2^* L^{D}_{|S})\\
s &\mapsto & s_{|S}\end{array}\end{equation}

What allows us to reduce to the dimension $1$ is that under the definitions of $\phi_M$ and of $W$, allow us to write
after the relation \ref{M-invariance} of the preceding lemma \ref{invariance} and with the definitions of the infinitesimal neighborhood that:
\begin{equation}\label{�quivalence}s_{|S}=0 \Longleftrightarrow s\circ\phi_{M|S_0}=0, \end{equation}
this last equivalence will give an almost optimal zero lemma.

\begin{lemma}\label{zeros}(Injectivity Criterion, Zeros Lemma)

Let $P_1,\ldots, P_Z$ be distinct points of the elliptic curve $E$ and for $1\leq j\leq Z$ we define
 $P_j'=\phi_M(P_j)=(P_j,[M]P_j)$. Let, $T_1,\ldots, T_Z$ be non-negative integers
and $D_1,D_2$ be positive integers. Denoting $L=\Ocal_E(O_E)$ et $L^{D_1,D_2}=p_1^*L^{D_1}\otimes
p_2^* L^{D_2}$, if the two following inequalities are satisfied:

\begin{itemize}
\item \[M^2>D_1\]
\item \[\sum\limits_{j=1}^Z T_j>D_1+M^2D_2\]
\end{itemize}
 and if we choose the vector field $W$ as in the beginning of this section, equation \ref{W},
then the following evaluation morphism, obtained by restriction,

 \[\Phi: H^0\left(E\times E,L^{D_1,D_2}\right)\longrightarrow \bigoplus\limits_{j=1}^Z H^0\left(V(\phi_M(P_j),W,T_j),L^
{D_1,D_2}_{|V(\phi_M(P_j),W,T_j)}\right),\]

 is injective.
\end{lemma}
\begin{proof}

We have to provide necessary conditions in order that the following morphism, $\begin{array}{l}H^0(E\times E,L^{D,D})\rightarrow
H^0(S,L^{D,D}_{|S})\\s\mapsto s_{|S}\end{array}$ be injective. Which is equivalent to find necessary conditions so that if
$s_{|S}=0$ then $s=0$, the relation \ref{equivalence} then tells us that it is equivalent to find necessary conditions so that if $s\circ\phi_{M|S_0}= 0$ then $s= 0$.

Two situations occurs:

\begin{itemize}
\item  Either $s\circ\phi_M\equiv 0$ (identically).

Then if $E_M$ denotes the subgroup $\phi_M(E)$, we have necessarily:
\[E_M\leq\textrm{div}(s),\]
because the section should then be zero all over $E_M$.

Which numerically implies,

\[\begin{array}{c}E_M\cdot p_1^*L=E_M\cdot(\{O_E\}\times E)=1\\
E_M\cdot p_2^*L=E_M\cdot(E\times\{O_E\})=M^2\end{array}\]

so that the condition $E_M\leq\textrm{div}(s)$ gives:

\begin{equation}\label{1}M^2=E_M\cdot p_2^*L\leq\textrm{div}(s)\cdot p_2^*L=(p_1^*L^{D_1}\otimes p_2^*L^{D_2})\cdot p_2^*L=D_1\end{equation}

So if we suppose that $M^2>D_1$ we are certain that $s\circ\phi_M\neq 0$

\item Or $s\circ\phi_M\neq 0$ and $s\circ\phi_M$ is zero up to order $T_1$ at $P_1$ as well as to order
$T_j$ at the points $P_j$ for $1\leq j\leq Z$.

We have,

\[\begin{array}{c}
p_2\circ\phi_M=[M]\\
p_1\circ\phi_M=[1]
\end{array}\]
so, as $L$ is symmetric and naturally cubist, we have by the cube formula,

\[\begin{array}{c}
\phi_M^*(p_2^*L)=L^{M^2}\\
\phi_M^*(p_1^*L)=L
\end{array}\]
so that
\[\phi_M^*(p_1^*L^{D_1}\otimes p_2^*L^{D_2})=L^{D_1+M^2D_2},\]
the line bundle $L$ being of degree $1$, we have
\[\textrm{deg}\left(\phi_M^*(p_1^*L^{D_1}\otimes p_2^*L^{D_2})\right)=D_1+M^2D_2,\]
so that if $s\circ\phi_{M|S_0}=0$ and $s\circ\phi_M\neq 0$ we should have a zero order which is less or equal to the degree of the line bundle, therefore

\begin{equation}\label{zero}\sum\limits_{j=1}^Z T_j\leq D_1+M^2D_2.\end{equation}

So that by the negation of the inequalities \ref{1} and \ref{zero} we obtain the result of the lemma.
\end{itemize}

\end{proof}

\begin{corollary}\label{criter}
In the preceding framework of our construction, if
\begin{enumerate}
\item If \[M^2>D\]
\item and if \[T_0+ZT_1>D+M^2D,\]
\end{enumerate}
then the morphism
\[\Phi:\mathcal{H}\rightarrow\mathcal{F},\]
of the section 5.2.2, is injective, hence the slope inequality of section 5.2.2 is true.
\end{corollary}

\section{Evaluation of each Term of the Slope Inequality}

Now that we have stated our version of the Slope Inequality in chapter 5 and that we have provided a criterion in order that this inequality be true; we are going to evaluate each term of this inequality.

This will give us Lang's conjecture in the semi-stable case in chapter 8. From this we will deduce the full proof of the conjecture in section 9.

We place ourselves in the context of the last case of section 4 and we suppose that the criterion of section 6 is satisfied. This gives that the slope inequality of section 5 is true and we then are going to make the corresponding evaluations.

Let us make the situation more precise, we place ourselves in the
difficult case of Serge Lang's conjecture (last case of theorem \ref{reduction} section 4).

We consider for this purpose a semi-stable elliptic curve, $E/K$, $P$ a rational point of $E$, from which we deduce $Z+1$
rational points $P_0,P_2,\ldots,P_Z$ of the form $P_k=[m_k]P$, $0\leq k\leq Z$, where $m_k$ is bounded depending on $Z$ after the reductions made explicit in section 4.

With the notation, for all $0\leq k\leq Z$,

 \[\tilde{S}(P_k)=\left\{v\in M_K^{0,sm}|\frac{1}{3}\leq\frac{\ord_v(P_k)}{N_v}\leq\frac{2}{3}\right\},\]
 $\ord_v(P_k)$ being the $v-adic$ order of the point $P_k$ of $E$ as defined previously.
 According to the reduction of section 4, we are in a case where for all $0\leq k\leq Z$

\begin{equation}\label{nonarchi}
\sum\limits_{v\in\tilde{S}(P_0)\cap \tilde{S}(P_k)}N_v\log N_{K/\Q}(v)\geq\frac{1}{550}\log N_{K/\Q}\Delta_{E/K}.
\end{equation}

We consider moreover an integer $M>1$ and we place ourselves over $E\times E$ on which we denote for  $0\leq k\leq Z$, \[P_k'=\phi_M(P_k)=(P_k,[M]P_k).\]

The points $P_k$ define sections $\mathcal{P}_k:\spec\;\Ocal_K\rightarrow\Nr$ and the points $P_k'$ sections
$\mathcal{P}_k':\spec\;\Ocal_K\rightarrow\Nr\times\Nr=\Nr'$.

Moreover in the tangent space $t_{E\times E}$ we consider the vector line
\begin{equation}\label{Wdef} W=(X_\inv^1+MX_\inv^2)K,\end{equation}
where $X_\inv^1$ and $X_\inv^2$
are defined in the definition \ref{XXX} of the preceding section.

Let us sum up the content of this section:

In section 7.1 we make some preliminary evaluations and we especially estimate the degree of our line bundle which is the left term of the slope inequality.

From this we are going to evaluate the slopes and the ranks of the sub-quotients $\mathcal{G}^k$ in section 7.2. From this we evaluate a full part of the slope inequality as corollary \ref{mu} of section 7.2.

It then remains two things to evaluate:

\begin{itemize}
	\item The archimedean norms
	\item The non-archimedean norms.
\end{itemize}

We compute an evaluation of the archimedean norms in the section 7.4 according to the evaluations previously made by Jean-Beno\^it Bost as lemma 5.8 of [3].

Section 7.5 represent by itself a progress. We make some accurate estimate of the non-archimedean norms. It seems that there were no known such non-trivial estimates in the literature. Till now only trivial estimates were known for these norms.

As a summary the slope inequality of section 5, looks like:

\[\widehat{\mathrm{deg}}_n(\mathcal{H})\leq\sum\limits_{k=0}^{k_{\mathrm{max}}}\mathrm{rk}\left(\mathcal{H}^k/\mathcal{H}^{k+1}\right)\left[\hat{\mu}_{\mathrm{max}}(\mathcal{G}^k)+h(\mathcal{H}^k,\mathcal{G}^k,\Phi^k)\right],\]
where:
\[h(\mathcal{H}^k,\mathcal{G}^k,\Phi_k)=\frac{1}{[K:\Q]}\sum\limits_{\sigma\in M_K^{\infty}}\log|||\Phi_k|||_{\sigma}+\frac{1}{[K:\Q]}\sum\limits_{v\in M_K^{0}}\log|||\Phi_k|||_v.\]

This means that our slope inequality can be written:

\[(A)\leq (B)+(C)+(D),\]

where:

\begin{itemize}
	\item $(A)=\widehat{\mathrm{deg}}_n(\mathcal{H})$ is estimated in section 7.2 and was indeed estimated in section 2.
	\item $(B)$ is given by:
	\[(B)=\sum\limits_{k=0}^{k_{\mathrm{max}}}\mathrm{rk}\left(\mathcal{H}^k/\mathcal{H}^{k+1}\right)\left[\hat{\mu}_{\mathrm{max}}(\mathcal{G}^k)\right]\],
	is estimated in section 7.2.
	\item $(C)$ is given by:
	\[C=\frac{1}{[K:\Q]}\sum\limits_{\sigma\in M_K^{\infty}}\sum\limits_{k=0}^{k_{\mathrm{max}}}\left(\mathcal{H}^k/\mathcal{H}^{k+1}\right)n_{\sigma}\log|||\Phi_k|||_{\sigma}\]
and is estimated in section 7.3.
\item $(D)$ is given by
\[(D)=\frac{1}{[K:\Q]}\sum\limits_{v\in M_K^{0}}\sum\limits_{k=0}^{k_{\mathrm{max}}}\left(\mathcal{H}^k/\mathcal{H}^{k+1}\right)\log|||\Phi_k|||_{v},\]
	and is estimated in section 7.4.
\end{itemize}

\subsection{Preliminary Computations}

In this section we mainly establish the lemma \ref{WWWW}. This lemma \ref{WWWW} will help a lot computing the slopes of the $\mathcal{G}^k$.

The lemma \ref{degree} that follows is a reminder of a computation make in section 2. It is an evaluation of the degree of our line bundle that is a whole part of the slope inequality.

We have the following lemma, that will help us simplifying the computations:

\begin{lemma}\label{WWW}

By the lemma \ref{invariance}, $W$ is the tangent space of the subgroup of $E\times E$ image of $E/K$ by $\phi_M$.

On the other hand we have also that $\mathcal{W}=W\cap t_{\Nr'/\Ocal_K}$, intersection of $W$ and of the tangent space at origin of $\Nr'=\Nr\times \Nr$, is equal $t_{\tilde{\phi}_M(\mathcal{N})}$, i.e; it is equal to the tangent module at origin of the sub-group $\tilde{\phi}_M(\Nr)$, image of $\Nr$ by the embedding defined by $\Pcal\longmapsto (\Pcal,[M]\Pcal)$.

\end{lemma}

\begin{proof}

First, the fact that $W$ be the tangent at origin of the sub-group $\phi_M(E)$ de $E\times E$
follows directly from lemma \ref{invariance}.

The same for $\mathcal{W}$ is not so obvious.

 We have got,
\[\mathcal{W}=W\cap t_{\Nr'}=t_{\phi_M(E)}\cap t_{\Nr'},\]
but, if one consider the following automorphism of $E\times E$ 
\[(A,B)\longmapsto (A,B-[M]A),\]
as well as the corresponding integral automorphism over $\Nr'=\Nr\times\Nr$,
\[(\mathcal{A},\mathcal{B})\longmapsto (\mathcal{A},\mathcal{B}-[M]\mathcal{A}),\]
we are now in a situation were we compare the tangent of the first factor of $E\times E$ with the tangent of the first factor of $\Nr\times\Nr$. We deduce that

\[   t_{\phi_M(E)}\cap t_{\Nr'}=t_{\tilde{\phi}_M(\Nr)},\]
therefore
\[\mathcal{W}=t_{\tilde{\phi}_M(\Nr)}.\]

\end{proof}

The main consequence of this is the following lemma

\begin{lemma}\label{WWWW}

The embedding \ref{G1} is indeed an isomorphism

\begin{equation}
\mathcal{G}^{k}\cong\left\{\begin{array}{l}\mathcal{P}_0^*\mathcal{L}^{(D,D)}_\mb\otimes \textrm{Sym}^k\check{\mathcal{W}};\textrm{for},k=0,\ldots,T_0;\\\\

  \mathcal{P}_\lambda^*\mathcal{L}^{(D,D)}_\mb\otimes \textrm{Sym}^l\check{\mathcal{W}},\textrm{for}\;k=T_0+1+(\lambda-1)T_1+\lambda+l..
  \end{array}\right.
\end{equation}
\end{lemma}

\begin{proof} This comes from the corollary \ref{WWW} because the embedding \ref{G1} of section 5.2.1 is an isomorphism in the case were the considered sub-space is the tangent of a smooth sub-group. This comes from the fact that a smooth sub-group is differentially smooth and therefore the divided powers of this field that appear when we do the restrictions to the infinitesimal neighborhoods do not alter the integral structures then considered.  See \cite{Gaudron1} annexe C and \cite{Viada2}.\end{proof}

Here we just remind the corollary \ref{degreecub} of section 2.3.:

\begin{lemma}\label{degree}(The degree)

\begin{equation}\label{deg}
\widehat{\textrm{deg}}_n(\pi_{*} \Line^{(D,D)}_\mb)\geq-D^2h_F(E/K)+\frac{D^2}{2}\log\left(\frac{D}{2\pi}\right)
\end{equation}
\end{lemma}
\begin{proof} See corollary \ref{degreecub} section 2.3. above.\end{proof}
	
\subsection{Ranks and slopes of the subquotients}

In this section we mainly compute the slopes $\hat{\mu}(\mathcal{G}^k)$ which provides then a estimation of a whole part of the slope method. This is corollary \ref{mu}. The main idea is to use the isomorphisms of section 7.1 so that these slopes become computable.

We have, after the lemma \ref{WWWW} of the preceding section 
\begin{equation}
\mathcal{G}^{k}\cong\left\{\begin{array}{l}
\Pcal_0^{'*}\Line_\mb^{(D,D)}\otimes \textrm{Sym}^k(\check{\mathcal{W}}),\;\textrm{for}\; k=0,\ldots,T_0,\\\\

\Pcal_\lambda^{'*}\Line_\mb^{(D,D)}\otimes \textrm{Sym}^l(\check{\mathcal{W}}),\;\textrm{for}\; k=T_0+1+(\lambda-1)T_1+\lambda+l.
\end{array}\right.
\end{equation}

but the $\Pcal^{'*}\Line_\mb^{(D,D)}$ are hermitian line bundles (over $\spec\;\Ocal_K$), therefore:

\begin{gather}
\hat{\mu}_{\max}(\mathcal{G}^{k})\\\leq\\\left\{\begin{array}{l}
 \widehat{\textrm{deg}}_n\left(\Pcal_0^{'*}\Line_\mb^{D,D}\right)+
\hat{\mu}_{\max}(\textrm{Sym}^k(\check{\mathcal{W}})),\;\textrm{for}\; k=0,\ldots,T_0,\\\\
\widehat{\textrm{deg}}_n\left(\Pcal_\lambda^{'*}\Line_\mb^{D,D}\right)+
\hat{\mu}_{\max}(\textrm{Sym}^l(\check{\mathcal{W}}))
,\;\textrm{for}\; k=T_0+1+(\lambda-1)T_1+\lambda+l.
\end{array}\right.
\end{gather}
as we are working on a cubist model, denoting by $\Pcal_{M,j}$ the integral section of $\Nr$
corresponding to a point $[M]P_j$, we have:

\begin{equation}
\begin{array}{rl}\widehat{\textrm{deg}}_n\left(\Pcal_j^{'*}\Line_\mb^{(D,D)}\right)&=\frac{1}{[K:\Q]}\widehat{\textrm{deg}}
\left(\Pcal_j^{'*}\left(p_1^*\Line_\mb^{(D)}\otimes p_2^*\Line_\mb^{(D)}\right)\right)\\
&=\frac{1}{[K:\Q]}\widehat{\textrm{deg}}\left(p_1^*\Pcal_j^{*}\Line_\mb^{(D)}\otimes p_2^*\Pcal_{M,j}^*\Line_\mb^{(D)}\right)\\
&=\frac{1}{[K:\Q]}\widehat{\textrm{deg}}(\Pcal_j^{*}\Line_\mb^{(D)})+\frac{1}{[K:\Q]}\widehat{\textrm{deg}}(\Pcal_{M,j}^{*}
\Line_\mb^{(D)})\\&=D\hat{h}(P_j)+D\hat{h}([M]P_j)\\&=D(1+M^2)m_j^2\hat{h}(P).\end{array}
\end{equation}

We also see that the $\mathcal{G}^{k}$ are line bundles, we deduce that

\begin{lemma}\label{i}
If $\Phi$ is injective:
\[\textrm{rk}(\mathcal{H}^{k}/\mathcal{H}^{k+1})=\textrm{rk}(\Phi(\mathcal{H}^{k}/\mathcal{H}^{k+1}))\leq \textrm{rk}(\mathcal{G}^{k})=1. \]
\end{lemma}

\begin{lemma}\label{crucial}

Suppose that the following condition is satisfied:

\begin{enumerate}
\item\[M^2>D,\]
\item\[T_0+ZT_1>D+M^2D,\]
\end{enumerate}

Then there exists integers $0\leq k_1<k_2<\ldots<k_{D^2}\leq D(1+M^2)$ so that
\[\forall\:1\leq i\leq D^2,\:\textrm{rk}(\mathcal{H}^{k_i}/\mathcal{H}^{k_i+1})=1,\]
and for the other indexes of the filtration:
\[\textrm{rk}(\mathcal{H}^{k}/\mathcal{H}^{k+1})=0.\]
\end{lemma}
\begin{proof}
We have the following filtration:
\[\mathcal{H}=\mathcal{H}^0\supset\mathcal{H}^1\supset\cdots\supset\mathcal{H}^k\supset\mathcal{H}^{k+1}\supset\cdots,\]
and for $k>D(1+M^2)$, the injectivity of the evaluation morphism implies that $\mathcal{H}^k=0$.

On the other hand, the Riemann-Roch theorem implies that
\[\textrm{rk}(\mathcal{H})=D^2,\]
and the lemma \ref{i} which precedes gives
\[\forall\; k,\;\textrm{rk}(\mathcal{H}^k/\mathcal{H}^{k+1})\leq 1,\]
said differently, we have a filtration
\[\mathcal{H}\supset\mathcal{H}^1\supset\cdots\supset\mathcal{H}^k\supset\mathcal{H}^{k+1}\supset\cdots\supset\mathcal{H}^{D(1+M^2)}\supset\{0\},\]
with $\textrm{rk}(\mathcal{H})=D^2$ and $\textrm{rk}(\mathcal{H}^k/\mathcal{H}^{k+1})\leq 1.$

We deduce from this the result of the lemma.
\end{proof}

On the other hand we have the following lemma:

\begin{lemma} For all $k\geq 0$
\[\hat{\mu}(Sym^l\check{\mathcal{W}})=\frac{l}{2}\log(D(1+M^2))+lh_F(E/K)-\frac{l}{2}\log(\pi).\]\end{lemma}
\begin{proof}

We want to compute

\[\hat{\mu}_{\textrm{max}}(Sym^l\check{\mathcal{W}}).\]

Let us first remark that $ \check{\mathcal{W}}$ is locally free of rank one, therefore:

\[Sym^l\check{\mathcal{W}}\cong \check{\mathcal{W}}^{\otimes l},\]

 \[\hat{\mu}_{\textrm{max}}(Sym^l\check{\mathcal{W}})=\hat{\mu}_{\textrm{max}}(\check{\mathcal{W}}^{\otimes l})=l\hat{\mu}_{\textrm{max}}(\check{\mathcal{W}})=\frac{l}{[K:\mathbb{Q}]}\widehat{\textrm{deg}}(\check{\mathcal{W}}).\]

 We uses the fact (true for any such subspace) that

 \[\widehat{\textrm{deg}}(\check{\mathcal{W}})=-\widehat{\textrm{deg}}(\mathcal{W}).\]

Remark that

\[\phi_{M}^*L^{D,D}=L^{D(1+M^2)},\]

in fact

\[\begin{array}{c}
p_2\circ\phi_M=[M]\\
p_1\circ\phi_M=[1]
\end{array}\]
therefore, as $L$ is symmetrical and naturally cubist, we have by the cube formula

\[\begin{array}{c}
\phi_M^*(p_2^*L)=L^{M^2}\\
\phi_M^*(p_1^*L)=L
\end{array}\]
therefore
\[\phi_M^*(p_1^*L^{D}\otimes p_2^*L^{D})=L^{D(1+M^2)},\]

We now use the formula (D.1) of the proposition D.1. de \cite{Bost1}  for our choice on $L$:

\begin{equation}\label{TTT1}h_{E,L^{D(1+M^2)}}(t_E)=h_F(E/K)+\frac{1}{2}\log(D(1+M^2))-\frac{1}{2}\log(\pi).\end{equation}

And we recall that in \cite{Bost1} this height is defined by:

\begin{equation}\label{TTT2} h_{E,L^{D(1+M^2)}}(t_E)=-\widehat{\textrm{deg}}(t_{\mathcal{N}},\|\cdot\|_{L^{D(1+M^2)}} ),\end{equation}
where for any archimedean place $\sigma$, $\|\cdot\|_{L^{D(1+M^2)}_\sigma}$ is the Riemann form associated to the first Chern class of $L_\sigma^{D(1+M^2)}$, 
which is the associated Riemann form.

For a section $s$ of $t_{\mathcal{N}}$, we have:

\[h_{E,L^{D(1+M^2)}}(t_E)=-\log\left(\textrm{Card}(t_{\mathcal{N}}/s\mathcal{O}_K)\right)+\sum\limits_\sigma\log\|s\|_{\sigma,L^{D(1+M^2)}},\]

But, from the corollary \ref{WWW}, we have

\begin{equation}\label{XXX1}\mathcal{W}=\phi_{M_*}(t_{\mathcal{N}}),\end{equation}


On the other hand, by the functoriality of the first Chern class:

\[\phi_M^*(c_1(L^{D,D}_\sigma))=c_1(L^{D(1+M^2)}_\sigma),\]

and therefore for the corresponding Riemann forms that we denote $H_{L^{D,D}}$ et $H_{L^{D(1+M^2)}}$:

\[\phi_M^*(H_{L^{D,D}_\sigma})=H_{L^{D(1+M^2)}_\sigma}.\]

But on the other hand, by the definition of $\phi_M^*(H_{L^{D,D}_\sigma})$ we have that for any section $s$ of $t_{\mathcal{N}}$:

\[\phi_M^*(H_{L^{D,D}_\sigma})(s,s)=H_{L^{D,D}_\sigma|W}(\phi_{M_*}(s),\phi_{M_*}(s)),\]

Where $_{|W}$ denotes the restriction to $W$.

We deduce that

\[H_{L^{D(1+M^2)}_\sigma}(s,s)=H_{L^{D,D}_\sigma|W}(\phi_{M_*}(s),\phi_{M_*}(s))\]

 hence for the norms that we denote $\|\cdot\|_{L^{D(1+M^2)}_\sigma}$ and $\|\cdot\|_{L^{D,D}_\sigma}$ we have for any section
$s$ of $t_{\mathcal{N}}$:

\begin{equation}\label{XXX2}\|s\|_{L^{D(1+M^2)}_\sigma}=\|\phi_{M_*}(s)\|_{L^{D,D}_\sigma|W},\end{equation}

on the other hand, as $\phi_M$ is an embedding:

\begin{equation}\label{XXX3}\textrm{Card}(\mathcal{W}/(\phi_{M_*}s)\mathcal{O}_K)=\textrm{Card}(t_{\mathcal{N}}/s\mathcal{O}_K).\end{equation}

We deduce from \ref{XXX1}, \ref{XXX2} and \ref{XXX3} that

\begin{equation}\widehat{\textrm{deg}}(\mathcal{W},\|\cdot\|_{L^{D,D}})=\widehat{\textrm{deg}}(t_{\mathcal{N}},\|\cdot\|_{L^{D(1+M^2)}}),\end{equation}

which, with the equations \ref{TTT1} and \ref{TTT2}, finally gives

\[\hat{\mu}(Sym^l\check{\mathcal{W}})=\frac{l}{2}\log(D(1+M^2))+lh_F(E/K)-\frac{l}{2}\log(\pi).\]

\end{proof}

We deduce the following estimates:

\begin{lemma}(Upper bound for the maximal slopes)

With the preceding notations:

For $k=0,\ldots,T_0$,
\[\hat{\mu}_{\max}(\mathcal{G}^{k})\leq D(1+M^2)m_0^2\hat
{h}(P)+kh_F(E/K)+k\log\sqrt{\frac{D(1+M^2)}{\pi}},\]

and for $k=T_0+1+(\lambda-1)T_1+\lambda+l$, with $1\leq\lambda\leq Z$ and $0\leq l\leq
T_1$,

\[\hat{\mu}_{\max}(\mathcal{G}^{k})\leq D(1+M^2)m_\lambda^2\hat
{h}(P)+lh_F(E/K)+l\log\sqrt{\frac{D(1+M^2)}{\pi}}.\]
\end{lemma}

\begin{proof}
This results from the evaluation of the degree and of the height made previously.
\end{proof}

We conclude this section by an estimation of the corresponding part of our inequality:

\begin{corollary}\label{mu}

If the evaluation morphism $\Phi$ is injective,

\begin{equation}\begin{array}{c}\sum\limits_{k=0}^{\m}\textrm{rg}(\mathcal{H}^{k}/\mathcal{H}^{k+1})\hat{\mu}_{\tiny{\max}}(\mathcal{G}^{k
})\\\\
\leq\\\\ D^3(1+M^2)m_{\max}^2\hat{h}(P)+\cdots\\\\
+\left[\frac{T_0(T_0+1)}{2}+T_1(D^2-T_0+T_1-1)\right]h_F(E/K)+\cdots\\\\
+\left[\frac{T_0(T_0+1)}{2}+T_1(D^2-T_0+T_1-1)\right]\log\sqrt{\frac{D(1+M^2)}{\pi}}.
\end{array}\end{equation}
\end{corollary}
\begin{proof}

By what precedes and with the notations of lemma \ref{crucial}, if $\Phi$ is injective
\[\begin{array}{c}\sum\limits_{k=0}^{\m}\textrm{rg}(\mathcal{H}^{k}/\mathcal{H}^{k+1})\hat{\mu}_{\tiny{\max}}(\mathcal{G}^{k
})\\\\\leq\\\\\sum\limits_{i=1}^{D^2}\hat{\mu}_{\tiny{\max}}(G^{k_i}),\end{array}\]
and we remark that we can maximize the $D^2$ terms, $\hat{\mu}_{\tiny{\max}}(G^{k_i})$, by taking for each $m_i$ the maximal $m_i$ and by choosing the $(T_0-T_1+1)$ zeros order in between $T_1$ and $T_0$ and the $(D^2-T_0+T_1-1)$ being equal to $T_1$.

We deduce that:

\[\begin{array}{c}\sum\limits_{k=0}^{\m}\textrm{rg}(\mathcal{H}^{k}/\mathcal{H}^{k+1})\hat{\mu}_{\tiny{\max}}(\mathcal{G}^{k
})\\\\\leq\\\\\ D^3(1+M^2)m_{\max}^2\hat{h}(P)+\cdots\\\\
+\left[\frac{T_0(T_0+1)-T_1(T_1+1)}{2}+T_1(D^2-T_0+T_1-1)\right]h_F(E/K)+\cdots\\\\
+\left[\frac{T_0(T_0+1)-T_1(T_1+1)}{2}+T_1(D^2-T_0+T_1-1)\right]\log\sqrt{\frac{D(1+M^2)}{\pi}},\end{array}\]
and by canceling the "$-\frac{T_1(T_1+1)}{2}$" which contributes favorably to the inequality, we obtain the result.

\end{proof}

\subsection{The archimedean norms}

In this section our goal is to make some estimation of the archimedean terms of our slope inequality. In order to make the computations we use the lemma 5.8 of [3] from Jean-Beno\^it Bost and simply apply it to our situation. This gives the lemma \ref{archim} and its corollary \ref{archi}.

The results of this section are mainly due to an Arakelov form of the Schwartz lemma due to Jean-Benoit Bost, it appears without any demonstration in \cite{Bost1}, for a proof see the thesis of Evelina Viada-Aehele \cite{Viada}.

We need to define an injectivity radius $\rho_\sigma(E\times E,\Line_\mb^{(D,D)})$ that we can see as the smallest real number so that the exponential map $exp:T_{E\times E}\rightarrow
E\times E$ between the ball of center 0 and of radius $\rho_\sigma(E\times E,\Line_\mb^{(D,D)})$ of $T_{E\times E}$ is an homeomorphism on its image for the metric induced by $c_1(\Line_\sigma^{(D,D)})$.

On $E$, the Riemann form associated to $L_\sigma$ is the standard Riemann form:

\[H(z_1,z_2)=\frac{z_1\bar{z}_2}{\textrm{Im}\;\tau_\sigma}.\]

We deduce, see for example \cite{CAV} 2.6 exercise (2), that

\[c_1(L_\sigma)=\frac{i}{2\textrm{Im}\:\tau_\sigma}(dv_1\wedge dv_2),\]

on the other hand, as $L^{D,D}=(p_1^*L\otimes p_2^*L)^{\otimes D}$ et and that the first Chern class of a tensorial product is the sum of the first Chern class:

\[c_1(\Line_{\sigma}^{(D,D)})=\frac{iD}{2\textrm{Im}\:\tau_\sigma}(dv_1\wedge dv_2+dv'_1\wedge dv'_2).\]

One deduce from this that

\[\rho_\sigma(E\times E,\Line^{D,D})=\sqrt{\frac{D}{\textrm{Im}\:\tau_\sigma}}.\]

If,

\[\tilde{\rho}_\sigma:=\max\left\{1,\sqrt{\frac{D}{\textrm{Im}\:\tau_\sigma}}\right\},\]

which implies that over the usual fundamental domain of the Poincar\'e upper half-plane:
\[1\leq\tilde{\rho}_\sigma\leq\sqrt{\frac{2D}{\sqrt{3}}},\]

we then have after \cite{Bost1} lemme 5.8 with a proof in the thesis Evelina Viada-Aehle \cite{Viada},for our metrics :

\begin{lemma}\label{archim}(Upper bound for the archimdean norms)(Jean-Beno�t Bost \cite{Bost1} lemme 5.8, see also \cite{Viada} chapter 5)

 For $k=0,\ldots,T_0$,
,
\[|||\Phi_{k}|||_\sigma^2\leq  \frac{D^2}{\pi^2}(k+2)(k+1) e^{\pi D\tilde{\rho}_\sigma^2}\tilde{\rho}_\sigma^{-2(k+2)},
\]
and for $k=T_0+1+(\lambda-1)T_1+\lambda+l$ with $\lambda=1,\ldots,Z$ et $l=0,\ldots,T_1$:
\[|||\Phi_{k}|||_\sigma^2\leq  \frac{D^2}{\pi^2}(l+2)(l+1) e^{\pi D\tilde{\rho}_\sigma^2}\tilde{\rho}_\sigma^{-2(l+2)},
\]

\end{lemma}
\begin{proof} This is a direct application of the general formula for such triple norms given by Jean-Beno\^it Bost as lemma 5.8. of \cite{Bost1}.\end{proof}

We conclude that

\begin{corollary}\label{archi}

If the evaluation morphism, $\Phi$, is injective

\begin{equation}\begin{array}{c}\sum\limits_{k=0}^{\m}\textrm{rg}(\mathcal{H}^{k}/\mathcal{H}^{k+1})\left(\frac{1}{[K:\Q]}
\sum\limits_{\sigma\in M_k^\infty} n_\sigma\log|||\Phi_{k}|||_\sigma\right)\\\\\leq\\\\
\frac{\pi D^4}{\sqrt{3}}+\frac{1}{2}\log\left(\frac{(D(1+M^2)+2)!(D(1+M^2)+1)!}{(D(1+M^2)+2-D^2)!(D(1+M^2)+1-D^2)!}\right)+D^2\log\left(\frac{D}{\pi}\right).
\end{array}\end{equation}
\end{corollary}
\begin{proof}

We have, if $\Phi$ is injective, after the lemma \ref{crucial} and with the corresponding notations,

\[\begin{array}{c}\sum\limits_{k=0}^{\m}\textrm{rg}(\mathcal{H}^{k}/\mathcal{H}^{k+1})\left(\frac{1}{[K:\Q]}
\sum\limits_{\sigma\in M_K^\infty}\log|||\Phi_{k}|||_\sigma\right)\\\\\leq\\\\
\frac{1}{[K:\Q]}\sum\limits_{\sigma\in M_K^\infty}n_\sigma\sum\limits_{i=1}^{D^2}\log|||\Phi_{k_i}|||_{
\sigma}.\end{array}\]

and with the lemma \ref{archim} ,
\begin{enumerate}
\item the $D^2$ terms in $\log\left(\frac{D}{\pi}\right)$ give a term in $D^2\log\left(\frac{D}{\pi}\right)$,
\item the $D^2$ terms in $\frac{1}{2}\log(l+2)(l+1)$ can be bounded by
\[\frac{1}{2}\log\left(\frac{(D(1+M^2)+2)!(D(1+M^2)+1)!}{(D(1+M^2)+2-D^2)!(D(1+M^2)+1-D^2)!}\right),\]
\item the $D^2$ terms in $\frac{\pi D\tilde{\rho}_{\sigma}^2}{2}$ are bounded by $\frac{\pi D^4}{\sqrt{3}}$ after having bounded $\tilde{\rho}_{\sigma}^2$ by $\frac{2D}{\sqrt{3}}$.
\item and the next terms contribute favorably to the inequality and can be ignored.
\end{enumerate}

Which concludes.

\end{proof}
\subsection{The non-archimedean norms}

In this section we mainly establish a "geometric non-archimedean" Schwartz lemma in the context of the Slope Inequality. 

To the author it seems that such non-trivial non-archimedean estimates are completely new.

In sub-section 7.4.1 we show that the estimates over the product $E\times E/K$ can be reduced to estimates over a single $E/K$ thanks to our construction. This is possible thanks to the use of the "plongement \'etir'e" and to our choice of the derivation direction W. In order to achieve this a few preliminary lemma are necessary, lemma \ref{isolocal} and \ref{pullback1}. Then we translate the construction of section 5 of this paper to equivalent settings on a single curve. This allows to state the main result of sub-section 7.4.1 that say that the triple norm of the evaluation morphism over the product $E\times E/K$ is bounded by the corresponding triple norm over a single curve. From then we can simplify the estimates by working on a single curve. The reader that would be interested by the reason why we do a construction on a square $E\times E/K$ in spite of a single curve $E/K$ where everything seems to work the same can consult the remark \ref{twook} of section 8. 

In sub-section 7.4.2 we make the necessary estimates from which our Schwartz lemma follows. Lemma \ref{isomtens} consists essentially in showing that an integral section that is zero at some order at a point "factorizes" integrally. This was unexpected but plausible that integral sections can factorize in an integral way. We conclude this section by the proposition \ref{evaeva} that gives a sharp estimate of the norm of the integral section by which we "factorize".

In sub-section 7.4.3 we will be back in our key settings of section 5 of this paper. We provide proper estimates for the norms $|||\Phi_k|||$ at a split-multiplicative non-archimedean place. We conclude this section by the theorem 7 that gives a good estimate of the non-archimedean part of the Slope Inequality.

\subsubsection{Reduction to a single curve} 

In this section, we reduce our evaluations to estimates on only a single curve. Lemma \ref{isolocal} and lemma \ref{pullback1} are preliminary computations. Definition \ref{filtration1} bring the construction of section 5 accordingly to a single curve. This sub-section concludes by proposition \ref{twook} that proves that our estimates are bounded by the corresponding estimates on a single curve.

As always we denote by $p_1$ and $p_2$, the projections over the
first factor and second factor respectively as well for $E\times_K E$ as for $\Nr\times_{\Ocal_K}\Nr$.
In order to have a good evaluation of the non-archimedean norms we will reduce our case
to the case of one curve. 

First remark that by the properties of the N\'eron model, le "plongement \'etir\'e" \[\phi_M:\left\{\begin{array}{ccc} E &\rightarrow  &E\times E\\
 P&\mapsto& (P,[M]P),\end{array}\right.\]
extends naturally as a morphism:
\[\tilde{\phi}_M:\left\{\begin{array}{ccc} \Nr &\rightarrow  &\Nr\times \Nr\\
 \mathcal{P}&\mapsto& (\mathcal{P},[M]\mathcal{P}),\end{array}\right.,\]
 where $[M]$ denotes the multiplication by $M$ on the group scheme $\Nr/\spec\;\Ocal_K$

We need the two following lemmas:

\begin{lemma}\label{isolocal}
With the preceding notations

\begin{enumerate}

\item \[\tilde{\phi}_M(\V(P,T))\subset \V(\phi_M(P),W,T),\]

\item \[p_1(\V(\phi_M(P),W,T))\subset \V(P,T),\]

\item \[\begin{array}{rl}\textrm{Id}_{|\V(\phi_M(P),W,T)}&=(\tilde{\phi}_M\circ p_1)_{|\V(\Phi_M(P),W,T)}
\\&=\tilde{\phi}_{M|\V(P,T)}\circ p_{1|\V(\phi_M(P),W,T)}.\end{array}\]

\end{enumerate}

\end{lemma}

\begin{proof}

\begin{enumerate}

\item

The relation \ref{M-invariance} of lemma \ref{invariance} from the beginning of this section gives:

\[\phi_M(V(P,T))= V(\phi_M(P),W,T),\]
, then by taking the Zariski closure in $\Nr'=\Nr\times\Nr$
\[\overline{\phi_M(V(P,T))}^{\Nr'} \subset \overline{V(\phi_M(P),W,T)}^{\Nr'}=\V(\phi_M(P),W,T),\]
on the other hand we have over the generic fiber $E'=E\times E$ of $\Nr'$:
\[\left( \tilde{\phi}_M(\V(P,T))\right)_{|K}\subset \phi_M(V(P,T)),\]
therefore:
\[\tilde{\phi}_M(\V(P,T))=\overline{\left( \tilde{\phi}_M(\V(P,T))\right)_{|K}}^{\Nr'}\subset\overline{\phi_M(V(P,T))}^{\Nr'}
\subset\V(\phi_M(P),W,T).\]

\item

We use still the equality:
\[V(\phi_M(P),W,T)=\phi_M( V(P,T)),\]

which gives

\[p_1(V(\phi_M(P),W,T))= p_1\circ\phi_M\left(V(\mathcal{P},T)\right)=V(P,T),\]
because we have obviously $p_1\circ\phi_M=\textrm{Id}_E$.

Hence,
\[\overline{p_1(V(\phi_M(P),W,T))}^{\Nr'}=\overline{V(\mathcal{P},T)}^{\Nr}=\V(P,T),\]
therefore, as over the generic fiber, $E\times E$, 
\[p_1(\overline{V(\phi_M(P),W,T))}^{\Nr'})_{|K}\subset p_1(V(\phi_M(P),W,T)),\]
we deduce by taking the schematic closure the second point.

\item

We have $(\tilde{\phi}_M\circ p_1)(\mathcal{P},[M]\mathcal{P})=(\mathcal{P},[M]\mathcal{P})$ but then
, the invariance of the vector field $W$ and the relations \ref{M-invariance}, shows that,
\[\tilde{\phi}_M\circ p_{1|\V(\phi_M(P),W,T)}=\textrm{Id}_{|\V(\phi_M(P),W,T)},\]
then
\begin{gather*}\tilde{\phi}_M\circ p_{1|\V(\phi_M(P),W,T)}\\=\tilde{\phi}_{M|p_{1|\V(\phi_M(P),W,T)}}\circ p_{1|\V(\phi_M(P),W,T)}\\=
\tilde{\phi}_{M|\V(P,T)}\circ p_{1|\V(\phi_M(P),W,T)}.\end{gather*}
where we have used the second point.
\end{enumerate}
\end{proof}

\begin{lemma}\label{pullback1}
With the preceding notations
\begin{enumerate}

\item

\[\tilde{\phi}_M^*\Line_\mb^{(D,D)}=\Line_\mb^{(D(1+M^2))}.\]

\item

\[p_1^*\left(\Line^{(D(1+M^2))}_{|\V(P,T)}\right)=\Line^{(D,D)}_{|\V(\phi_M(P),W,T)}.\]

\end{enumerate}
\end{lemma}
\begin{proof}

\begin{enumerate}

\item

We have simply:
\[\begin{array}{l}p_2\circ\tilde{\phi}_M=[M]\cdot\\ p_1\circ\tilde{\phi}_M= \textrm{Id}_{\Nr},\end{array}\]
hence, as $\Line_\mb$ is cubist and symmetric, Mumford's formula gives
\[\begin{array}{l}\tilde{\phi}_M^*(p_2^*\Line_\mb)=[M]^*\Line=\Line_\mb^{M^2}\cdot\\ \tilde{\phi}_M^*(p_1^*\Line_\mb)=\Line_\mb,
\end{array}\]
therefore
\[\tilde{\phi}_M^*(p_1^*\Line_\mb^{(D)}\otimes p_2^*\Line_\mb^{(D)})=\Line_\mb^{(D(1+M^2))},\]
which is exactly the first point of the lemma.

\item

After the point 1,

\[\tilde{\phi}_M^*\Line^{(D,D)}=\Line^{(D(1+M^2))},\]
and we also have the relation,
\[(p_1^*\phi_M^*)\Line^{(D,D)}_{|\V(\phi_M(P),W,T)}=\Line^{(D,D)}_{|\V(\phi_M(P),W,T)},\]
and after the lemma \ref{isolocal}, point 3.,
\[\phi_M\circ p_{1|\V(\phi_M(P),W,T)}=\textrm{Id}_{|\V(\phi_M(P),W,T)},\]
as well as thanks to this same lemma \ref{isolocal}, point 1.,
\[\tilde{\phi}_M(\V(P,T))\subset\V(\phi_M(P),W,T),\]
so:
\[p_1^*\left(\Line^{(D(1+M^2))}_{|\V(P,T)}\right)=\Line^{(D,D)}_{|\V(\phi_M(P),W,T)},\]
which concludes.
\end{enumerate}
\end{proof}

We are going to pull back the line bundle $\Line_\mb^{(D,D)}$ by
$\phi_M$ in order to place ourselves in the case of one curve. Let us consider in analogy of the preceding filtration \ref{filtration} the following:

\begin{definition}\label{filtration1}(Filtration over one curve)
We define

 \[\mathcal{F}_0^{(1)}:=\{0\},\]
 
  then, for $k=0,\ldots, T_0$,
 
 \[\mathcal{F}_{k+1}^{(1)}=H^0\left(\mathcal{V}(\cur_0,k),\mathcal
{L}^{(D(1+M^2))}_{\mb|\mathcal{V}(\cur_0,k)}\right),\]

and for $k=T_0+1+(\lambda-1)T_1+\lambda+l$ with $\lambda=1,\ldots, Z$ and $l=0\ldots, T_1$

\begin{gather*} \mathcal{F}_{k+1}^{(1)}\\=\\ H^0\left(\mathcal{V}(\mathcal{P}_k,T_0),\mathcal
{L}^{(D(1+M^2))}_{\mb|\mathcal{V}(\mathcal{P}_k,T_0)}\right)...\\...\bigoplus\limits_{j=1}^{\lambda-1} H^0\left(\V(P_j,T_1),\mathcal{L}^{(D(1+M^2))}_{
\mb|\V(P_j,T_1)}
\right)...\\...\bigoplus H^0\left(\V(P_\lambda,l),\mathcal{L}^{(D(1+M^2))}_{\mb|\V(P_\lambda,l)}
\right),\end{gather*}

by continuing until

\[\mathcal{F}_{k_{\textrm{max}}+1}^{(1)}=:\mathcal{F}^{(1)},\]

And then we define in perfect analogy with the filtration over a product:

\[\mathcal{H}^{(1)}=H^0\left(\Nr,\Line^{(D(1+M^2))}\right),\]
\[\mathcal{F}^{(1)}=H^0\left(\V(P_0,T_0),\Line^{(D(1+M^2)}_{|\V(P_0,T_0)}\right)\bigoplus\limits_{j=1}^Z H^0\left(\V(P_j,T_1),\Line^{(D(1+M^2)}_{|\V(P_j,T_1)}\right),\]
\[\Phi^{(1)}:\mathcal{H}^{(1)}\rightarrow\mathcal{F}^{(1)},\]
\[\mathcal{F}^k_{(1)}=\textrm{ker}(\mathcal{F}^{(1)}\rightarrow \mathcal{F}^{(1)}_{k}),\],
\[\mathcal{G}^{k}_{(1)}=\textrm{ker}(\mathcal{F}^{(1)}_{k}\rightarrow \mathcal{F}^{(1)}_{k+1}),\]
\[\Phi^{(1)}:\pi_{*}\left(\Line_\mb^{D(1+M^2)}\right)\rightarrow \mathcal{F}^{(1)}\] and
\[\mathcal{H}^k_{(1)}=(\Phi^{(1)})^{-1}\left(\mathcal{F}^k_{(1)}\right),\] as well as the maps 
\[\Phi_k^{(1)}:\mathcal{H}^k_{(1)}\longrightarrow \mathcal{G}^k_{(1)},\]
pour $1\leq k\leq\m$.

\end{definition}

\begin{center} \textbf{The $p-$adic norms}\end{center}
If one endows naturally

\begin{itemize}\item $H^k_{K_v}=\mathcal{H}^k\otimes K_v$ of the norm defined by the module $\mathcal{H}^k\otimes\Ocal_{K_v}$

\item $H^{(1)k}_{K_v}=\mathcal{H}^{(1)k}\otimes K_v$ of the norm defined by the module $\mathcal{H}^{(1)k}\otimes\Ocal_{K_v}$
\item

                     $G^k_{K_v}=\mathcal{G}^k\otimes K_v$ of the norm defined by the module $\mathcal{G}^k\otimes\Ocal_{K_v}$
\item as well as:

                      $G^{(1)k}_{K_v}=\mathcal{G}^{(1)k}\otimes K_v$ of the norm defined by the module $\mathcal{G}^{(1)k}\otimes\Ocal_{K_v}$.
\end{itemize}

One then remark the following, as seen in \ref{G10}, for $s\in\mathcal{H}^k$,
\[\Phi_k(s)=s_{|\V(\cur_\lambda,W,l)}\in\pi_*\left(\Line^{D,D}_{\V(\cur_\lambda,W,l)}\right),\]
with $\lambda$ and $l$ being defined in what precedes and taking eventually the values $0$
for $k\leq T_0+1$. 

Then we show the following:

\begin{proposition}\label{twoone}

 With the preceding notations, we have by the lemma \ref{isolocal} and the lemma \ref{pullback1}, the following commutative diagram:
\begin{equation}\label{unevar}
\xymatrix{\mathcal{H}^k_{\Ocal_{K_v}}\ar[rr]^-{\phi_M^*}\ar[dd]_-{\Phi_k}   & & \mathcal{H}^k_{(1)\Ocal_{K_v}}\ar[dd]^-{\Phi_k^{(1)}}
\\\\
\mathcal{G}^k_{\Ocal_{K_v}}\ar[rr] & & \mathcal{G}^k_{(1)\Ocal_{K_v}}}
\end{equation}

and so with the preceding natural norms:

\begin{equation}\label{unevar2}
|||\Phi_k|||_v\leq|||\Phi_k^{(1)}|||_v
\end{equation}
\end{proposition}

\begin{proof} 
We have, as seen before,
\begin{itemize}\item For $l=0$
\[\mathcal{G}^k\cong H^0\left(\V(P_\lambda',W,0),\Line^{(D,D)}_{|\V(P_\lambda',W,0)}\right),\]
\item otherwise
\[\mathcal{G}^k\cong H^0\left(\V(P_\lambda',W,l),\Line^{(D,D)}_{|\V(P_\lambda',W,l)}\right)\Big{/}H^0\left(\V(P_\lambda',W,l-1),\Line^{(D,D)}_{|\V(P_\lambda',W,l-1)}\right),\]
\end{itemize}

And over only one curve:
\begin{itemize}\item For $l=0$
\[\mathcal{G}_{(1)}^k\cong H^0\left(\V(P_\lambda,0),\Line^{(D)}_{|\V(P_\lambda,0)}\right),\]
\item otherwise
\[\mathcal{G}_{(1)}^k\cong H^0\left(\V(P_\lambda,l),\Line^{(D)}_{|\V(P_\lambda,l)}\right)\Big{/}H^0\left(\V(P_\lambda,l-1),\Line^{(D)}_{|\V(P_\lambda,l-1)}\right),\]
\end{itemize}

We deduce from this that the inequality \ref{unevar2} is a consequence of the following equality that we are going to prove:

For any section $s\in H^0(\Nr\times \Nr,\Line_\mb^{(D,D)})\otimes\Ocal_{K_v}$, for any integer $T\geq 0$
\begin{equation}
\|s_{|\V(\phi_M(P),W,T)}\|_v^{\bullet }=\|s\circ\phi_{M|\V(P,T)}\|_v^{\bullet\bullet} 
\end{equation}
where the norm "$\bullet$" of the LHS is defined canonically by the $\Ocal_{K_v}-$module \[H^0(\V(\phi_M(P),W,T), 
\Line_{\mb|\V(\phi_M(P),W,T)}^{(D,D)})\otimes\Ocal_{K_v}\] whereas the norm "$\bullet\bullet$" of the RHS is defined canonically by the
$\Ocal_{K_\p}-$module \[H^0(\V(P,T),\Line_{\mb|\V(P,T)}^{(D(1+M^2))})\otimes\Ocal_{K_v}.\]

By use of the lemmas \ref{isolocal} and \ref{pullback1} which precedes, we have, 

\begin{gather*}\|s_{|\V(\phi_M(P),W,T)}\|_v^{\bullet}\leq 1\\\Leftrightarrow \\ s_{|\V(\phi_M(P),W,T)}\in H^0(
\V(\phi_M(P),W,T), 
\Line_{\mb|\V(\phi_M(P),W,T)}^{(D,D)})\otimes\Ocal_{K_v} \\\Rightarrow\\ s_{|\V(\phi_M(P),W,T)}\circ\tilde{\phi}_{M|\V(P,T)
}\in H^0(\V(P,T),\Line_{\mb|\V(P,T)}^{(D(1+M^2))})
\otimes\Ocal_{K_v},\;(\textrm{lemma \ref{pullback1}})\\ \Rightarrow\\  s\circ\tilde{\phi}_{M|\V(P,T)}\in H^0(\V(P,T),\Line_{
\mb|\V(P,T)}^{(D(1+M^2))})
\otimes\Ocal_{K_v},\;(\textrm{lemma \ref{isolocal}, 1.)}\\ \Rightarrow\\  \|s\circ\tilde{\phi}_{M|\V(P,T)}\|_v^{
\bullet\bullet}\leq 1.\end{gather*}

as well as:

\begin{gather*}\|s\circ\tilde{\phi}_{M|\V(P,T)}\|^{\bullet\bullet}\leq 1 \\\Leftrightarrow \\ s\circ\tilde{\phi}_M\in H^0(
\V(P,T),\Line_{
\mb|\V(P,T)
}^{(D(1+M^2)}))\otimes\Ocal_{K_v},\\(\textrm{and the lemmas \ref{isolocal}, 3., et \ref{pullback1}, 2., give})
\\  \Rightarrow\\ s\circ\tilde{\phi}_{M|\V(P,T)}\circ p_{1|\V(\phi_M(P),W,T)}\in H^0(\V(\phi_M(P),W,T), 
\Line_{\mb|\V(\phi_M(P),W,T)}^{(D,D)})\otimes\Ocal_{K_v}\\\Rightarrow\\  s_{|\V(\phi_M(P),W,T)}\in H^0(\V(\phi_M(P),W,T), 
\Line_{\mb|\V(\phi_M(P),W,T)}^{(D,D)})\otimes\Ocal_{K_v}\\\Rightarrow\\ \|s_{|\V(\phi_M(P),W,T)}\|_v^{\bullet}\leq 1.
\end{gather*}

as moreover $\|\phi^*s\|_v\leq\|s\|_v$ we are done.
\end{proof}

This last lemma allows us to work over $\Nr$ in spite of $\Nr\times\Nr$

\subsubsection{Evaluations on a single curve}

In this sub-section we do the necessary computations in order to get proper estimates. Lemma \ref{factorization} provides an "integral factorization" property for integral sections. We factorize accordingly by an integral section $s_0$ that is a generator for section that are zero at some point. We conclude by proposition \ref{evaeva} which computes some estimation of the corresponding norm.

\begin{lemma}\label{isomtens}
Let $\Nr$ be the N\'eron model of an elliptic curve over a number field $E/K$, $\Line$ a line bundle on $\Nr$, for a rational point $P_0$
of $E/K$, we denote $\Pcal_0$ the point of $\Nr$ extending $P_0$.

One remarks that by the Riemann-Roch theorem, the module $H^0(\Nr,\Ocal(\Pcal_0))$ is of rank 1, we denote by $s_0$ a generator.

With those data,and for an integer $T\geq 0$, the following morphism is an isomorphism:

\begin{equation}\label{isomtens1}\xymatrix{
H^0(\Nr,\Line\otimes\Ocal_\Nr(-(T+1)(\Pcal_0)))\ar[r]&\mathrm{Ker}\left(H^0(\Nr,\Line)\rightarrow H^0(\mathcal{V}(P_0,T),\Line_{|\mathcal{V}(P_0,T)})\right)\\
\tilde{s}\ar@{|->}[r] &\tilde{s}\otimes s_0^{T+1}}
\end{equation}

Said differently:

\begin{equation}\label{isomtens2}
\mathrm{Ker}\left(H^0(\Nr,\Line)\rightarrow H^0(\mathcal{V}(P_0,T),\Line_{|\mathcal{V}(P_0,T)})\right)=s_0^{T+1}\otimes H^0(\Nr,\Line\otimes\Ocal_\Nr(-(T+1)(\Pcal_0)))
\end{equation}

\end{lemma}

\begin{proof}
Let us consider the cokernel sheaf $\mathcal{M}$ of the multiplication by the section $s_0^{T+1}$, which is defined as:

\begin{equation}
\xymatrix{
0\ar[r] & \Line\otimes\Ocal_\Nr(-(T+1)(\Pcal_0))\ar[r]^-{\otimes s_0^{T+1}} &\Line\ar[r]&\mathcal{M}\ar[r]&0}
\end{equation}

From this exact sequence one deduces the cohomology exact sequence:
\begin{equation}\label{suiteline}
\xymatrix{
0\ar[r]&H^0(\Nr,\Line\otimes\Ocal_\Nr(-(T+1)(\Pcal_0))\ar[r]^-{\otimes s_0^{T+1}}&H^0(\Nr,\Line)\ar[r]&H^0(\Nr,\mathcal{M})}
\end{equation}

We have on the other hand that
\begin{equation}\label{s01}
\Line_{|\mathcal{V}(P_0,T)}=\Line\otimes\Ocal_{\mathcal{V}(P_0,T)}=\Line\otimes\left(\Ocal_{\Nr}\Big{/}\mathcal{I}_{\mathcal{V}(P_0,V)}\right)
=\Line\Big{/}\left(\Line\otimes \mathcal{I}_{\mathcal{V}(P_0,T)}\right),
\end{equation}
where $\mathcal{I}_{\mathcal{V}(P_0,T)}$ is the defining ideal of $\mathcal{V}(P_0,T)$ and we have by construction:

\begin{equation}\label{s02}
\mathcal{M}=\Line\Big{/}\left(s_0^{T+1}\otimes\Line\otimes\Ocal_\Nr(-(T+1)(\Pcal_0))\right)
\end{equation}

We are going to show that

\begin{equation}\label{s03}
\mathcal{I}_{\mathcal{V}(P_0,T)}=s_0^{T+1}\otimes\Ocal(-(T+1)(\Pcal_0))
\end{equation}

which will conclude.

We have as we have seen in the remark following the construction of the infinitesimal neighborhoods

\begin{equation}\label{s04}\mathcal{I}_{\mathcal{V}(P_0,T)}=\mathcal{I}_0^{T+1},\end{equation}
where $\mathcal{I}_0$ is the ideal of functions that are zeros at $\Pcal_0$ in $\Nr$, it suffices to show that
\begin{equation}\label{s05}
\mathcal{I}_0=s_0\otimes\Ocal(-(\Pcal_0))
\end{equation}

But if one denotes the Cartier divisor
\[(\Pcal_0)=\textrm{div}(s_0)=\{U_i,f_i\},\]
denoting by $\iota_0:\Pcal_0\hookrightarrow\Nr$ the morphism defining $(\Pcal_0)$, we have that for any open set
$U$ of $\Nr$:
\begin{equation}
\mathcal{I}_0(U)=\mathrm{Ker}\left(\Ocal_\Nr(U)\rightarrow \iota_{0*}\Ocal_{\Pcal_0}(U)\right).
\end{equation}

The section $s_0$ can be then written in local charts $\{(U_i,s_i)\}$ which satisfies:
\[f_i^{-1}s_i=f_j^{-1}s_j,\;\textrm{on}\;U_i\cap U_j,\]
and:
\[\mathrm{div}(s_i)=(\mathcal{P}_0)_{|U_i}.\]
whereas a section of $\Ocal_\Nr(-\mathcal{P}_0)$ writes in local charts $\{U_i,\sigma_i\}$ where:
\[f_i\sigma_i=f_j\sigma_j,\;\textrm{on}\;U_i\cap U_j.\]

Hence:
\[s_i\sigma_i=s_j\sigma_j,\;\textrm{on}\; U_i\cap U_j,\]
therefore the chart $\{(U_i,s_i\sigma_i)\}$ defines a rational function and we have:
\[\mathrm{div}(s_i\sigma_i)\geq (\mathcal{P}_0)_{|U_i}.\]

We first deduce from this that:
\begin{equation}s_0\otimes\Ocal_\Nr(-\mathcal{P}_0)(U)\subset\mathcal{I}_0(U),\end{equation}
for any open set $U$.

Moreover, if $j\in\mathcal{I}_0(U)$, $\frac{j}{s_i}$ is a function with no pole of the normal scheme $U_i\cap U $ and is therefore a regular function on $U\cap U_i$. Moreover on $U_i\cap U$

\[f_i\frac{j}{s_i}=f_{j}\frac{j}{s_{j}},\;\textrm{on}\;U_i\cap U_j\cap U\]
	and moreover
	\[\mathrm{div}(\frac{j}{s_i})\geq 0,\;\textrm{on}\; U_i\cap U.\]
	
	Therefore:
	\[\frac{j}{s_i}\in\mathcal{O}_\Nr(-(\mathcal{P}_0))(U_i\cap U),\]
	and so
\begin{equation}
\mathcal{I}_0(U)\subset s_0\otimes\Ocal_{\Nr}(-\Pcal_0)(U)
\end{equation}

We have thus proven the equation \ref{s05} which by itself prove the equation \ref{s03} from which we deduce with \ref{s01} and \ref{s02} that
\begin{equation}
\mathcal{M}=\Line_{|\mathcal{V}(P_0,T)}
\end{equation}
the exact sequence \ref{suiteline} then concludes the proof of this theorem.
\end{proof}

From this one can deduce the following:

\begin{corollary}\label{factorization}

Let $s_0$ be a generator of $H^0(\Nr,\Ocal_{\Nr}(\mathcal{P}_0))$ then with the notations of our construction, any section $s\in \mathcal{H}_{(1)}^k$ for $k=T_0+1+(\lambda-1)+\lambda+l$ with $\lambda=1,\ldots,Z$ and $l=0,\ldots,T_1$ can be written
\[s=s_0^{T_0+1}\otimes \tilde{s},\]
with
\[\tilde{s}\in \left\{\tilde{s}\in H^0\left(\Nr,\Line^{(D(1+M^2))}\otimes\Ocal_\Nr(-(T_0+1)(\mathcal{P}_0))\right)\left|\right.\tilde{s}_{|\tilde{\V}_k}=0\right\},\]
and
\[\tilde{\V}_{k+1}=\bigcup\limits_{j=1}^{\lambda-1}\V(\cur_j,W,T_1)\bigcup\V(\cur_\lambda,W,l),\]

\end{corollary}

\begin{proof} 
This results from the fact that if $s\in \mathcal{H}^k_{(1)}$ then with the help of the preceding lemma $s$ appears to be in the corresponding kernel thanks to the previous isomorphism. One knows moreover that $s$ and therefore $\tilde{s}$ satisfy zeros conditions over
\[\tilde{\V}_{k+1}=\bigcup\limits_{j=1}^{\lambda-1}\V(\cur_j,W,T_1)\bigcup\V(\cur_\lambda,W,l),\]

\end{proof}

The following lemma will allow to make the evaluations.

\begin{proposition}\label{evaeva} Let $s_0$ be a generator of $H^0(\Nr,\Ocal_{\Nr}(\Pcal_0))$, let $v\in M_K^{0,sm}$, the correspondence with the Tate curve, $E/K_v\cong K_v^*/q_v^{\mathbb{Z}}$ with $|q_v|_v<1$, allows us to identify a point $P\in E(K_v)$ with an element of the fundamental domain $|q_v|_v<|t_{P,v}|_v\leq 1$, in this identification we denote by $t_0$ the representative of $P_0$, then if $|q_v|_v<|t_0|_v<1,\;|q_v|_v<|t_{P,v}|_v\leq 1$, the norm $\|s_0(P)\|_v$ defined by the fiber $\mathcal{O}_{\Nr}(\Pcal_0)\otimes\Ocal_{K_v}$ over the section $\Pcal$ extending $P$ is given by:

\[\|s_0(P)\|_v=|t_{P,v}-t_0|_v\]
\end{proposition}
\begin{proof}
First we notice that thanks to the interpretation of such norms in terms of intersection multiplicities, something that seems to be due to Faltings and that is a quite direct interpretation of intersection multiplicities:
\[-\log\|s_0(P)\|_v=\mathrm{div}(s_0)\cdot\mathcal{P}=(\mathcal{P}_0)\cdot\mathcal{P},\]
as the intersection of the divisor $(\mathcal{P}_0)$ with the 1-cycle $\mathcal{P}$ at $v$ on the N\'eron model.

Now, as the divisor is horizontal and the cycle also,
\[(\mathcal{P}_0)\cdot\mathcal{P}=(P_0)\cdot P,\]
where the intersection is on $K_v$.

Then we consider the morphism given by Tate:

\[\xymatrix{K_v^*\ar[r]^-\pi & K_v^*/q_v^{\mathbb{Z}}\ar[r]^-{\cong} & E(K_v)}.\]

The section $s_0$ defines the divisor $P_0$ as a Cartier divisor and as such the space of such sections is of dimension 1 according to Riemman-Roch theorem. Moreover $\pi$ is a local isomorphism and, as usual, for $s_0$ a section of $H^0(\mathcal{N},\mathcal{O}_{\mathcal{N}})$ the pullback $\pi^*s_0$ is a theta function that vanishes on $t_0q_v^ {\mathbb{Z}}$.

Such theta function is commonly known as (a translate by $t_0$ of) the Jacobi theta function.

It is given by:

\[\theta_{P_0}(t)=(t-t_0)\prod\limits_{n\geq 1}(1-q^n\frac{t}{t_0})(1-q^n\frac{t}{t_0}).\]

As such it has all the good properties to represent the pullback $\pi^*s_0$ up to a constant.

This function can be also defined as a rigid analytic function that has $(t_0)q_v^{\mathbb{Z}}$ as divisor, the pullback of $P_0$.

We deduce from this that $\pi^*s_0$ is represented on $K_v^*$ up to a constant (dimension 1 by Riemann-Roch) $c_0$ by:

\[c_0\theta_{P_0}(t).\]

Now, let us denote by:

\[\mathcal{D}_1=\left\{t\in K_v^*\left|\right. |q_v|_v<|t|_v<1\right\},\]

We have the following diagram:

\[\xymatrix{\mathcal{D}_1\ar[r]^-{i_1} & K_v^*\ar[r]^-{\pi} & K_v^*/q_v^{\mathbb{Z}}}\]

Where $(\pi\circ i_1)$ is an isomorphism onto its image.

\[\mathcal{D}_1\cong (\pi\circ i_1) \left(\mathcal{D}_1\right)\]

and we therefore have for some constants $c_1$:

\[(\pi\circ i_1)^*(s_0)=c_1\theta_{P_0|\mathcal{D}_1},\]

From the isomorphism we moreover deduce that there exists a constant $c_1$ such that for $P\in (\pi\circ i_1) \left(\mathcal{D}_1\right):$ corresponding to $t_{P,v}$ on $\mathcal{D}_1$

\begin{equation}\label{norm1}-\log\|s_0(P)\|_v=-\log|c_1\theta_{P_0}(t_{P,v})|_v,\end{equation}

Now, it is well known that such algebraic norms (local heights) and such absolute values extend respectively to the algebraic closures, to $E(\bar{K}_v)$ to $\bar{\mathcal{D}}_1=\mathcal{D}_1\otimes \bar{K}_v$ and to $(\bar{K}_v)$.

And, on $\bar{\mathcal{D}}_1$, we can easily prove that:

\begin{equation}\label{norm2}|\theta_{P_0}(t)|_v=|t-t_0|_v,\end{equation}

of which the maximum is equal to $1$.

On the other hand, as $s_0$ is an integral generator:

\begin{equation}\label{norm3} \sup_{P\in E(\bar{K}_v)}\|s_0(P)\|_v=\|s_0\|_v=1.\end{equation}

where the "sup" is taken over the algebraic closure.

We can deduce from the three last equations, \ref{norm1}, \ref{norm2}, and \ref{norm3}, that for such $s_0$, $|c_1|_v=1$, then from \ref{norm1} we deduce easily the result of this proposition.

\end{proof}

\begin{remark}
	
	The last lemma is the main part of a geometrical proof for the formula of the canonical local height. This canonical local height is given by the "$-\log$" of the norm of a global section of the corresponding cubist line bundle. The interpretation of this norm in terms of intersection gives for the intersection $\mathcal{P}\cdot (O_\mathcal{N})$ the term $-\log|1-t_{P,v}|_v=-\log|\theta_0(t_{P,v})|_v$ as in the proposition 1 in the introduction of this paper. The remaining term of the canonical local height is given by intersection with the vertical divisor.
	\end{remark}

\subsubsection{The non-archimedean norms}

We begin sub-section 7.4.3 by our "Key Lemma" that sums up the computations of section 7.4. Our "Key Lemma" corresponds to some "geometric non-archimedean" form
of a Schwartz lemma. This moreover brings suitable estimates for the norms we consider. We conclude section 7.4 by the theorems \ref{bingo0} and \ref{nonarchibingo} that provide a good estimate for the non-archimedean part of our slope inequality.

Such sharp estimates for the non-archimedean part of the Slope Inequality were unknown till now.
 
 \begin{lemma}\label{key}(Key Lemma) Let $v$ be a place of split multiplicative reduction of a semi-stable elliptic curve over a number field, $E/K$, if we consider in the framework of our previous construction the points $P_0$ and $P_\lambda$, pour $\lambda=1,\ldots Z$, correspond respectively on the fundamental domain
 $|q_v|_v<|t|_v\leq 1$ of the Tate curve $K_v^*/q_v^{\mathbb{Z}}$, the points  $t_0$ and $t_\lambda$ ,
we denote by $s_0$ a generator of $H^0(\Nr,\Ocal_{\Nr}(\mathcal{P}_0))$, then for all sections $s\in \mathcal{H}^k_{(1)}$,
and for $k=T_0+1+(\lambda-1)T_1+\lambda+l$ with $\lambda=1,\ldots Z$ and $l=0,\ldots,T_1$, :
\[\|\Phi_k^{(1)}(s)\|_v\leq |t_0-t_\lambda|_v^{T_0+1}\]

\end{lemma}

\begin{proof}
By the corollary \ref{factorization}, $s\in\mathcal{H}^k_{(1)}$ writes $s=s_0^{T_0+1}\otimes\tilde{s}$ where:

\[\tilde{s}\in \left\{\tilde{s}\in H^0\left(\Nr,\Line^{(D(1+M^2))}\otimes\Ocal_\Nr(-(T_0+1)(\mathcal{P}_0))\right)\left|\right.\tilde{s}_{|\tilde{\V}_k}=0\right\},\]
and
\[\tilde{\V}_{k+1}=\bigcup\limits_{j=1}^{\lambda-1}\V(\cur_j,W,T_1)\bigcup\V(\cur_\lambda,W,l).\] 

Then, as we saw that in our construction that for such a section, $\Phi_k^{(1)}(s)=s_{|\V(P_\lambda,l)}$, and as moreover, by construction $s_{\V(P_\lambda,l-1)}=0$ for $l\geq 1$ and as $s_0$ is not zero on $P_\lambda$ as our points are distinct, we get from the Leibniz rule:
\[\Phi_k^{(1)}(s)=s_0(P_\lambda)^{T_0+1}\otimes(\tilde{s}_{|\V(P_\lambda,l)}).\]

The norm $\|\Phi_k^{(1)}\|_v$ is defined by the $\Ocal_{K_v}$-module $\mathcal{G}_{(1)}^k\otimes_{\Ocal_K}\Ocal_{K_v}$. 

Now if one consider the modules:
\begin{gather*}
M_1=H^0(\Pcal_\lambda,\Ocal_\Nr(\Pcal_0)_{|\Pcal_\lambda})\\
\textrm{and}\\\\M_2\\=\\\mathrm{ker}\left(H^0(\tilde{\V}_{k+1},\Line^{(D)}\otimes_\Nr(-(T_0+1)(\Pcal_0)_{|\tilde{\V}_{k+1}}))\rightarrow H^0(\tilde{\V}_{k},\Line^{(D)}\otimes_\Nr(-(T_0+1)(\Pcal_0)_{|\tilde{\V}_{k}}))\right),
\end{gather*}
one notice that $M_1^{\otimes (T_0+1)}\otimes M_2\subset\mathcal{G}^k_{(1)}$, hence for any $s_1\in M_1$ and $s_2\in M_2$ we have for the norms defined by the modules in index:
\begin{gather*}\|s_1^{\otimes (T_0+1)}\otimes s_2\|_{v,\mathcal{G}^k_{(1)}\otimes\Ocal_{K_v}}\\\leq\\\|s_1^{\otimes (T_0+1)}\otimes s_2\|_{v,M_1^{\otimes (T_0+1)}\otimes M_2\otimes\Ocal_{K_v}}\\\leq\\\left(\|s_1\|_{v,M_1\otimes\Ocal_{K_v}}\right)^{(T_0+1)}\|s_2\|_{v,M_2\otimes\Ocal_{K_v}}.\end{gather*}

From this, one deduces that:
\[\|\Phi_k^{(1)}(s)\|_v\leq\|s_0(P_\lambda)\|_v^{T_0+1}\|\tilde{s}_{|\tilde{V}(P_\lambda,l)}\|_{v,M_2\otimes\Ocal_{K_v}},\]
but one sees that $\tilde{s}_{\V(P_\lambda,l)}\in M_2$, therefore:
\[\|\tilde{s}_{|\tilde{V}(P_\lambda,l)}\|_{v,M_2\otimes\Ocal_{K_v}}\leq 1,\]
and so:
\[\|\Phi_k^{(1)}(s)\|_v\leq\|s_0(P_\lambda)\|_v^{T_0+1}\leq|t_0-t_\lambda|_v^{T_0+1}.\]
\end{proof}

\begin{theorem}\label{bingo0}
In the framework of our construction, for any $k=T_0+1+(\lambda-1)T_1+\lambda+l$ with $\lambda=1,\ldots, Z$ and $l=0,\ldots,T_1$ and for any $v\in \tilde{S}(P_0)\cap\tilde{S}(P_\lambda)$, if $\textrm{rg}(\mathcal{H}^k)\geq 1$:
\begin{equation}
|||\Phi_k|||_v\leq \left(N_{K/\Q}v\right)^{-\frac{(T_0+1)N_v}{12}}.
\end{equation}

\end{theorem}
\begin{proof}
In the case where $\textrm{rg}(\mathcal{H}^k)\geq 1$, we first uses the proposition \ref{twoone} which gives:
\[|||\Phi_k|||_v\leq|||\Phi_k^{(1)}|||_v,\]
and then the lemma \ref{key} which gives:
\begin{equation}\label{888}|||\Phi_{k}|||_v\leq|t_0-t_\lambda|_v^{T_0+1},\end{equation}
The points $t_0$ and $t_\lambda$ are chosen so that
\[v\in\tilde{S}(P_0)\cap\tilde{S}(P_\lambda)=\left\{v|\frac{1}{3}\leq\frac{\textrm{ord}_v(t_0)}{N_v}\leq\frac{2}{3}\right\}\bigcap \left\{v|\frac{1}{3}\leq\frac{\textrm{ord}_{v}(t_\lambda)}{N_v}\leq\frac{2}{3}\right\},\]
which means
\[\begin{array}{c} t_0=u\pi_v^\alpha\\
t_{\lambda}=u\pi_v^{\alpha'}\end{array}\]
where $u,u'$ are units in $K_v$, and $\pi_v$ is a uniformizing parameter of  $K_v$ and
\[\frac{N_v}{3}\leq\alpha,\alpha'\leq\frac{2N_v}{3}.\]

This finally gives in the inequality \ref{888}:
\[|||\Phi_{k}|||_v\leq|u\pi_v^\alpha-u'\pi_v^{\alpha'}|_v^{T_0+1}\leq \left(N_{K/\Q}v\right)^{-\frac{(T_0+1)N_v}{3}},\]
and concludes.
\end{proof}

\begin{theorem}\label{nonarchibingo}(Evaluation of the non-archimedean term)

In the framework of our construction, reductions of section 4 and the slope inequality of section 5.2.2. following our construction of section 5.2, we have

\begin{equation}
\begin{array}{c}
\frac{1}{[K:\Q]}\sum\limits_{k=0}^{\m}\textrm{rg}(\mathcal{H}^k/\mathcal{H}^{k+1})\sum\limits_{v\in M_K^0}\log|||\Phi_{k}|||_v\\\leq\\
-\frac{(T_0+1)(D^2-T_0)}{1650[K:\Q]}\log\left(N_{K/\Q}\Delta_{E/K}\right).
\end{array}
\end{equation}
\end{theorem}

\begin{proof}

First we remark that all norms are less or equal to 1, with the lemma \ref{crucial} denoting $k_i=T_0+1+(\lambda_i-1)T_1+\lambda+l_i$, we notice that in the $D^2$ indexes that contributes at least $D^2-T_0$ correspond to points that are different from $P_0$, hence:
\[\begin{array}{c}\frac{1}{[K:\Q]}\sum\limits_{k=0}^{\m}\textrm{rg}(\mathcal{H}^{k}/\mathcal{H}^{k+1})\sum\limits_{v\in M_K^0}\log|||\Phi_{k}|||_v\\\leq\\
\frac{1}{[K:\Q]}\sum\limits_{i=T_0+1}^{D^2}\sum\limits_{v\in\tilde{S}(P_0)\cap\tilde{S}(P_{\lambda_i})}\log|||\Phi_{k_i}|||_{v|k_i=T_0+1+(\lambda_i-1)T_1+\lambda_i+l_i}\\\leq\\
\frac{1}{[K:\Q]}\sum\limits_{i=T_0+1}^{D^2}\sum\limits_{v\in\tilde{S}(P_0)\cap\tilde{S}(P_{\lambda_i})}\left(-\frac{T_0+1}{3}N_v\log(N_{K/\Q}v)\right)\\\leq\\
-\frac{(T_0+1)(D^2-T_0)}{1650[K:\Q]}\log\left(N_{K/\Q}\Delta_E\right).
\end{array}\]

Because according to the reduction of section 4: 

\[\sum\limits_{v\in\tilde{S}(P_0)\cap\tilde{S}(\lambda)}N_v\log N_{K/\Q}(v)\geq\frac{1}{550}\log\left(N_{K/\Q}\Delta_{E/K}\right).\]

\end{proof}

\section{The proof of Lang Conjecture: semis-stable case}

Here we establish an effective inequality in the last situation of the theorem \ref{reduction} of section 4 in the case where the elliptic curve is \underline{semi-stable}. We will deduce from this the general case in the next section. Our estimates here are not that bad, and indeed quite sharp, but there are certainly also not completely optimal yet.

We are thus in a situation where the elliptic curve $E/K$ is semis-stable over a number field $K$ of degree $d=[K:\Q]$ and where for some $C_0(d)$,
\[\log \left(N_{K/\Q}\Delta_{E/K}\right)\geq C_0(d),\]
and where
\[\max\limits_{\sigma\in M_K^\infty}\left\{log^{(1)}|j_\sigma|\right\}\leq\frac{C_1}{1-\epsilon}\log N_{K/\Q\Delta_{E/K}}.\]

We suppose for simplicity that $C_1=2$ and $\epsilon=1/2$, we have thus especially:
\begin{equation}\label{h2} 4h_F(E/K)<\log|N_{K/\Q}\Delta_{E/K}|.\end{equation}

Moreover we have $Z+1$ distinct points of the form $P_k=[m_k]P$ for $0\leq k\leq Z$, and once we will have a proper value for $Z$, we can take after theorem \ref{reduction} section 4: 
\[n=2880(Z+1),\; \textrm{with}\;400d\log(2d)\]
\[N=46^2n+1,\]
and we mill have for the multiples $m_k$:
\[|m_k|\leq 24(N-1).\]

Moreover, we defined the integer $N_E$ by

\[N_E=\mathrm{lcm}\left\{N_v|\Delta_{E/K}=\prod\limits_{v|\Delta_{E/K}}\p_v^{N_v}\right\}\]

and we have after lemma \ref{N},

\begin{equation}\label{NE} N_E\leq\left(|N_{K/\mathbb{Q}}\Delta_{E/K}|\right)^{\frac{1}{e\log(2)}}\leq\left(|N_{K/\Q}\Delta_{E/K}|\right)^{0.54}.
\end{equation}

We are endowed with the line bundle $L=\Ocal_E(O_E)$ on $E$ and of with a totally symmetric cubist line bundle $\Line^{(D)}$ of $L^D$ for some $D$ such that $4N_E|D$
on the N\'eron model $\Nr/Spec\:\Ocal_K$ of $E/K$.

We built the line bundle $\Line^{(D,D)}=p_1^*\Line^{(D)}\otimes p_2^*\Line^{(D)}$ on $\Nr'=\Nr\times_{\Ocal_K}\Nr$ and for each point $P_k'=(P_k,[M]P_k)$ of $E\times E$ we consider the integral infinitesimal neighborhood $\V(P_k',W,l)$ , definition \ref{intneir}.

Then after the lemma \ref{zero}, if

\begin{enumerate}
\item $M^2>D$
\item $T_0+ZT_1>D+M^2D$
\end{enumerate}

then the morphism $\Phi$ of section 5.2.2 is injective and the the slope inequality of section 5.2 is true:

\[\begin{array}{c}
\underbrace{\widehat{\textrm{deg}}_n(\mathcal{H})}_{(A)}\\\leq\\

\underbrace{\sum\limits_{k=0}^{\m}\textrm{rg}(\mathcal{H}^k/\mathcal{H}^{k+1})\hat{\mu}_{\textrm{max}}(\mathcal{G}^k)}_{(B)}\\+\\
\underbrace{\sum\limits_{k=0}^{\m}\textrm{rg}(\mathcal{H}^k/\mathcal{H}^{k+1})\frac{1}{[K:\Q]}\sum\limits_{\sigma\in M_K^\infty}n_\sigma\log|||\Phi_k|||_\sigma}_{(C)}\\+\\
\underbrace{\sum\limits_{k=0}^{\m}\textrm{rg}(\mathcal{H}^k/\mathcal{H}^{k+1})\frac{1}{[K:\Q]}\sum\limits_{v\in M_K^0}\log|||\Phi_k|||_v}_{(D)}
\end{array}\]

This means

\begin{equation}\label{penta}(A)\leq (B)+(C)+(D),\end{equation}
with, 
\begin{itemize}
\item according to lemma \ref{degree}
\[(A)\geq-D^2h_F(E/K)+\frac{D^2}{2}\log\left(\frac{D}{2\pi}\right),\]
\item according to corollary \ref{mu}
\[\begin{array}{rl}(B)\leq & D^3(1+M^2)m_{\tiny{\max}}^2\hat{h}(P)\cdots\\&\\&
+\left[\frac{T_0(T_0+1)}{2}+T_1(D^2-T_0+T_1-1)\right]h_F(E/K)\cdots\\&\\&+\left[\frac{T_0(T_0+1)}{2}+T_1(D^2-T_0+T_1-1)\right]\log\sqrt{\frac{D(1+M^2)}{\pi}}\end{array},\]
\item according to corollary \ref{archi},
\[\begin{array}{rl} (C)\leq &\frac{\pi D^4}{\sqrt{3}}\\&+D^2\log\left(\frac{D}{\pi}\right)\\&
+\frac{1}{2}\log\left(\frac{(D(1+M^2)+2)!(D(1+M^2)+1)!}{(D(1+M^2)+2-D^2)!(D(1+M^2)+1-D^2)!}\right),\end{array}\]

\item and according to theorem \ref{nonarchibingo},
\[(D)\leq -\frac{(T_0+1)(D^2-T_0)}{1650[K:\Q]}\log\left(N_{K/\Q}\Delta_E\right).\]
\end{itemize}

As a consequence, we have the following proposition:

\begin{proposition}\label{ine2}(The inequality)
With the previous conditions, under the conditions of lemma \ref{zero}, $M^2\geq D$ and
$T_0+ZT_1>D(1+M^2)$, if, moreover $D$ is a multiple of $4N_E$, then the following inequality is true after the evaluation of the degree, lemma \ref{degree}, the evaluation of the maximal slopes, corollary \ref{mu}, the evaluation of the archimedean norms, corollary \ref{archi}, and of the non-archimedean norms, theorem \ref{nonarchibingo}:

\begin{equation}
\begin{array}{c}
D^3(1+M^2)m_{\tiny{\max}}^2\hat{h}(P)\\\geq\\\\ \frac{(T_0+
1)(D^2-T_0)}{1650[K:\Q]}\log\left(N_{K/\Q}\Delta_E\right)\left(1-t_1-t_2-t_3-t_4-t_5\right),\end{array}
\end{equation}
where

\[\begin{array}{l}
t_1=\frac{825(T_0(T_0+1)[K:\Q])}{(T_0+1)(D^2-T_0)}\times\frac{h_F(E/K)}{\log\left(N_{K/\Q}\Delta_E\right)}, \\\\
t_2=\frac{1650D^2[K:\Q]}{(T_0+1)(D^2-T_0)}\times\frac{h_F(E/K)}{\log\left(N_{K/\Q}\Delta_E\right)}\\\\
t_3=\frac{1650 T_1(D^2-T_0+T_1-1)[K:\Q]}{(T_0+1)(D^2-T_0)}\times\frac{h_F(E/K)}{\log\left(N_{K/\Q}\Delta_E\right)}\\\\
t_4=\frac{1650\pi D^4}{(T_0+1)(D^2-T_0)\log\left(N_{K/\Q}\Delta_E\right)},\\\\
\begin{array}{ll}
t_5=\frac{1650[K:\Q]}{(T_0+1)(D^2-T_0)\log\left(N_{K/\Q}\Delta_E\right)}\times\\\left(\left[\frac{(T_0+1)(T_0+1)}{2}+T_1(D^2-T_0+T_1-1)\right]\log\sqrt{\frac{D(1+M^2)}{\pi}}+\right.\\
\cdots+D^2\log\left(D(1+M^2)\right)+\cdots\\\cdots\left.+\frac{D^2}{2}\log\left(\frac{D}{2\pi}\right)+D^2\log\left(\frac{D}{\pi}\right)\right).\end{array}\end{array}\]

\end{proposition}
\begin{proof}

Here we have simply rewritten the inequality \ref{penta} which precedes and we use the upper bound:
\[\frac{1}{2}\log\left(\frac{(D(1+M^2)+2)!(D(1+M^2)+1)!}{(D(1+M^2)+2-D^2)!(D(1+M^2)+1-D^2)!}\right)\leq D^2\log(D(1+M^2)).\]
\end{proof}

\begin{lemma}
We make the following choice:
\begin{equation}\label{H}
\begin{array}{lcl}
\textrm{(H0)} && \log\left(N_{K/\Q}\Delta_E\right)\geq 10^8[K:\Q]^3\log(2[K:\Q])\\\\
\textrm{(H1)} &&  D=4000N_E[K:\Q]^2,\\\\
\textrm{(H2)} && M=\lfloor\sqrt{D}\rfloor+1,\\\\
\textrm{(H3)} && Z=17100000[K:\Q]^2,\\\\
\textrm{(H4)} && T_1=\frac{N_ED}{4000[K:\Q]}=N_ E^2\\\\
\textrm{(H5)} && T_0=2N_ED=8000N_E^2d
\end{array}
\end{equation}

With this special choice:

\begin{enumerate}
\item The morphism $\Phi$ of our version of the slope inequality is injective, therefore the inequality of proposition \ref{ine2} is true and with the preceding choice,
\item
\[\begin{array}{l}
t_1<0.1039\\
t_2<0.0516\\
t_3<0.0516\\
t_4<0.104\\
t_5<0.000002\\
\end{array}\]
\item thus:
\[1-t_1-t_2-t_3-t_4-t_5\geq \frac{1}{2}\]
\end{enumerate}
\end{lemma}

\begin{proof} This is quite straightforward and can be done by hand.
For the upper bound on $t_1,t_2$ and $t_3$ we use the fact that in our situation
\[\frac{h_F(E/K)}{\log\left|N_{K/\Q}\Delta_E\right|}<\frac{1}{4}.\]

We also use the fact that:
\[\log\left(N_{K/Q}\Delta_{E/K}\right)\geq \frac{\log{N_E}}{0.54}.\]

Let us do the verify the computations we the given hypothesis.

For the computations and only for those, we denote by $d=[K:\mathbb{Q}]$ and by $\Delta=N_{K/\mathbb{Q}}\Delta_{E/K}$

\begin{itemize}
\item Evaluation of $t_1$:

\begin{gather*}t_1=\frac{825(T_0(T_0+1)d)}{(T_0+1)(D^2-T_0)}\times\frac{h_F(E/K)}{\log\Delta}\\ \leq \\\frac{825\times T_0^2\times d\times (1+\frac{1}{T_0})}{T_0D^2\times (1+\frac{1}{T_0})(1-\frac{T_0}{D_2})}\times\frac{1}{4}\\  \leq \\\frac{825T_0d}{4D^2(1-\frac{T_0}{D^2})}=\frac{825\times 8000N_E^2d^2}{4\times(4000)^2N_E^2d^2\left(1-\frac{T_0}{D^2}\right)}\\\leq\\\frac{825\times 8000}{4\times (4000)^2\times 0.999}\leq 0.1032... \end{gather*}

\item Evaluation of $t_2$:

\begin{gather*}
t_2=\frac{1650D^2d}{(T_0+1)(D^2-T_0)}\times \frac{h_F(E/K)}{\log\Delta}\\\leq\\
\frac{1650D^2d}{T_0D^2(1+\frac{1}{T_0})(1-\frac{T_0}{D^2})}\times\frac{1}{4}\\\leq\\\frac{1650d}{4\times T_0\times(1+\frac{1}{T_0})(1-\frac{T_0}{D^2})}\\\leq\\ \frac{1650d}{4\times 8000N_E^2d\times 0.999}\\\leq\\
\frac{1650}{4\times 8000\times 0.999}=0.0516...
\end{gather*}

\item Evaluation of $t_3$:

\begin{gather*}
t_3=\frac{1650(D^2-T_0+T_1-1)d}{(T_0+1)(D^2-T_0)}\times\frac{h_F(E/K)}{\log\Delta}\\\leq\\\frac{1650T_1D^2(1-\frac{T_0}{T_1D^2}(1-\frac{T_1}{T_0}))d}{T_0D^2(1+\frac{1}{T_0})(1-\frac{T_0}{D^2})}\times\frac{1}{4}\\\leq\\ \frac{1650}{4}\times\frac{dT_1}{T_0}\times\frac{(1-\frac{T_0}{T_1D^2}(1-\frac{T_1}{T_0}))}{(1+\frac{1}{T_0})(1-\frac{T_0}{D^2})}\\\leq\\\frac{1650}{4}\times\frac{1}{8000}\times\frac{1}{0.999}\\\leq 0.0516...
\end{gather*}
\item Evaluation of $t_4$

\begin{gather*} t_4=\frac{1650\pi D^4}{(T_0+1)(D^2-T_0)\log(\Delta)}\\\leq\\
\frac{1650\pi D^4}{T_0D^2\log(\Delta)(1+\frac{1}{T_0})(1-\frac{T_0}{D^2})}\\\leq\\ \frac{1650\pi D^2}{T_0\log(\Delta)\times 0.999}\\\leq\\
\frac{1650\pi (4000)^2N_E^2d^4}{8000N_E^2d\times 10^8d^3\log(2d)}\\\leq\\\frac{1650\pi}{2\times 10^8}\leq 0.000...
\end{gather*}

\item Evaluation of $t_5$.

We have:
\[t_5=(\alpha+\beta+\gamma)\]
where:
\begin{itemize}
\item \[\alpha=\frac{1650d\left(\frac{(T_0+1)^2}{2}+T_1(D^2-T_0+T_1-1)\right)}{(T_0+1)(D^2-T_0)\log(Delta)}\times\log\;\sqrt{\frac{D(1+M^2)}{\pi}},\]
\item \[\beta=\frac{1650dD^2\log(D(1+M^2))}{(T_0+1)(D^2-T_0)\log(\Delta)},\]
\item \[\gamma=\frac{1650d\times 2D^2 \log(D)\left(1-\frac{log(2\pi)-log(\pi))}{D^2\log(D)}\right)}{(T_0+1)(D^2-T_0)\log(\Delta)}.\]
\end{itemize}
Therefore, let us evaluate the following:
\begin{gather*}
\frac{1650d\left(\frac{(T_0+1)^2}{2}+T_1(D^2-T_0+T_1-1)\right)}{(T_0+1)(D^2-T_0)\log(\Delta)}\times\log\;\sqrt{\frac{D(1+M^2)}{\pi}}\\\leq\\\frac{1650d\left(\frac{T_0^2}{2}+T_0(1-T_1)+\frac{1}{2}+T_1D^2(1-\frac{T_1-1}{D^2}\right)}{(T_0+1)(D^2-T_0)}\times\frac{\log(D)}{\log{\Delta}}\\\leq\\
\frac{1650d\left(\frac{T_0^2}{2T_0D^2}+\frac{T_1}{T_0}\right)}{(1+\frac{1}{T_0})(1-\frac{T_0}{D^2})}\times\frac{\log(D)}{\log(\Delta)}\\\leq\\
\frac{1650}{0.999}\times\left(\frac{dT_0}{T_0D}+\frac{dT_1}{T_0}\right)\times\frac{\log(D)}{\log(\Delta)}
\end{gather*}

Now we have:
\[\frac{dT_0}{D^2}=\frac{N_ED}{D^2}=\frac{N_Ed}{D}=\frac{1}{4000},\]
\[\frac{dT_1}{T_0}=\frac{N_E^2d}{8000N_E�d}=\frac{1}{8000},\]
and:
\begin{gather*} \frac{\log(D)}{\log(\Delta)}\\=\\\frac{\log(4000N_Ed)}{\log(\Delta)}\\\leq\\ \frac{\log(4000)+\log(N_E)+\log(d)}{\frac{1}{2}\left(10^8d\log(2d)+\frac{\log(N_E)}{0.54}\right)}\\\leq\\0.275001...\end{gather*}

We then deduce that:

\[\alpha\leq\frac{1650}{0.999}\times\left(\frac{1}{4000}+\frac{1}{8000}\right)\times 0.275001...\leq 0.1703...\]

On the other hand:

\begin{gather*}\beta\leq \frac{1650d\log(2D^2)}{T_0D^2\times 0.999\times\log(Delta)}\\\leq\\\frac{1650d}{0.999T_0\log(\Delta)}\times\underbrace{\frac{\log(2D^2)}{D^2}}_{\leq 0.00000046...}\\\leq 0.000000...
\end{gather*}

and according to the previous calculations:
\begin{gather*}
\gamma\leq\frac{1650dD^2\log(D)\times 0.999}{T_0D^2\log(\Delta)\times 0.999}\\\leq\\\frac{1650d\times 2}{T_0}\frac{\log(D)}{\log(\Delta)}\\\leq\\
\frac{1650\times 2\times d}{8000N_E^2d}\times 0.275001\\\leq\\ 0.1134...
\end{gather*}
 
 Therefore we found:
 
 \[t_5\leq 0.1703...+0.1134..\leq 0.2838\]
\end{itemize}

Therefore, from the previous computations, we deduce that:

\[t_1+t_2+t_3+t_4+t_5\leq 0.1032+0.0516+0.0516+0.2838=0.49...<\frac{1}{2},\]

this concludes.
\end{proof}

\begin{corollary}\label{bingo1}

Let $E/K$ be a \underline{semi-stable} elliptic curve over a number field $K$ and $P$ a rational points of $E/K$, if the hypothesis (H0), (H1), (H2), (H3), (H4) et (H5)
are satisfied, then in the last case of the reductions of theorem \ref{reduction} sections 4, we we suppose $\epsilon=1/2$ and $C_1=2$:
\begin{equation}
\hat{h}(P)\geq\frac{(T_0+1)(D^2-T_0)}{3300[K:\Q]D^3(1+M^2)(24(N-1))^2}\log\left|N_{K/\Q}\Delta_E\right|\end{equation}
and simultaneously:
\begin{equation}
\hat{h}(P)\geq\frac{(T_0+1)(D^2-T_0)}{825[K:\Q]D^3(1+M^2)(24(N-1))^2}h_F(E/K).\end{equation}
\end{corollary}

Now we can rewrite the theorem \ref{reduction} of section 4 in an effective way according to what precedes.

\begin{theorem}\label{BW0}
	
	Let $E/K$ be a \underline{semi-stable} elliptic curve over a number field $K$ of degree $d$. We denote by $\Delta_{E/K}$ its minimal discriminant ideal and by $N_{K/\Q}\Delta_{E/K}$ its norm over $\Q$. We denote by $h_F(E/K)$ the Faltings height of $E/K$.

	For $\sigma\in M_K^{\infty}$ we denote $j_\sigma$ the $j$-invariant of $E\times_\sigma\C /\C $.
	
	\begin{itemize}
		\item Either
		\[\max\limits_{\sigma\in M_K^\infty}\left\{\log^{(1)}|j_\sigma|\right\}\geq 4\log|N_{K/\Q}\Delta_{E/K}|,\]
		then
		\begin{itemize}
			\item if $P$ is a torsion rational point, its order satisfies:
			\[\mathrm{Ord}(P)\leq10207584*d*\log\left(2d\right) .\]
			\item if $P$ is of infinite order,
			\[\hat{h}(P)\geq\frac{1}{5.502* 10^{14} d^3\left(\log(2d)\right)^2}\log|N_{K/\Q}\Delta_{E/K}|,\]
			and simultaneously:
			\[\hat{h}(P)\geq\frac{1}{2.063*10^{14} d^3\left(\log(2d)\right)^2(1+4d)}h_F(E/K).\]
		\end{itemize}
		\item Either
		\[\max\limits_{\sigma\in M_K^\infty}\left\{\log{(1)}|j_\sigma|\right\}\leq 4\log|N_{K/\Q}\Delta_{E/K}|,\]
		then:
		\[h_F(E/K)\leq\frac{1}{12d}\left(1+4d\right)\log|N_{K/\Q}\Delta_{E/K}|.\]
		Moreover:
		\begin{enumerate}
			\item Either:
			\[|N_{K/\Q}\Delta_{E/K}|\leq 10^8d^3\log(2d).\]
			Then:
			\begin{itemize}
				\item If $P$ is a torsion point,
				\[\mathrm{Ord}(P)\leq 3*10^{11}d^{2.62}(\log(2d))^{1.54}.\]
				
				\item If $P$ is of infinite order:
				\[\hat{h}(P)\geq\frac{1}{3.18*10^{23}*d^{5.24}(\log(2d))^{3.08}}\log|N_{K/\Q}\Delta_{E/K}|,\]
				and simultaneously:
				\[\hat{h}(P)\geq\frac{(1-\epsilon)^3}{6*10^{22}*d^{5.24}(\log(2d))^{3.08}}h_F(E/K).\]
			\end{itemize}
			\item or,
			\[|N_{K/\Q}\Delta_{E/K}|\geq 10^8d^3\log(2d).\]
			Then the constants of theorem \ref{reduction} of section 4 can be written:
			\begin{gather*}
			Z=171*10^5d^2\\
			n=\lfloor 2880(Z+1)\rfloor\leq (2880*17100001)d^2\geq 5*10^{10}*d^2\\
			N= 11*10^{13}*d^2
			\end{gather*}
			then:
			\begin{itemize}
				\item If $P$ is of finite order and:
				\[\mathrm{Ord}(P)\leq 12N\leq 2* 10^{14} * d^2\]
				\item If $P$ is of infinite order and then
				\begin{itemize}
					\item Either
					\[\hat{h}(P)\geq\frac{1}{4*10^{32}*d^5}\log|N_{K/\Q}\Delta_{E/K}|,\]
					and simultaneously:
					\[\hat{h}(P)\geq\frac{1}{4*10^{32} d^5}h_F(E/K).\]
					\item or, finally, 
					\[\hat{h}(P)\geq\frac{1}{1.16*10^{38}*d^5}\log|N_{K/\Q}\Delta_{E/K}|,\]
					and simultaneously:
					\[\hat{h}(P)\geq\frac{1}{5*10^{38}d^5}h_F(E/K).\]
				\end{itemize}
			\end{itemize}
		\end{enumerate}
	\end{itemize}
\end{theorem}

\begin{proof}
Here we have fixed for simplicity $\epsilon=1/2$ and $C_1=2$ in theorem \ref{reduction} of section 4, it then just suffices to replace each constant of theorem \ref{reduction} of section 4 by their values, previously computed, and to make some approximations.
\end{proof} 

\begin{theorem}\label{semi} (Semi-stable case)
There exist some positive constants $\tilde{C}_d,\tilde{C}_d'$ et $\tilde{B}_d$, effectively computable, depending only on the degree $d$,
such that for any \underline{semi-stable}  elliptic curve $E/K$ over a number field $K$ of degree $d$ ,
\begin{itemize}
\item For any $K$-rational point of torsion $P$,
\[\textrm{ord}(P)\leq \tilde{B}_d,\]
\item For any $K$-rational point $P$ of infinite order:
\[\hat{h}(P)\geq \tilde{C}_d \log|N_{K/\Q}\Delta_{E/K}|,\]
and simultaneously
\[\hat{h}(P)\geq \tilde{C}_d' h_F(E/K).\]
\end{itemize}
Moreover the previous inequalities are satisfied by:
\begin{gather*}
\tilde{B}_d=2*10^{14}*d^{2.08}(log(2d))^{1.54}\\
\tilde{C}_d=\frac{1}{2*10^{38}*d^{5.24}(\log(2d))^{3.08}}\\
\tilde{C}_d'=\frac{1}{5*10^{38}*d^{5.24}(\log(2d))^{3.08}}.
\end{gather*}

\end{theorem}

\begin{proof} This is a direct corollary of the previous theorem where we just have to make a comparison of the various bounds.\end{proof}

\begin{remark}\label{twook} One could reasonably ask if a similar construction on only a single curve $E$ in spite of $E\times E$ would work. Indeed in most evaluations we reduced ourselves to the case of a single curve, it could therefore be reasonable to think that the same would work on a single curve.

A direct answer is "No" in the current formulation. The problem on only one curve seems to come from the archimedean evaluations. In fact, most of the terms of the final slope inequality are essentially the same either on a single curve or on the square if one replace $D^2$ by $D$. Notably the non-archimedean term would be given by
\[-\frac{(T_0+1)(D-T_0)}{1650[K:\Q]}\log\left|N_{K/\Q}\Delta_{E/K}\right|,\]
and almost all the other terms could be controlled proportionally but the archimedean term that would be proportional to
\[\frac{D\times D\pi\tilde{\rho}_\sigma^2}{2}\leq\frac{\pi D^3}{\sqrt{3}},\]
where now the term $D\times\cdots$ comes from the fact that we will have on only a single curve $D$ terms to consider (the dimension of the space of global sections would be $D$).

Thus, the comparison of the preceding terms would require to control the quotient:
\[\frac{D^3}{(T_0+1)(D-T_0)\log\left|N_{K/\Q}\Delta_{E/K}\right|},\]
with $T_0\leq D$,
and one then sees that such a control seems impossible because even if $T_0$ if of the order of $D$ it will remain to bound a term of the form
\[\frac{D}{\log\left|N_{K/\Q}\Delta_{E/K}\right|},\]
which for $D$ such that $4N_E|D$ is unbounded.
\end{remark}
\section{Conclusion: Full Generality}

\begin{theorem}\label{BW}
	
	Let $d\geq 1$ be an integer, there exist
	some positive constants $B_d$, $C_d$, $C_d'$, effectively computable, depending only on $d$, such that for any
	elliptic curve
	$E$ defined on a number field $K$ de degree $d$
	and for any $K-$rational point $P\in E(K)$
	
	\begin{enumerate}
		\item either $P$ is a torsion point and its order in the group of rational points satisfies
		\begin{equation}\label{Mazur-Merel}\textrm{Ord}(P)\leq B_d,\end{equation}
		\item or $P$ is not a torsion point and then, denoting $N_{K/\Q}\Delta_{E/K}$ the norm over $\Q$ of the minimal discriminant ideal of $E/K$ and $h_F(E/K)$ its Faltings height over $K$, the following two inequalities are simultaneously true
		\begin{equation}\label{BW1}\hat{h}(P)\geq C_d\log|N_{K/\Q}\Delta_{E/K}|,\end{equation}
		and
		\begin{equation}\label{BW2}\hat{h}(P)\geq C_d'\; h_F(E/K).\end{equation}
	\end{enumerate}
	
	Moreover the previous inequalities are satisfied with:
	\begin{gather*}
	B_d=2*10^{14}*d^{2.08}(log(2d))^{1.54}\\
	C_d=\frac{1}{7*10^{45}*d^{5.24}(\log(48d))^{3.08}}\\
C_d'=\frac{1}{2*10^{46}*d^{5.24}(\log(48d))^{3.08}}.
	\end{gather*}
\end{theorem}

\begin{proof}
	
	The general case can be deduced from the semi-stable case as follows.
	
	The semi-stable case was established as theorem \ref{semi} of the last section.
	
	The torsion case is quite easy with the literature. If $E/K$ is semi-stable the bound on torsion was established in theorem \ref{semi} but if is not semi-stable then for example the theorem 5.2 of the chapter 3
	of \cite{Szpiro} shows that $B_d\leq 48d$. This provides, in either cases, a bound on the torsion.
	
	On the other hand, one knows that there always exists an extension $L/K$ of degree $[L:K]\leq 24$ such that
	$E/L$ is semi-stable. Therefore we will have according to theorem \ref{semi}:
	\[\hat{h}(P)\geq \tilde{C}_{24d}\log|N_{L/\Q}\Delta_{E/L}|,\]
	where $\tilde{C}_{24d}$ is the constant of theorem \ref{semi}.
	
	Then, denoting by $w$ a place of $L$ over a place $v$ of $K$, denoting by $e(w|v)$
	the corresponding relative ramification index, we have according to the theory of classification of the special fibers of the N\'eron models of elliptic curves
	(see for example \cite{Joe2} Ch. IV)
	
	\begin{itemize}
		\item $\textrm{ord}_w\Delta_{E/L}=e(w|v)\textrm{ord}_v\Delta_{E/K}$ if $E/K_v$ is semi-stable
		\item $\textrm{ord}_w\Delta_{E/L}=0$ and $2\leq\textrm{ord}_v\Delta_{E/K}\leq 10$ if $E/K_v$ has potentially good reduction
		\item $\textrm{ord}_w\Delta_{E/L}=-\textrm{ord}_w(j_E)=-e(w|v)\textrm{ord}_v(j_E)$ and $\textrm{ord}_v\Delta_{E/K}=-\textrm{ord}_v(j_E)+6$ if $E/K_v$ has potentially multiplicative reduction.
	\end{itemize}
	
	One deduce from this an expression of the form:
	
	\[\left(N_{L/\Q}\Delta^{\textrm{inst}}_{E/K}\right)\left(\cdot N_{L/\Q}\Delta_{E/L}\right)= \left(N_{K/\Q}\Delta_{E/K}\right)^{[L:K]},\]
	where $\Delta^{\textrm{inst}}_{E/K}$ is obtained by the previous classification by comparing the orders.
	
	There are then two cases:
	
	\begin{enumerate}
		\item either $|N_{L/\Q}\Delta^{\textrm{inst}}_{E/K}|\leq \sqrt{|N_{L/\Q}\Delta_{E/L}|}$ then the previous relation shows that \[\log|N_{L/\Q}\Delta_{E/L}|\geq \frac{2[L:K]}{3}\log|N_{K/\Q}\Delta_{E/K}|\geq \frac{2}{3}\log|N_{K/\Q}\Delta_{E/K}|,\]
		therefore:
		\[\hat{h}(P)\geq \frac{2}{3}\tilde{C}_{24d}\log|N_{K/\Q}\Delta_{E/K}|.\]
		\item or $|N_{L/\Q}\Delta^{\textrm{inst}}_{E/K}|\geq \sqrt{|N_{L/\Q}\Delta_{E/L}|}$ which implies
		\[\log |N_{L/\Q}\Delta^{\textrm{inst}}_{E/K}|\geq\frac{2[L:K]}{3}\log|N_{K/\Q}\Delta_{E/K}|\geq\frac{2}{3}\log|N_{K/\Q}\Delta_{E/K}|\].
		
		Then, after the theorem \ref{local-heights} of section 2, one sees that multiplying a point point $P\in E(K)$,
		by 12 in order that $[12P]$ has smooth reduction at the additive additive places, the previous inequality implies:
		
		\[\hat{h}([12]P)\geq \frac{1}{[K:\Q]}\left(h_\infty([12]P)+\frac{1}{12}\log |N_{L/\Q}\Delta^{\textrm{inst}}_{E/K}|-\frac{1}{24}
		\log|N_{K/\Q}\Delta_{E/K}|\right)\]\[\geq\frac{1}{[K:\Q]}\left(h_\infty([12]P)+\frac{1}{72}\log|N_{K/\Q}\Delta_{E/K}|\right).\]
		here $h_\infty(\cdot)$ denotes the sum of the archimedean local heights.
		
		One can then re-use the reduction of section 3, especially lemma \ref{Elkies} and lemma \ref{HS} and the lemma \ref{pigeonhole} in order to obtain a positive archimedean contribution. We thus obtain $n$ multiples $P_1=[12m_1]P,\ldots,P_n=[12m_n]P$ with
		\[n= \left\lceil 400d\log(2d)\right\rceil,\] and
		\[N=46^2n+1,\]
		such that:
		\[\frac{1}{n(n-1)}\sum\limits_{1\leq i\neq j\leq n}h_\infty([12m_i]P-[12m_j]P)\geq 0.\]
		
		This gives, following lemma \ref{max} and remark \ref{rem0}:
		\[\frac{144(n+1)^2(N+1)^2}{2n(n-1)}\hat{h}(P)\geq\frac{1}{72[K:\Q]}\log|N_{K/\Q}\Delta_{E/K}|,\]
		or:
		\[\hat{h}(P)\geq\frac{1}{1.49*10^{16}d^3\left(\log(2d)\right)^2}\log|N_{K/\Q}\Delta_{E/K}| \]
		
	\end{enumerate}
	
	With the same notations, we have also:
	\[\hat{h}(P)\geq \tilde{C}_{24d}' h_F(E/L),\]
	
	and the formula \ref{eq:falt0} of section 2.2 shows that:
	
	\[h_F(E/K)=h_F(E/L)+\frac{1}{12[K:\Q]}\log|N_{L/\Q}\Delta^{\textrm{inst}}_{E/K}|.\]
	
	Then, as well, there are two cases:
	
	\begin{enumerate}
		\item either
		\[\frac{1}{12[K:\Q]}\log|N_{L/\Q}\Delta^{\textrm{inst}}_{E/K}|<h_F(E/L),\]
		and therefore
		\[h_F(E/L)>\frac{1}{2}h_F(E/K),\] so that
		\[\hat{h}(P)\geq \frac{1}{2}\tilde{C}_{24d}' h_F(E/K).\]
		\item or
		\[\frac{1}{12[K:\Q]}\log|N_{L/\Q}\Delta^{\textrm{inst}}_{E/K}|>h_F(E/L),\]
		and in this case
		\[h(E/K)\leq\frac{1}{6[K:\Q]}\log|N_{L/\Q}\Delta^{\textrm{inst}}_{E/K}|\leq\frac{1}{6[K:\Q]}\log|N_{K/\Q}\Delta_{E/K}|,\]
		we then can deduce the inequality involving the Faltings height from the inequality with the discriminant.
	\end{enumerate}
	
\end{proof}
\input{LangProof-BW-Appendix.tex}
\bibliography{LangProof-BW-biblio.bib}
\bibliographystyle{plain}

\end{document}

%% file: LangProof-BW-macromaths.tex
\newcommand{\N}{\mathbb{N}}
\newcommand{\Z}{\mathbb{Z}}
\newcommand{\Q}{\mathbb{Q}}
\newcommand{\R}{\mathbb{R}}
\newcommand{\C}{\mathbb{C}}
\newcommand{\p}{\mathfrak{p}}
\newcommand{\spec}{\textrm{Spec}}

\newcommand{\V}{\mathcal{V}}

\newcommand{\Ocal}{\mathcal{O}}

\newcommand{\Line}{\mathcal{L}}

\newcommand{\inv}{\textrm{inv}}

\newcommand{\mb}{{}}
\newcommand{\Pcal}{\mathcal{P}}
\newcommand{\cur}{\mathcal{P}'}

\newcommand{\Nr}{\mathcal{N}}
\newcommand{\m}{k_{\textrm{max}}}
\newcommand{\mor}{\textrm{mb}}
\newcommand{\ab}{\mathcal{A}} 
\newcommand{\norm}{\log\left|N_{K/\Q}\Delta_{E/K}\right|}
\newcommand{\ceil}{\left\lceil\frac{23}{\epsilon}\right\rceil}
\newcommand{\dg}{d}
\newcommand{\h}{\hat{h}}
\newcommand{\la}{\hat{\lambda}}
\newcommand{\disc}{\Delta_{E/K}}
\newcommand{\lo}{\log^{(1)}}
\newcommand{\ord}{\textrm{ord}}

%% file: LangProof-BW-macros.tex
\newtheorem{proof}{Proof}
\newtheorem{lemma}{Lemma}
\newtheorem{theorem}{Theorem}
\newtheorem{proposition}{Proposition}
\newtheorem{corollary}{Corollary}
\newtheorem{definition}{Definition}
\newtheorem{remark}{Remark}


%% file: LangProof-BW-Appendix.tex
\section{Appendix}

\begin{lemma} Let $N\geq n$ be positive integers and define
\[V_{N,n}=\left\{(m_1,m_2,\ldots,m_N)\in \N^n\left|\right. \forall\: 1\leq i\neq j\leq n,\; 1\leq m_i\neq m_j\leq N \right\},\]
then
\[\max\limits_{(m_1,m_2,\ldots m_n)\in V_{N,n}}\sum\limits_{1\leq i,j\leq n} (m_i-m_j)^2\leq\frac{(n+1)^2(N+1)^2}{2},\]

\end{lemma}
\begin{proof}
Let us define the two following functions:

\[\xymatrix{
F:&\R^n\ar[r]&\R\\
&\bar{x}:=(x_1,\ldots,x_n)\ar@{|->}[r] &\sum\limits_{1\leq i\neq j\leq n}(x_i-x_j)^2},\]

and

\[\xymatrix{
M:&\R^n\ar[r]&\R\\
&\bar{x}:=(x_1,\ldots,x_n)\ar@{|->}[r] &\frac{1}{n}\sum\limits_{1\leq i\leq n}x_i}.\]

We have then,

\[\frac{\partial F}{\partial x_i}(\bar{x})=4n(x_i-M(\bar{x})),\]
moreover we notice that for the following transformations, given $h\geq 0$,

\[\xymatrix{i^{(+h)}:&\bar{x}=(x_1,\ldots,x_n)\ar@{|->}[r] & \bar{x}^{(+i_h)}=(x_1,\ldots,x_{i-1},x_i+h,x_{i+1},\ldots,x_n)\\
i^{(-h)}:&\bar{x}=(x_1,\ldots,x_n)\ar@{|->}[r] & \bar{x}^{(-i_h)}=(x_1,\ldots,x_{i-1},x_i-h,x_{i+1},\ldots,x_n),}\]
we have,
\[\xymatrix{x_i\geq M(\bar{x})\ar@{=>}[r] &x_i+h\geq M(\bar{x}^{(+i_h)})\\
 x_i\leq M(\bar{x})\ar@{=>}[r] & x_i-h\leq M(\bar{x}^{(-i_h)})}.\]
 
 We deduce from this that fixing an integer $k$ such that, the function $F$ is decreasing in each variable $x_i\leq M(\bar{x})$ for $i\leq k$ and increasing for each variable $x_i\geq M(\bar{x})$, for $i\geq k$. Therefore let us fix $k$ such that
 
 \[\xymatrix{x_1,x_2,\ldots x_k\leq M(\bar{x})\\
 x_{k+1},x_{k+1},\ldots, x_n\geq M(\bar{x})},\]
 
 the maximum of $F$ with the $x_i'$ different positive integers satisfying the previous hypothesis will be given for
 
 \[\xymatrix{\forall \: 1\leq i\leq k,\:x_i=i\\
 \forall\: k+1\leq i\leq n,\:x_{i+i}=N-n+i.}\]
 
Denoting, $F_{\max}^k$ the corresponding value of $F$, for $n=2l$ an elementary comparison between the terms of $F_{\max}^l$ and of $F_{\max}^k$ for $k<l$ show that amongst
the $F_{\max}^k$ the maximum is obtained for $k=l$. One can also notice that simply for $n=2l$, $F_{\max}^k=F_{\max}^{n-k}$ and that for $k<l$, $F_{\max}^{k+1}>F_{\max}^k$.

Then, an elementary computation of $F_{\max}^l$ for $n=2l$ obtained by expanding the squares gives,

\begin{align}
 F_{\max}^l=  2l^2(N+1)^2\left(1-2\frac{l+1}{N+1}+\frac{2(l+1)(2l+1)}{3(N+1)^2}\right)\\
=\frac{n²(N+1)^2}{2}\underbrace{\left(1-\frac{n+2}{N+1}+\frac{(n+1)(n+2)}{3(N+1)^2}\right)}_{\leq 1}.
\end{align}

As the maximum for $n=2l$ is less or equal than the maximum for $n=2l+1$ which is less or equal
to the maximum for $n=2l+2$, we obtain finally the inequality given in the proposition just by replacing
$n$ by $n+1$ in the last equation and inequality. 
\end{proof}

\begin{lemma}
\[N_E\leq \left(|N_{K/\Q}\Delta_{E/K}|\right)^{1/e\log(2)}\leq\left(|N_{K/\Q}\Delta_{E/K}|\right)^{0.54}.\]
\end{lemma}
\begin{proof}
We denote here $\delta$ the number of distinct prime (ideal) divisors of $\Delta_{E/K}$.

\[N_E\leq\prod\limits_{v|\Delta_{E/K}}N_v,\]
and
\[|N_{K/\Q}\Delta_{E/K}|\geq\prod\limits_{v|\Delta_{E/K}}2^{N_v}\geq 2^{\sum\limits_{v|\Delta_{EK}}N_v},\]
therefore the arithmetic-geometric mean inequality provides:
\[\log|N_{K/\Q}\Delta_{E/K}|\geq\left(\sum\limits_{v|\Delta_E}N_v\right)\log(2)\geq\delta(N_E)^{1/\delta}\log(2).\]
Moreover, for fixed $N_E$, the function $x\mapsto x(N_E)^{1/x}$ is minimum for $x=\log(N_E)$,
from which we deduce that
\[\log|N_{K/\Q}\Delta_{E/K}|\geq \log(N_E)\underbrace{(N_E)^{1/\log(N_E)}}_{=e}\log(2),\]
therefore
\[N_E\leq |N_{K/\Q}\Delta_{E/K}|^{1/e\log(2)},\]
and
\[\frac{1}{e\log(2)}\approx 0,5307....\]
\end{proof}

\begin{lemma}(Combinatorial Lemma)
Let $(r_v)_{v\in S}$ be a sequence of non-negative real numbers indexed by a set $S$. Let $S_i$, $1\leq i\leq n$ be subsets of $S$ and let $l$ be a positive real number such that
\[\forall\/1\leq i\leq n,\;\sum\limits_{v\in S_i}r_v\geq\frac{1}{l}\sum_{v\in S}r_v.\]

Under the assumption that there exists an integer $Z$ such that $n\geq l(Z+1)$, there exists distinct integers $i_0,i_1,\ldots,i_Z$ such that:
\[\forall\:1\leq j\leq Z,\:\sum\limits_{S_{i_0}\cap S_{i_j}}r_v\geq\frac{2}{l^2(l+1)}\sum\limits_{v\in S}r_v.\]
\end{lemma}

\begin{proof}
Let us define, 
\[\mu_0(S)=\sum\limits_{v\in S}r_v,\] and
for any subset $T$ of $S$,
\[\mu(T)=\frac{1}{\mu_0(S)}\sum\limits_{v\in T}r_v.\]

This defines a probability measure $\mu$ on $S$ which satisfies the usual relations,

\[\forall\:A,B\subset S; \mu(A\cup B)+\mu(A\cap B)=\mu(A)+\mu(B),\]
\[\forall\:A_1,A_2,\ldots,A_k\subset S,\:\mu\left(\bigcup\limits_{i=1}^k\right)\geq\sum\limits_{i=1}^k\mu(A_i)-\sum\limits_{1\leq i< j\leq k}\mu(A_i\cap A_j).\]

We define
\[e=\frac{2}{l^2(l+1)},\]

then, let us choose a set $S_0$:
\begin{itemize}
\item either there exists $Z$ subset $S_i$ with distincts $i\neq 0$ such that
$\mu(S_0\cap S_i)\geq e$ and then the lemma is true,
\item or for at least $(n-Z)$ subset $S_i$ of distinct $i\neq 0$, $\mu(S_0\cap S_i)<\epsilon$.
\end{itemize}.

In the second case, we choose a set that we denote $S_1$ amongst the $S_i$, and then
\begin{itemize}
\item either there exists $Z$ subsets $S_i$ with distinct $i\neq 1$, such that $\mu(S_1\cap\mu S_i)\geq e$ and the lemma is true,
\item or there exists at least $n-2Z$ subsets $S_i$ with disctinct $i\neq 0$ and $i\neq 1$, such that $\mu(S_0\cap S_i)<e$ and $\mu(S_1\cap S_i)<e$.
\end{itemize}

By iterating this process $k$ times we obtain the following:
\begin{itemize}
\item either there are $Z+1$ subset $S_i$ that solve the proposition,
\item or there exists $k+1$ subsets that we denote $S_0,S_1,\ldots,S_k$ such that:
\[\forall\:1\leq i<j\leq k, \mu(S_i\cap S_j)<e.\]
\end{itemize}

For those last sets we deduce that,
\[\mu\left(\bigcup\limits_{i=0}^k S_i\right)\geq\frac{k+1}{l}-\frac{k(k+1)}{2}e,\]
and therefore if we define
\[S':=S\setminus\left(\bigcup\limits_{i=0}^k S_i\right),\]
and for all $0\leq j\leq k$,
\[S_j'=S'\cap S_j,\]
then
\[\mu(S')\leq 1-\frac{k+1}{l}+{k(k+1)}{2}e,\]
and
\[\mu(S_j')\geq\frac{1}{l}-(k+1)e.\]

And so for ny of those i's we have at least $n-(k+1)Z$ sets $S_j$ such that
\[\mu(S_i\cap S_j)\geq\mu(S_i'\cap S_j')\geq\mu(S_i')+\mu(S_j')-\mu(S')\geq\frac{2}{l}-2(k+1)e-1+\frac{k+1}{l}-\frac{k(k+1)}{2}e.\]

If we choose $k+1=l-1$ we obtain at least $n-(l-1)Z\geq Z$ subsets $S_0,S_1,\ldots,S_Z$ with distinct $i's$ such that
\[\forall\:1\leq j\leq Z,\:\mu(S_0\cap S_j)\geq e.\]

Which concludes.
\end{proof}  